\documentclass[11pt,twoside]{article}

\usepackage{amssymb}
\usepackage{amsmath}
\usepackage{mathrsfs}
\usepackage{amsthm}
\usepackage{amsfonts}
\usepackage{txfonts}
\usepackage{enumerate}

\usepackage{hyperref}

\usepackage{latexsym,amssymb}
\usepackage{color}
\usepackage{indentfirst}
\usepackage{graphicx}
\usepackage{pgf,tikz}

\let\mathbb=\varmathbb
\DeclareSymbolFont{letters}{OML}{ztmcm}{m}{it}
\hypersetup{colorlinks   = true,
        citecolor    = blue,
        linkcolor    = blue,
        urlcolor     = blue}

\usepackage{anysize}

\allowdisplaybreaks

\newtheorem{theorem}{Theorem}[section]
\newtheorem{lemma}[theorem]{Lemma}
\newtheorem{corollary}[theorem]{Corollary}
\newtheorem{proposition}[theorem]{Proposition}

\theoremstyle{definition}
\newtheorem{remark}[theorem]{Remark}

\newtheorem{definition}[theorem]{Definition}

\numberwithin{equation}{section}

\numberwithin{equation}{section}
\allowdisplaybreaks

\begin{document}

\arraycolsep=1pt

\title{\bf\Large Equivalent Characterizations and Applications
of Fractional Sobolev Spaces with Partially Vanishing Traces
on $(\epsilon,\delta,D)$-Domains Supporting $D$-Adapted
Fractional Hardy Inequalities
\footnotetext{\hspace{-0.35cm}
2020 {\it Mathematics Subject Classification}. Primary: 46E35;
Secondary: 42B35, 26A33, 26D10, 35J25.
\endgraf {\it Key words and phrases}. Fractional Sobolev space,
mixed boundary condition, Hardy inequality,
extension operator, elliptic operator, fractional power.
\endgraf This project is supported by the Zhejiang Provincial
Natural Science Foundation of China (Grant No. LR22A010006)
and the National Natural Science Foundation of China
(Grant Nos. 12431006 and 12371093).}}
\author{Jun Cao, Dachun Yang\footnote{Corresponding author, E-mail:
\texttt{dcyang@bnu.edu.cn}/{\color{red}{\today}}/Final version.} \ and Qishun Zhang}
\date{}
\maketitle

\vspace{-0.7cm}

\begin{center}
\begin{minipage}{13.5cm}
{\small {\bf Abstract} \quad
Let $\Omega\subset\mathbb{R}^n$ be an $(\epsilon,\delta,D)$-domain,
with $\epsilon\in(0,1]$, $\delta\in(0,\infty]$, and
$D\subset \partial \Omega$ being a closed part of $\partial \Omega$,
which is  a general open connected set when $D=\partial \Omega$ and
an $(\epsilon,\delta)$-domain when $D=\emptyset$.
Let $s\in(0,1)$ and $p\in[1,\infty)$. If ${W}^{s,p}(\Omega)$,
${\mathcal W}^{s,p}(\Omega)$, and $\mathring{W}_D^{s,p}(\Omega)$ are
the fractional Sobolev spaces on $\Omega$ that are defined
respectively via the restriction of $W^{s,p}(\mathbb{R}^n)$ to $\Omega$,
the intrinsic Gagliardo norm, and the completion of all $C^\infty(\Omega)$
functions with compact support away from $D$, in this article we
prove their equivalences [that is,
${W}^{s,p}(\Omega)={\mathcal{W}}^{s,p}(\Omega)
=\mathring{W}_D^{s,p}(\Omega)$]
if $\Omega$ supports a $D$-adapted fractional Hardy inequality
and, moreover, when $sp\ne 1$ such a fractional Hardy inequality is shown to be
necessary to guarantee these equivalences under some mild geometric
conditions on $\Omega$. 
Using the aforementioned equivalences,
we show that the real interpolation space $(L^p(\Omega), \mathring{W}_D^{1,p}(\Omega))_{s,p}$
equals to some weighted fractional order  Sobolev space
$\mathcal{W}^{s,p}_{d_D^s}(\Omega)$ when $p\in (1,\infty)$.
Applying this to the elliptic operator $\mathcal{L}_D$
in $\Omega$ with mixed boundary condition,
we characterize both the domain of its fractional power and
the parabolic maximal regularity of its Cauchy initial
problem by means of $\mathcal{W}^{s,p}_{d_D^s}(\Omega)$.}
\end{minipage}
\end{center}

\vspace{0.2cm}

\tableofcontents

\section{Introduction}
We divide this introduction into two subsections. In Subsection
\ref{s1.1}, we simply present some developments about
fractional Sobolev spaces on domains, while in Subsection
\ref{s1.2} we state the main results of this article.

\subsection{Backgrounds}\label{s1.1}
Let $s\in (0,1)$ and $p\in [1,\infty)$.
The fractional Sobolev space
$W^{s,p}(\mathbb{R}^n)$ on the Euclidean space $\mathbb{R}^n$ is
defined as the collection of all functions $f\in L^p(\mathbb{R}^n)$ such that
\begin{align}\label{eqn-normfs}
\left\|f\right\|_{W^{s,p}(\mathbb{R}^n)}:=\|f\|_{L^p(\mathbb{R}^n)}
+\left[\int_{\mathbb{R}^n}\int_{\mathbb{R}^n}\dfrac{|f(x)-f(y)|^p}
{|x-y|^{n+sp}}\,dxdy\right]^{\frac 1p}<\infty.
\end{align}
It is known that the  space $W^{s,p}(\mathbb{R}^n)$ was introduced
almost simultaneously
by Aronszajn \cite{Ar55}, Gagliardo \cite{Ga58}, and Slobodecki\u{i} \cite{Sl58},
to fill the smoothness gap of function spaces
between the Lebesgue space $L^p(\mathbb{R}^n)$ and the first-order
Sobolev space $W^{1,p}(\mathbb{R}^n)$ (see, e.g., \cite{RuSi96,Tri95}
for some historical reviews).
Since then, $W^{s,p}(\mathbb{R}^n)$ has been a classical subject in
various areas of mathematics, which
offers a more refined smoothness description than the integer order
Sobolev spaces. A substantial body of literature has approached this
research topic from diverse perspectives (see, e.g., \cite{Ad03,AlYaYu24,DaLiYaYu18,DaLiYaYuZh23,EE23,HaTr08,Lua,DPV12,
SHSHJ24,Sa18}).

To extend the theory of  fractional Sobolev spaces from the whole
space $\mathbb{R}^n$ to a general domain $\Omega\subset\mathbb{R}^n$,
a usual way is to consider the following trace space $W^{s,p}(\Omega)$
(see, e.g., \cite{Tri83}) of  $W^{s,p}(\mathbb{R}^n)$, defined  by setting
\begin{align}\label{eqn-fstrace}
 W^{s,p}(\Omega):=\left\{f\in L^p(\Omega):\ \text{there is} \ g\in W^{s,p}(\mathbb{R}^n)\ \text{such that } g|_\Omega=f\right\},
\end{align}
equipped with the standard quotient norm
\begin{align}\label{eqn-intrinn}
\|f\|_{W^{s,p}(\Omega)}:=\inf \left\{\left\|g\right\|_{W^{s,p}(\mathbb{R}^n)}:\ g\in W^{s,p}(\mathbb{R}^n)\ \text{such that } g|_\Omega=f\right\}.
\end{align}
For the fractional Sobolev space $W^{s,p}(\Omega)$ defined in \eqref{eqn-fstrace},
the following questions naturally appear:

\begin{itemize}
\item
{\bf Intrinsic characterization}: Can the norm
$\|\cdot\|_{W^{s,p}(\Omega)}$ in \eqref{eqn-intrinn}  be characterized by
other kinds of intrinsic norms which involve only points in $\Omega$?
A typical candidate intrinsic norm is the following {\it Gagliardo norm} $\|\cdot\|_{\mathcal{W}^{s,p}(\Omega)}$ defined by setting
\begin{align}\label{eqn-normfsdm}
\left\|f\right\|_{\mathcal{W}^{s,p}(\Omega)}:=\,&\|f\|_{L^p(\Omega)}
+\left\|f\right\|_{\dot{\mathcal{W}}^{s,p}(\Omega)}\notag\\
:=\,&\|f\|_{L^p(\Omega)}+\left[\int_\Omega\int_\Omega
\dfrac{|f(x)-f(y)|^p}{|x-y|^{n+sp}}\,dxdy\right]^{\frac 1p}.
\end{align}
Recall that \eqref{eqn-normfsdm} is the norm used by Lions--Magenes \cite{LiMa60} and Ne\u{c}as \cite{Ne67}
to introduce the fractional Sobolev space on domain
\begin{align}\label{eqn0iss}
\mathcal{W}^{s,p}(\Omega):=\left\{f\in L^p(\Omega):\ \left\|f\right\|_{\mathcal{W}^{s,p}(\Omega)}<\infty\right\}.
\end{align}
Also, there are other kinds of intrinsic norms that are defined by
the fractional maximal function \cite{Ch84},
the modulus of smoothness \cite{DeSh93}, etc.

\item {\bf Geometric characterization}:
How to characterize the behaviors of fractional Sobolev
functions in various geometric objects derived from $\Omega$?
A typical problem in this aspect is the
linear extension problem \cite{Cal61,Ste70}, which asks for a linear operator
$\mathcal{E}$ bounded from $W^{s,p}(\Omega)$ or
$\mathcal{W}^{s,p}(\Omega)$ to
$W^{s,p}(\mathbb{R}^n)$ such that,  for
any $f\in W^{s,p}(\Omega)$, $\mathcal{E}f|_\Omega=f.$
It is well known that such an extension bridges the theory of $W^{s,p}(\Omega)$ with the
corresponding one of
$W^{s,p}(\mathbb{R}^n)$ (\cite{Tri08}). A closely related
inverse problem is the boundary trace problem \cite{BrBr12},
which asks
for characterizations of the limiting behavior of functions in  $W^{s,p}(\Omega)$
as they approach the boundary $\partial\Omega$.

\item
{\bf Approximation characterization}: Can the space
$W^{s,p}(\Omega)$ be characterized by the approximation
of  some particular smooth function classes
(see, e.g., \cite{Ad03,FiSe15}),  such as $C_{\rm c}^\infty(\Omega)$ and $C^\infty(\Omega)$, etc.
Notice that, if  $C_{\rm c}^\infty(\Omega)$ is dense in
$W^{s,p}(\Omega)$, then we have the dense inclusions
$C_{\rm c}^\infty(\Omega)\subset W^{s,p}(\Omega)\subset \left(C_{\rm c}^\infty(\Omega)\right)'$,
which implies that $W^{s,p}(\Omega)$ is a regular space so that
a theory of distributions in the sense of H\"ormander \cite{Ho03} can been applied.
\end{itemize}

It has been observed that the research on the aforementioned issues is significantly contingent upon the
geometric characteristics of the underlying domain $\Omega$ under consideration.
We refer to \cite{Tri95,Tri83} for the case where $\Omega=\mathbb{R}_+^n:=\mathbb{R}^{n-1}\times
(0,\infty)$ is the upper half space of $\mathbb{R}^n$,
to \cite{Tri83,Tri92} for the case where $\Omega$ is a bounded smooth domain,
to \cite{DPV12,Pra15,Ryc00} for the case where $\Omega$ is a bounded Lipschitz domain, and
to \cite{AlYaYu22,WaLiLi25,Pra17,ShXi17,Tri08,Yao23} for the case where $\Omega$ is a $(\epsilon,\delta)$-domain.
It is worth emphasizing that the literature we have cited here merely constitutes a small fraction of the extensive
body of works pertinent to the research in this field. Moreover, the aforementioned questions are closely related  each other. For instance, to
obtain the approximation characterization, one usually needs a suitable intrinsic norm characterization.
For the latter characterization, some geometric characterizations such as bounded linear extensions play a crucial role.

In this article, we address the aforementioned problems in the setting where the functions in  $W^{s,p}(\Omega)$ have
partially vanishing traces, that is, the boundary $\partial\Omega$ is decomposed into two parts
$\partial\Omega=\Gamma\cup D$
and functions in $W^{s,p}(\Omega)$ vanish at $D$ (see Figure \ref{p1}).
\begin{figure}[h]
\centering
\includegraphics[width=6cm]{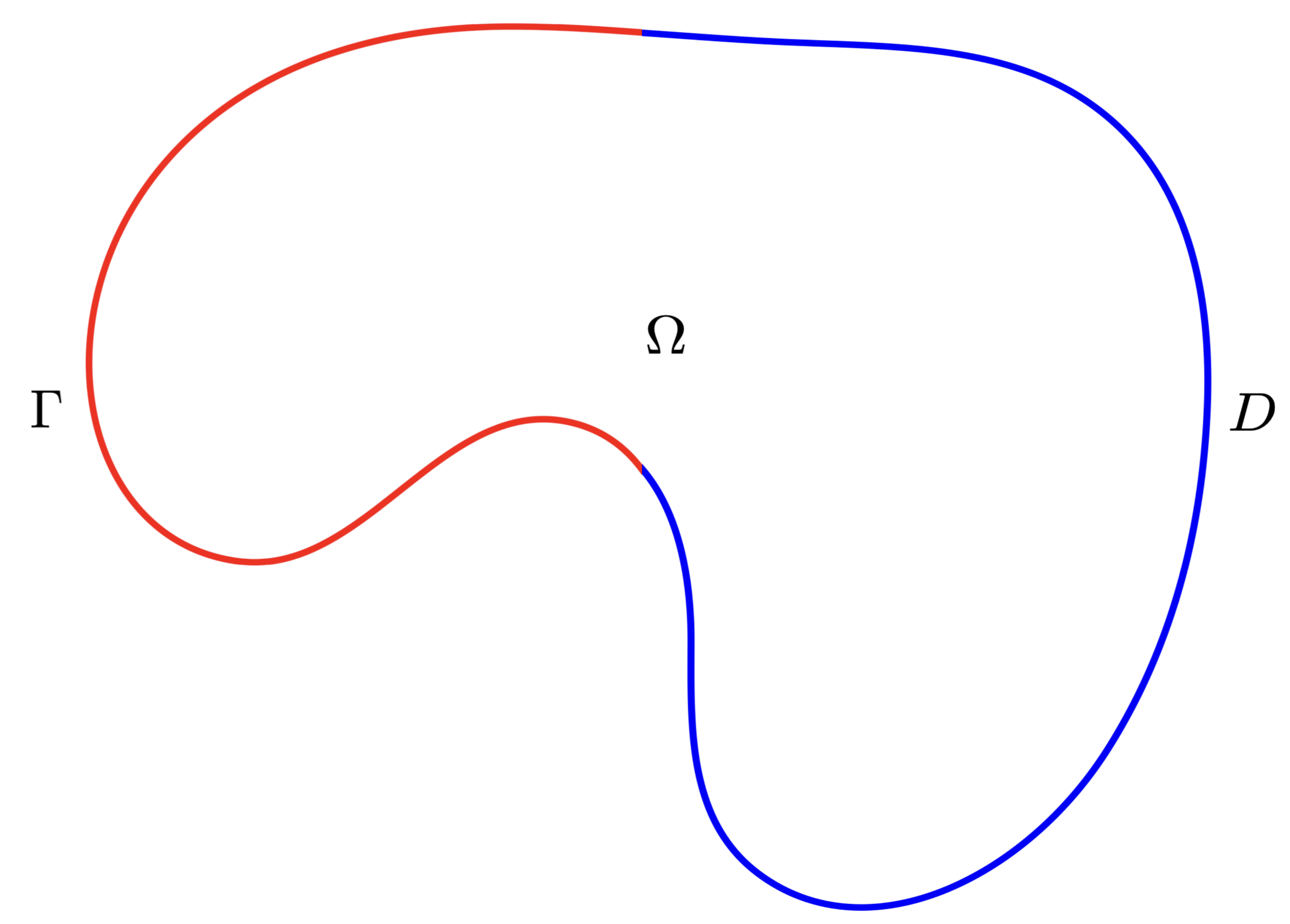}
\caption{Domain with mixed boundary\label{p1}}
\end{figure}
Such setting aries naturally in many scenarios such as elliptic
problems with mixed boundary conditions (see, e.g., \cite{Bre14,TOB13,Bui21,BuDu21,Mit07,Sn66}),
where the solutions are required to
satisfy different types of boundary conditions on different
non overlapping parts of the boundary.

To be more precise, let $\Omega\subset\mathbb{R}^n$ be an $(\epsilon,\delta,D)$-domain,
with $\epsilon\in(0,1]$, $\delta\in(0,\infty]$, and
$D\subset \partial \Omega$ being a closed part of $\partial \Omega$,
which is  a general open connected set when $D=\partial \Omega$ and
an $(\epsilon,\delta)$-domain when $D=\emptyset$.
Let $s\in(0,1)$ and $p\in[1,\infty)$. If ${W}^{s,p}(\Omega)$,
${\mathcal W}^{s,p}(\Omega)$, and $\mathring{W}_D^{s,p}(\Omega)$ are
the fractional Sobolev spaces on $\Omega$ that are defined
respectively via the restriction of $W^{s,p}(\mathbb{R}^n)$ to $\Omega$,
the intrinsic Gagliardo norm, and the completion of all $C^\infty(\Omega)$
functions with compact support away from $D$, in this article we
prove their equivalences [that is,
${W}^{s,p}(\Omega)={\mathcal{W}}^{s,p}(\Omega)
=\mathring{W}_D^{s,p}(\Omega)$]
if $\Omega$ supports a $D$-adapted fractional Hardy inequality
and, moreover, when $sp\ne 1$ such a fractional Hardy inequality is shown to be
necessary to guarantee these equivalences under some mild geometric
conditions on $\Omega$.
Using the aforementioned equivalences,
we show that the real interpolation space $(L^p(\Omega), 
\mathring{W}_D^{1,p}(\Omega))_{s,p}$
equals to some weighted fractional order Sobolev space
$\mathcal{W}^{s,p}_{d_D^s}(\Omega)$ when $p\in (1,\infty)$.
Applying this to the elliptic operator $\mathcal{L}_D$
in $\Omega$ with mixed boundary condition,
we characterize both the domain of its fractional power and
the parabolic maximal regularity of its Cauchy initial
problem by means of $\mathcal{W}^{s,p}_{d_D^s}(\Omega)$.

Recall in \cite{Bec19-In,Bre14} that a natural way,
to incorporate the vanishing trace condition at $D$ into the fractional Sobolev space
$W^{s,p}(\Omega)$ under consideration, is to
define the following trace space
\begin{align}\label{eqn-fstraceD}
{W}_D^{s,p}(\Omega):=\left\{f\in L^p(\Omega):\ \text{there is}\  g\in W^{s,p}
(\mathbb{R}^n)\ \text{satisfying}\ \mathcal{R}_Dg\equiv 0\ \text{such that }\ g|_\Omega=f \right\},
\end{align}
equipped with the quotient norm
\begin{align*}
\|f\|_{ W_D^{s,p}(\Omega)}:=\inf \left\{\left\|g\right\|_{W^{s,p}(\mathbb{R}^n)}:\ g\in
W^{s,p}(\mathbb{R}^n)\ \text{satisfying}\ \mathcal{R}_Dg\equiv 0 \
\text{such that }\ g|_\Omega=f\right\},
\end{align*}
where, for any $x\in D$,
\begin{align}\label{eqn-Rd}
\mathcal{R}_Dg(x):=\lim_{r\in (0,\infty), r\to 0}\dfrac{1}{|B(x,r)|}\int_{B(x,r)}g(y)\,dy
\end{align}
denotes the trace of $g$ at $x$. However, such definition has
a restriction that it makes sense only for a limited range of $s$.
In particular, it was proven in \cite[Theorem VI.1]{Jon1984}
that the trace operator $\mathcal{R}_D$ in \eqref{eqn-Rd} is
well-defined only when $p\in(1,\infty)$, $s\in(1/p,1)$, and $D$ is
an $(n-1)$-set [see Definition \ref{def-d-set} for the precise
definition of $(n-1)$-sets].

In this article, we utilize a $D$-adapted Hardy inequality
to characterize the boundary vanishing trace property
of fractional Sobolev functions. To be precise, we assume that, for any $f\in\mathcal{W}^{s,p}(\Omega)$ as in
\eqref{eqn-normfsdm}, it holds that
\begin{align}\label{eqn-HI}
\int_\Omega\frac{|f(x)|^p}{[{{\mathop\mathrm{dist\,}}}(x,D)]^{sp}}\,dx \lesssim\|f\|^p_{\mathcal{W}^{s,p}(\Omega)},
\end{align}
where the implicit positive constant is independent of $f$.
The validity of \eqref{eqn-HI} inherently requires that functions in $\mathcal{W}^{s,p}(\Omega)$ tend to zero as
they approach $D$. In particular, if $D=\emptyset$,
then \eqref{eqn-HI} always holds; if
$D=\partial\Omega$, \eqref{eqn-HI} is a variation of the classical
Hardy inequality in domains, where the latter  has been widely studied (see, e.g., \cite{Bec21,Dy04,DyVa14,EE23,lll12,LiNi05}).
Recall that the idea of exploiting the Hardy inequality to characterize the partially vanishing trace
property of (fractional) Sobolev spaces has already been appeared in \cite{Bec21,Ege14-Hd,EgTo17}. Compared
with \eqref{eqn-fstraceD}, the $D$-adapted Hardy inequality
\eqref{eqn-HI} has the advantage that it makes sense
for any $s\in(0,1)$ and $p\in[1,\infty)$ and only requires that $D$ is a
subset of $\mathbb{R}^n$.

Throughout this article, we work within the framework of the 
so-called $(\epsilon, \delta, D)$-domains introduced in
\cite{Bec19-Ext} (see Figure \ref{p2}). Here, and hereafter, 
we always assume that $\Omega$ is a domain, i.e., a connected open set in 
$\mathbb{R}^n$.

\begin{center}
\begin{figure}[h]
\centering
\includegraphics[width=6cm]{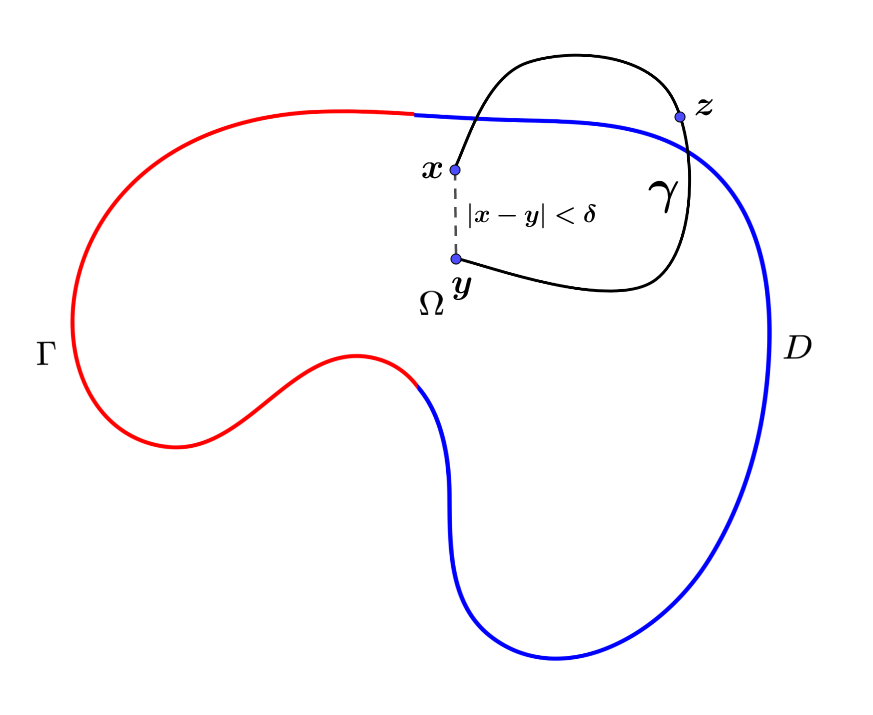}
\caption{The $(\epsilon,\delta,D)$-domain}\label{p2}
\end{figure}
\end{center}

\begin{definition}[\cite{Bec19-Ext}]\label{ass-1}
Let $\epsilon\in(0,1]$, $\delta\in(0,\infty]$, and $\Omega\subset\mathbb{R}^n$ be a domain with $D\subset\partial\Omega$
being closed. We call $\Omega$ an {\it $(\epsilon,\delta,D)$-domain} if there exists a positive constant $K$ such that, for
every pair of points $x,y\in\Omega$ with $|x-y|<\delta$, there exists a rectifiable curve $\gamma$, connecting $x$ and $y$ and
taking values in $\Xi:=\mathbb{R}^n\setminus\overline{\Gamma}$ with $\Gamma:=\partial\Omega\setminus D$, that satisfies
\begin{itemize}
\item[{\rm (i)}]$\ell(\gamma)\leq \frac{|x-y|}{\epsilon}$, where $\ell(\gamma)$ denotes the length of $\gamma$,

\item[{\rm (ii)}]${{\mathop\mathrm{dist\,}}}(z,\Gamma)\geq\epsilon\dfrac{|x-z|\, |y-z|}{|x-y|}$ for any $z\in\gamma$, and

\item[{\rm (iii)}]$k_\Xi(z,\Omega)\leq K$ for any $z\in\gamma$, where $k_\Xi(z,\Omega)$ denotes the {\it quasihyperbolic distance}
between $z$ and $\Omega$ in $\Xi$.
\end{itemize}
\end{definition}

\begin{remark}
\begin{itemize}
\item[(i)] Recall in \cite{Geh76} that, for any two points $x,y\in\Xi$,
their {\it quasihyperbolic distance} $k_\Xi(x,y)$ is defined by setting
\begin{align*}
k_\Xi(x,y):=\inf_\gamma\int_\gamma\frac{1}{{{\mathop\mathrm{\,dist\,}}}(z,\partial\Xi)}\,dz,
\end{align*}
where the infimum is taken over all rectifiable curves $\gamma\subset\Xi$ joining $x$ and $y$.
For any $x\in\Xi$ and $\Omega\subset\Xi$, define
$k_\Xi(x,\Omega):=\inf_{y\in\Omega}k_\Xi(x,y).$
\item[(ii)] By Definition \ref{ass-1}, we easily conclude that every domain $\Omega$ is an
$(\epsilon,\delta,\partial\Omega)$-domain. On the other hand, the
$(\epsilon,\delta,\emptyset)$-domain reduces to the so-called $(\epsilon,\delta)$-domain introduced in
\cite{JeKe82,Jones81}. There are more interesting examples of $(\epsilon,\delta,D)$-domains
when $\emptyset\neq D\subsetneqq\partial\Omega$. For instance, it was proven in
\cite[Propositions 3.2 and 3.4]{Bec19-Ext}
that, if a domain $\Omega$ is locally $(\epsilon,\delta)$ near $\Gamma$
as in \cite[Definition 3.4]{Bre14} or locally Lipschitz along $\Gamma$ in the sense of
\cite[Assumption 3.3]{Bec19-Ext}, then $\Omega$ is an $(\epsilon,\delta,D)$-domain.
\end{itemize}
\end{remark}

\subsection{Main results}\label{s1.2}

We are now in a position to state the first main result of this article, which establishes the characterizations of $W^{s,p}(\Omega)$
under partially vanishing trace condition in the sense of \eqref{eqn-HI}, respectively, in terms of the intrinsic
Gagliardo norm $\|\cdot\|_{\mathcal{W}^{s,p}(\Omega)}$ in \eqref{eqn-normfsdm} and the approximation of functions in
\begin{align*}
C_D^\infty(\Omega):=\left\{f|_\Omega:f\in C_{\rm c}^\infty(\mathbb{R}^n),\ {{\mathop\mathrm{dist\,}}}({\rm supp}\, f,D)>0\right\}.
\end{align*}

\begin{theorem}\label{thm-equiv}
Let $s\in(0,1)$, $p\in[1,\infty)$, and $\Omega\subset \mathbb{R}^n$ be an $(\epsilon,\delta,D)$-domain as in
Definition \ref{ass-1} with $\epsilon\in(0,1]$, $\delta\in(0,\infty]$,
and $D\subset\partial\Omega$ being closed. If $\Omega$ supports the $D$-adapted Hardy inequality \eqref{eqn-HI}, then
\begin{align}\label{h-con1}
W^{s,p}(\Omega)=\mathcal{W}^{s,p}(\Omega)=\mathring{W}_D^{s,p}(\Omega),
\end{align}
where
\begin{align}\label{h-con121}
\mathring{W}_D^{s,p}(\Omega):=\overline{C_D^\infty
(\Omega)}^{\|\cdot\|_{\mathcal W^{s,p}(\Omega)}}
\end{align}
with $\|\cdot\|_{\mathcal W^{s,p}(\Omega)}$ as in \eqref{eqn-normfsdm}.
\end{theorem}

\begin{remark}\label{rem1.4}
\begin{itemize}
\item[(i)] Theorem \ref{thm-equiv} is an intermediate result between
two endpoint cases:
\begin{itemize}
\item[(a)] if $D=\emptyset$,   then $\Omega$ is an $(\epsilon,\delta)$-domain.
In this case,  the equality
$W^{s,p}(\Omega)=\mathcal{W}^{s,p}(\Omega)$ follows essentially
from constructing a linear reflection extension bounded from
$\mathcal{W}^{s,p}(\Omega)$ to $W^{s,p}(\mathbb{R}^n)$ (see, for instance,
\cite[Theorem 1.1]{ZY2015}), together with an argument similar to that used in the proof of
Proposition \ref{lemprop-equiv}. On the other hand,
$ W^{s,p}(\Omega)=\mathring{W}_D^{s,p}(\Omega)$ holds automatically
based on their definitions and the fact that $C_{\rm c}^\infty(\mathbb{R}^n)$
is dense in $W^{s,p}(\mathbb{R}^n)$ (see, for instance,
\cite[Theorem 2.4]{DPV12});

\item[(b)] if $D=\partial\Omega$, then $\Omega$ is a general connected
domain. In this case, the equality
$\mathcal{W}^{s,p}(\Omega)=\mathring{W}_D^{s,p}(\Omega)$ can be deduced
from \cite[Theorem 17]{DyKi22}: for any domain $\Omega$ supporting
the $\partial\Omega$-adapted Hardy inequality as in \eqref{eqn-HI},
it holds that
\begin{align}\label{cdw}
\mathring{W}_{\partial\Omega}^{s,p}(\Omega)=
\left\{f\in \mathcal{W}^{s,p}(\Omega):\ \int_\Omega\dfrac{|f(x)|^p}
{[{{\mathop\mathrm{dist\,}}}(x,\partial\Omega)]^{sp}}\,dx
<\infty\right\}.
\end{align}
The equality $\mathring{W}_{\partial\Omega}^{s,p}(\Omega)=W^{s,p}(\Omega)$
in this case seems new, with the proof based on the boundedness of
the zero extension from
$\mathcal{W}^{s,p}(\Omega)$ to $W^{s,p}(\mathbb{R}^n)$ for any domain
$\Omega\subset\mathbb{R}^n$ [see \eqref{ext-x1} for the definition of the
extension and Theorem \ref{thm1} for its boundedness].

\item [(c)] Assume that $\emptyset\neq D\subsetneqq\partial\Omega$, $sp>1$,
$D$ is a uniformly $(n-1)$-set, and $\Omega$ satisfies 
the \emph{interior thickness condition in $\Gamma:=\partial \Omega\setminus D$}
(that is, for any $x\in \Gamma$ and $r\in (0,1]$, it holds that
\begin{align}\label{eqn-isc}
|B(x,r)\cap \Omega|  \gtrsim |B(x,r)|,
\end{align}
where the implicit positive constant is independent of $x$ and $r$).
In this case, the following identity was proved in \cite[Theorem 4.0.2]{Bec22} that
\begin{align*}
{W}_{D}^{s,p}(\Omega)= \left\{f\in {W}^{s,p}(\Omega):\ \int_\Omega\dfrac{|f(x)|^p}
{[{{\mathop\mathrm{dist\,}}}(x,D)]^{sp}}\,dx
<\infty\right\},
\end{align*}
where ${W}_{D}^{s,p}(\Omega)$ is defined as in \eqref{eqn-fstraceD}.
This, together with the fact
$W_D^{s,p}(\Omega)=\mathring W_D^{s,p}(\Omega)$ when $sp>1$
(see \cite[Lemma 3.3]{Bec19-In}), implies
the equality $\mathring W_D^{s,p}(\Omega)={W}^{s,p}(\Omega)$
if the following variant of Hardy inequality \eqref{eqn-HI}
holds that, for any $f\in {W}^{s,p}(\Omega)$,
\begin{align*}
\int_\Omega\frac{|f(x)|^p}{[{{\mathop\mathrm{dist\,}}}(x,D)]^{sp}}\,dx
\lesssim\|f\|^p_{{W}^{s,p}(\Omega)},
\end{align*}
where the implicit positive constant is independent of $f$. Notice that
every $(\epsilon,\delta,D)$-domain supporting the
$D$-adapted Hardy inequality \eqref{eqn-HI} is indeed an $n$-set (see the argument
after Theorem \ref{thm1}).
\end{itemize}
\item[(ii)] Theorem \ref{thm-equiv} can be extended to the integer case $s=1$.
Indeed,
for any $p\in (1,\infty)$, let $W^{1,p}(\Omega)$,
$\mathcal{W}^{1,p}(\Omega)$, and $\mathring{W}_D^{1,p}(\Omega)$
be the Sobolev spaces defined, respectively, similar to \eqref{eqn-fstrace}, \eqref{eqn-normfsdm},  and \eqref{h-con121},
with the intrinsic norm in \eqref{eqn-normfsdm} replaced by
\begin{align*}
\|f\|_{\mathcal{W}^{1,p}(\Omega)}:=\|f\|_{L^p(\Omega)}+\|\nabla f\|_{L^p(\Omega)}.
\end{align*}
Also, the $D$-adapted Hardy inequality \eqref{eqn-HI} is replaced by the following
integer version that, for any $f\in \mathcal{W}^{1,p}(\Omega)$, it holds that
 \begin{align}\label{eqn-IHI}
\int_\Omega\frac{|f(x)|^p}{[{{\mathop\mathrm{dist\,}}}(x,D)]^{p}}\,dx \lesssim\|f\|^p_{{\mathcal W}^{1,p}(\Omega)},
\end{align}
where the implicit positive constant is independent of $f$.
In this case, by Remark \ref{rem-notepf}, we have $\mathring W_D^{1,p}(\Omega)=\mathcal{W}^{1,p}(\Omega)$.
Moreover,  it was proven in  \cite[Theorem 1.2]{Bec19-Ext}  that $\Omega$ admits an bounded linear extension from
$\mathcal{W}^{1,p}(\Omega)$ to  $W^{1,p}(\mathbb{R}^n)$. This, combined with an argument similar
to that used in the proof of  Proposition \ref{lemprop-equiv}, further implies the identity
 $W^{1,p}(\Omega)=\mathcal{W}^{1,p}(\Omega)$. Thus, we have
\begin{align}\label{eqn-W1p}
W^{1,p}(\Omega)=\mathcal{W}^{1,p}(\Omega)=\mathring{W}_D^{1,p}(\Omega).
\end{align}
We refer to \cite[Section 7]{Bec19-Ext} for some related density results 
in the integer order case.
\end{itemize}
\end{remark}

The assumptions that either the $D$-Hardy inequalities \eqref{eqn-HI}
and \eqref{eqn-IHI} or the equivalences
in \eqref{h-con1} and \eqref{eqn-W1p} hold is very strong. For instance, \eqref{eqn-HI} and \eqref{eqn-IHI}
may fail to hold even for constant functions. On the other hand, 
constants may belong to the intrinsic space $\mathcal{W}^{s,p}(\Omega)$,
which are excluded from the space $\mathring W_{D}^{s,p}(\Omega)$.
Our next proposition shows that, under the case $sp \neq 1$  
and some mild geometric conditions on $\Omega$, 
\eqref{eqn-HI} and \eqref{eqn-IHI} are equivalent, respectively, to
\eqref{h-con1} and \eqref{eqn-W1p}.  It is worth noting that, 
in the applications of this article,
we will employ the weighted fractional Sobolev space
in Definition \ref{eqn-W1p} to substitute the stronger 
conditions given by \eqref{eqn-HI} and \eqref{eqn-IHI}.
Recall that a set $E\subset\mathbb{R}^n$ is called a $d$-{\it set}
with $0\le d\le n$
(see, e.g., \cite[(4.29)]{Tri08}) if there exists a constant $C\in (0,\infty)$
such that
\begin{align}\label{def-d-set}
C^{-1}r^d\le \mathcal{H}^d\left(B(x,r)\cap E\right)\le Cr^d
\end{align}
holds for every $x\in E$ and $r\in(0,1]$, where $\mathcal{H}^d$ denotes the
$d$-dimensional Hausdorff measure and $B(x,r):=\{y\in\mathbb{R}^n:\,|y-x|<r\}$.
If \eqref{def-d-set} holds for any $x\in E$ and
$r\in(0,\text{diam\,}E)$, then $E$ is called a {\it uniformly} $d$-{\it set},
here and thereafter, $\text{diam\,}E$ denotes the {\it diameter} of the set $E$.

\begin{proposition}\label{prop-Hcx}
Let $\Omega\subset\mathbb{R}^n$ be an $(\epsilon,\delta,D)$-domain with $\Omega$ being an $n$-set and
$D\subset\partial\Omega$ a closed uniformly $(n-1)$-set.
Suppose $s\in(0,1]$ and $p\in(1,\infty)$ satisfy $sp\neq1$. Then
\begin{align*}
\mathring{W}_D^{s,p}(\Omega)=
\left\{f\in \mathcal W^{s,p}(\Omega):\ \int_\Omega\frac{|f(x)|^p}
{[{{\mathop\mathrm{dist\,}}}(x,D)]^{sp}}\,dx<\infty\right\}.
\end{align*}
\end{proposition}

\begin{remark}\label{rem-1.5}
\begin{itemize}
\item[(i)] Assume that $\Omega$ satisfies the same geometric assumptions as those in
Proposition \ref{prop-Hcx}. By Theorem \ref{thm-equiv} and Proposition \ref{prop-Hcx},
we find that, if  the equivalences in \eqref{h-con1} hold, then, for any $s\in(0,1)$ and $p\in(1,\infty)$
satisfying $sp\neq1$, the $D$-adapted Hardy inequalities \eqref{eqn-HI} and \eqref{eqn-IHI}
hold (see Lemma \ref{mixed-H} for more details).
This indicates that the $D$-adapted Hardy inequalities \eqref{eqn-HI}
and \eqref{eqn-IHI} are also
necessary conditions to guarantee
\eqref{h-con1} in this setting.

\item[(ii)] From the previous arguments,
we deduce that Proposition \ref{prop-Hcx} is an extension of \eqref{cdw}, which
corresponds to the endpoint case $D=\partial \Omega$.
The restriction $sp\neq1$ here comes from the verification of the $D$-adapted Hardy inequality
in the proof (see Lemma \ref{mixed-H} and \cite[Theorem 1.1]{FrSe10}
for such a restriction even in the half space $\mathbb{R}_+^n$).

\item[(iii)]
Recall that,  in many literatures working on fractional Sobolev spaces with partially vanishing traces
(see, e.g., \cite{Bec19-In,EHT14}), one usually needs to use
the space  $W_D^{s,p}(\Omega)$  in \eqref{eqn-fstraceD}
for the case $sp>1$, but to use the space $W^{s,p}(\Omega)$ in \eqref{eqn-fstrace}
for the case $sp<1$.
Assume now  $\Omega$ satisfies the same geometric assumptions as those in
Proposition \ref{prop-Hcx}.
Recall that it was proven in \cite[Lemma 3.3]{Bec19-In} that, for any $sp>1$,
\begin{align}\label{eqn4.15}
\mathring W_D^{s,p}(\Omega)=W_D^{s,p}(\Omega).
\end{align}
On the other hand,  Theorem \ref{thm-equiv}, combined with the proof of Lemma \ref{mixed-H},
implies that, for any $sp<1$,
$\mathring{W}_D^{s,p}(\Omega)=W^{s,p}(\Omega).$
This shows that the space $\mathring{W}_D^{s,p}(\Omega)$
provides a unified framework of function spaces to deal with  the case $sp\ne 1$.
\end{itemize}
\end{remark}

As mentioned in the above Remark \ref{rem-1.5}(iii), the spaces
 $\mathring{W}_D^{s,p}(\Omega)$ are suitable for all $sp\ne1$.
To cover the critical case $sp=1$,
we need the following \emph{weighted} Hardy-type fractional Sobolev spaces $\mathcal{W}^{s,p}_{d_D^s}(\Omega)$,
which is precisely the intersection space 
${\mathcal{W}}^{s,p}(\Omega)\cap L^p(\Omega,d_D^{-sp})$ used in \cite[Theorem 1.1]{Bec21}.
For further details regarding such weighted Sobolev spaces, 
we refer to \cite{DaHaScSi19} for the integer order case and
to \cite{Ki25,Ro25} for certain closely related fractional order cases.

\begin{definition}\label{def-wfs}
Let $\Omega\subset\mathbb{R}^n$ be a domain
with $D\subset\overline{\Omega}$.
Suppose $s\in(0,1]$ and $p\in[1,\infty)$.  The \emph{weighted fractional Sobolev space}
$\mathcal{W}^{s,p}_{d_D^s}(\Omega)$ is defined by setting
\begin{align*}
{\mathcal{W}}^{s,p}_{d_D^s}(\Omega):=\left\{f\in \mathcal{W}^{s,p}(\Omega):\ \int_\Omega\frac{|f(x)|^p}
{[{{\mathop\mathrm{dist\,}}}(x,D)]^{sp}}\,dx<\infty\right\}
\end{align*}
equipped with the norm
\begin{align*}
 \|f\|_{{\mathcal{W}}^{s,p}_{d_D^s}(\Omega)}:=\|f\|_{\mathcal{W}^{s,p}(\Omega)}
 +\left\{\int_\Omega\frac{|f(x)|^p}
 {[{{\mathop\mathrm{dist\,}}}(x,D)]^{sp}}\,dx\right\}^{\frac 1p},
\end{align*}
here and thereafter, $d_D(x):={\rm dist}\,(x,D)$ for any $x\in \Omega$.
\end{definition}

\begin{remark}\label{rem1.9}
For any $s\in (0,1]$ and $p\in (1,\infty)$, let $\Omega$ be an $(\epsilon,\delta,D)$-domain with
 $D\subset\partial\Omega$ being closed. By the proof of Theorem \ref{thm-equiv} and Remark \ref{rem-notepf},
we always have the inclusion
${\mathcal{W}}^{s,p}_{d_D^s}(\Omega)\subset \mathring{W}^{s,p}_D(\Omega).$
On the other hand, if $sp\ne 1$, $\Omega$ is an $n$-set, and
$D$ is a uniformly $(n-1)$-set, then, by Proposition \ref{prop-Hcx},
we obtain the identity
\begin{align}\label{eqn-ee}
{\mathcal{W}}^{s,p}_{d_D^s}(\Omega)= \mathring{W}^{s,p}_D(\Omega).
\end{align}
Notice that in general we cannot expect the identity \eqref{eqn-ee}
in the critical case $sp=1$, due to the failure of the Hardy inequality as mentioned in Remark \ref{rem-1.5}(ii).
\end{remark}

We now use the space ${\mathcal{W}}^{s,p}_{d_D^s}(\Omega)$ to
characterize the  real interpolation space
$(L^p(\Omega),\mathring{W}_{D}^{1,p}(\Omega))_{s,p}$ for any $s\in (0,1)$ and $p\in (1,\infty)$

\begin{proposition}\label{t5}
Let $s\in (0,1)$, $p\in (1,\infty)$, and
 $\Omega\subset\mathbb{R}^n$ be an $(\epsilon,\delta,D)$-domain
with and $D\subset\partial\Omega$
being a closed uniformly $(n-1)$-set.
Then it holds that
\begin{align}\label{int-hh}
\left(L^p(\Omega),\mathring{W}_{D}^{1,p}
(\Omega)\right)_{s,p}={\mathcal{W}}_{d_D^{s}}^{s,p}(\Omega)
\end{align}
with $\Omega$ being an $n$-set when $sp=1$.
\end{proposition}

Proposition \ref{t5} is an analogue of  \cite[Theorem 1.1(b)]{Bec19-In},
where the additional restriction $sp\ne 1$ is needed and, moreover,
the geometric assumptions on $\Omega$ therein are also different from
those in Proposition \ref{t5}. Let next $sp\ne 1$.
In this case, by Remark \ref{rem1.9}, we have
${\mathcal{W}}_{d_D^{s}}^{s,p}(\Omega)=\mathring{W}_{D}^{s,p}(\Omega)$.
Let $\Omega$ be an $(\epsilon,\delta,D)$-domain with $D$ being a uniformly $(n-1)$-set. Assume that
$\Omega$ satisfies the {interior thickness condition in $\Gamma$} as in \eqref{eqn-isc}.
It was proven in \cite[Theorem 4.0.4]{Bec22} that the following interpolation identity
\begin{align*}
\left(L^p(\Omega),\mathring{W}_{D}^{1,p}
(\Omega)\right)_{s,p}=W^{s,p}(\Omega)\cap L^p\left(\Omega,d_D^{-sp}\right)
\end{align*}
holds, where $W^{s,p}(\Omega)$ is defined as in \eqref{eqn-fstrace}.

As mentioned before, the proof of Theorem \ref{thm-equiv}
depends essentially on the existence of the linear bounded extension operator
$\mathcal{E}_D$ from $\mathcal{W}^{s,p}(\Omega)$ to
$\mathring{W}_D^{s,p}(\mathbb{R}^n)$, which is defined as follows
(see Definition \ref{def-extension1} for its precise definition):
\begin{align}\label{ext-x1}
\mathcal{E}_Df(x):=
\begin{cases}
f(x) &\mathrm{if}\ x\in \Omega, \\
0 &\mathrm{if}\ x\in D, \\
\displaystyle\sum_{Q_j\in\mathscr{W}_{\rm e}}\left(E_{Q_j^*}f\right)_{Q_j^*}\psi_j(x) &\mathrm{if}\ x\in{\overline\Omega}^{\complement}.
\end{cases}
\end{align}

\begin{theorem}\label{thm1}
Let $\epsilon\in(0,1]$, $\delta\in(0,\infty]$, and $\Omega\subset\mathbb{R}^n$ be an $(\epsilon,\delta,D)$-domain with
$D\subset\partial\Omega$ being closed.
Assume $s\in(0,1)$, $p\in[1,\infty)$, and
$\mathcal{W}^{s,p}(\Omega)$ is the fractional Sobolev space as in \eqref{eqn0iss}.
If $\Omega$ supports the $D$-adapted Hardy inequality \eqref{eqn-HI},
then the linear extension operator $\mathcal{E}_D$ in \eqref{ext-x1} is
bounded from $\mathcal{W}^{s,p}(\Omega)$ to $\mathring{W}_D^{s,p}(\mathbb{R}^n)$.
\end{theorem}

Theorem \ref{thm1}, together with the characterization of the fractional Sobolev
extension from \cite[Theorem 1.1]{ZY2015},
implies that every $(\epsilon,\delta,D)$-domain supporting the
$D$-adapted Hardy inequality \eqref{eqn-HI} is  an $n$-set as in \eqref{def-d-set}.
Recall that the construction of various linear extension operators from different kinds of function
spaces on domain to the whole space have been extensively studied in
\cite{AlMi13,Bre14,Cal61,Ch84,GHS23,Hei16,Jon1984,
KoZh22,Ste70}.
In particular, Theorem \ref{thm1} is an analogue of \cite[Theorem 1.1]{Bec21}.
To be precise, let $\Omega$ be an open set satisfying the interior thickness condition in
$\Gamma$ as in \eqref{eqn-isc}, Bechtel \cite[Theorem 1.1]{Bec21} shows 
that the linear extension $\mathcal{E}:=\mathbf{E}\circ E_0$,
constructed as a composition of the reflection extension
$\mathbf{E}$ in the sense of Zhou \cite{ZY2015}
and a zero extension $E_0$ from $\Omega$ to some suitable larger superset, is bounded
from the weighted fractional Sobolev space
${\mathcal W}^{s,p}(\Omega)\cap L^p(\Omega,d_D^{-sp})={\mathcal{W}}^{s,p}_{d_D^s}(\Omega)$
to ${W}^{s,p}(\mathbb{R}^n)$ (see also \cite[Theorem 4.0.1]{Bec22}).
Unlike the extension operator $\mathcal{E}$ in \cite[Theorem 1.1]{Bec21}, Theorem \ref{thm1} uses the linear extension
$\mathcal{E}_D$ in  \eqref{ext-x1}, which is the same as that used in the
integer case $s=1$ (see \cite[p.\,27]{Bec19-Ext}). This also shows 
that $\mathcal{E}_D$ in \eqref{ext-x1} is a
universal linear extension on $\mathring W_D^{s,p}(\Omega)$ for all $s\in [0,1]$.

In contrast to the constructions presented in \cite{Bre14,Cal61,Ch84,Jon1984,ShYa24,ShYa25,Ste70},
for functions with support away from  $D$
the operator $\mathcal{E}_D$ performs a zero extension along $D$ and a reflection extension along
$\Gamma$ (see Figure \ref{p3} and Proposition \ref{lem-extensionOP} for more
details). Recall also that in \cite{Bec19-Ext} the same
extension operator $\mathcal{E}_D$ was shown to be bounded from $\mathring{W}_D^{1,p}(\Omega)$ to $\mathring{W}_D^{1,p}(\mathbb{R}^n)$,
where, for any open set $O$ and a closed set $D\subset\overline O$,
$\mathring{W}^{1,p}_D(O):=\overline{C_D^\infty(O)}^{\|\cdot\|_{W^{1,p}(O)}}$
with $\|f\|_{W^{1,p}(O)}:=\|f\|_{L^p(O)}+\|\nabla f\|_{L^p(O)}$ for any
$f\in C_D^\infty(O)$. Thus, Theorem \ref{thm1}
is an extension of  \cite[Theorem 1.2]{Bec19-Ext} from $s=1$ to $s\in(0,1)$.
Unlike the treatment of the case $s=1$ in \cite{Bec19-Ext},
the proof of Theorem \ref{thm1} needs to borrow some ideas from
the estimation of \cite[(2.4)]{ZY2015},
but requires an elaborate modification in light of the partially vanishing
trace conditions on $D$. In particular,
based on the definition of $\mathcal{E}_D$ in \eqref{ext-x1},
for any $f\in\mathcal{W}^{s,p}(\Omega)$, we decompose the norm $\|\mathcal{E}_Df\|_{W^{s,p}(\mathbb{R}^n)}$ into two types of integrals: ${\rm I}(f)$
and ${\rm II}(f)$
[see \eqref{eqn-integralI} and \eqref{eqn-integralII}]. The estimations of ${\rm I}(f)$ and ${\rm II}(f)$ are the most technical part of this article,
which are presented, respectively,  in Sections \ref{sei} and \ref{seii}.
Notice that it is in this part that we need the $D$-adapted Hardy inequality to handle with the partially vanishing trace condition, which is not needed in both the
fractional order case \cite{ZY2015} of Zhou and the integer order case \cite{Bec19-Ext}
of Bechtel et al.

\begin{center}
\begin{figure}[h]
\centering
\includegraphics[width=7cm]{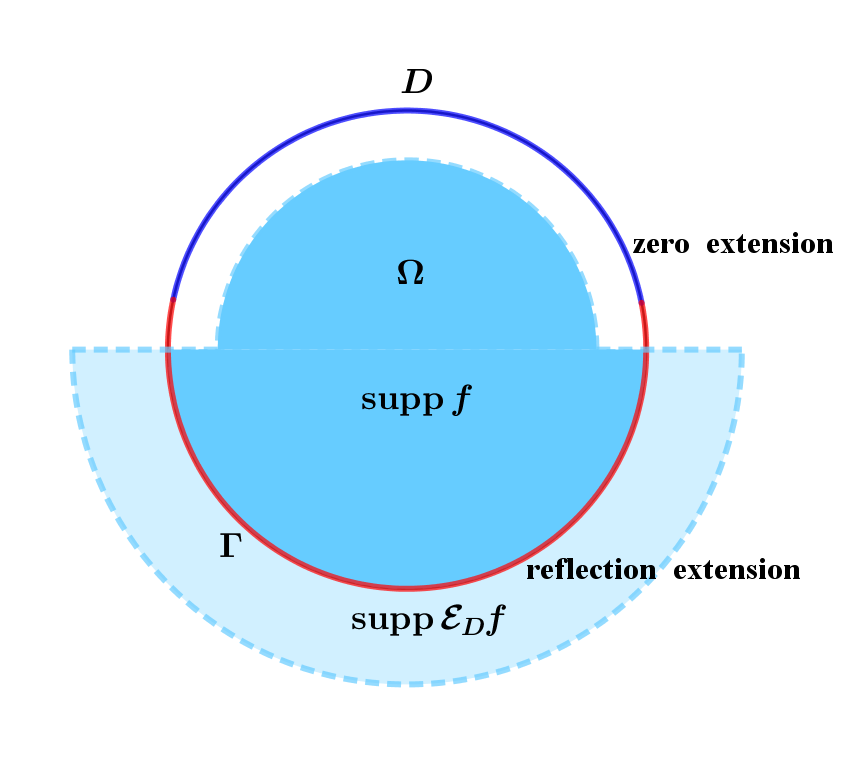}
\caption{Extension of the support for functions in $C_D^\infty(\Omega)$}\label{p3}
\end{figure}
\end{center}

Finally, we apply the above obtained results to the study of second-order
elliptic operators $\mathcal{L}_D$ in $\Omega$ with mixed boundary conditions,
with the hope to extend the corresponding results in existing works
(see, e.g., \cite{Bec21,Bre14,Ege18,EHT14,EgTo17}) from $sp\ne 1$ to the critical case $sp=1$.

To begin with, let us recall some basic definitions on $\mathcal{L}_D$. Let $A=\{a_{ij}\}_{i,j=1}^n:\
\Omega\to \mathbb{R}^{n\times n}$ be a bounded real symmetric matrix
satisfying the \emph{elliptic condition}: there exists $\lambda\in(0,\infty)$ such that,
for a.e. $x\in \Omega$ and any $\xi\in \mathbb{R}^n$, it holds that
$A(x)\xi \cdot\xi\ge \lambda |\xi|^2.$
Let $\Omega\subset \mathbb{R}^n$ be a domain with $D\subset \partial \Omega$
being closed. For any $u$, $v\in \mathcal{W}_{d_D}^{1,p}(\Omega)$,
the symmetric bilinear form $\mathfrak a(u,v)$ is defined by setting
\begin{align}\label{eqn-DF}
\mathfrak a(u,v):=\int_\Omega A(x)\nabla u(x)\cdot \nabla v(x)\,dx.
\end{align}
As proven in Lemma \ref{lem-DF},  $(\mathfrak a,\mathcal{W}_{d_D}^{1,p}(\Omega))$ is a Dirichlet form
in $L^2(\Omega)$. By the theory of Dirichlet forms (see, e.g., \cite{FuOsTa11}),   $(\mathfrak a,\mathcal{W}_{d_D}^{1,p}(\Omega))$ uniquely
corresponds
a nonnegative self-adjoint operator $\mathcal{L}_D$ on $L^2(\Omega)$, which is called a \emph{second-order elliptic operator in $\Omega$
with mixed boundary condition}. We refer to  \cite{Bec21,Bre14,Ege18,EgTo17}  for the definition of elliptic operators in
more general settings where the coefficients of operators are
complex and the theory of sesquilinear forms is used.

Our main concerns on the elliptic operator $\mathcal{L}_D$ with mixed boundary are
the characterizations of its fractional power and its maximal regularity, which are two major motivations
to establish a systematic theory of functional calculus for sectorial operators on Banach spaces
\cite{Haa06}.
The following corollary shows that the weighted fractional Sobolev spaces ${\mathcal{W}}_{d_D^{s}}^{s,2}(\Omega)$
in Definition \ref{def-wfs} can be used to characterize ${\mathrm{dom}}_2(\mathcal{L}_D^{s/2})$, the domain
 in \eqref{eqn-dfp} of the fractional power, for all $s\in (0,1]$ in a unified way.

\begin{corollary}\label{cor1.11}
Let $\Omega\subset\mathbb{R}^n$ be an $(\epsilon,\delta,D)$-domain
with $\Omega$ being an $n$-set and $D\subset\partial\Omega$ being a closed uniformly $(n-1)$-set. Then,
for any $s\in(0,1]$, it holds that
\begin{align*}
{\mathrm{dom}}_2(\mathcal{L}_D^{s/2})
={\mathcal{W}}_{d_D^{s}}^{s,2}(\Omega).
\end{align*}
\end{corollary}

Our second concern is to study the parabolic maximal regularity of  $\mathcal{L}_D$. Let us consider the following
inhomogeneous parabolic equation
\begin{align}\label{eqn-pe}
\begin{cases}
u_t(t,x)+\mathcal{L}_D u(t,x)=f(t,x),\ &(x,t)\in \Omega\times(0,T),\\
u(0,x)=0,   \ &x\in \Omega
\end{cases}
\end{align}
with $T\in (0,\infty)$, $p\in (1,\infty)$, and $f\in L_{\rm loc}^1([0,T); L^p(\Omega))$ belonging to the Bochner space defined as in
\eqref{eqn-BS}.
It is known (see, e.g., \cite{Haa06}) that \eqref{eqn-pe} has a unique mild solution of the form:
for any $t\in [0,T)$,
\begin{align}\label{eqn-milds}
u(t,\cdot)=\int_0^t e^{-(t-s)\mathcal{L}_D}f(s,\cdot)\,ds\in C([0,T); L^p(\Omega)).
\end{align}
The problem of the maximal regularity asks for the additional regularity of the mild
solution, which plays a fundamental role in various problems such as the linearization of the nonlinear parabolic equation
(see, e.g., \cite{Den21,Haa06,HNVW16}).
In particular, a Banach space $E\subset L_{\rm loc}^1([0,T); \ L^p(\Omega))$ is called \emph{a space of the maximal regularity}
for $\mathcal{L}_D$ if,
for any $f\in E$, the corresponding mild solution $u$ in
\eqref{eqn-milds} satisfies
\begin{align}\label{eqn-mr}
u_t, \,\mathcal{L}_D u\in E.
\end{align}
As pointed out in \cite{Haa06}, this means that the mild solution should have the best regularity
property one can expect, with respect to both the operators
$\mathcal{L}_D$ and $\frac{d}{dt}$.

Now, let
$(p_-,q_+)$ be the internal of the maximal interval of $p$ such that
${\rm{dom}}_p(\mathcal{L}_D^{1/2})=\mathring W^{1,p}_D(\Omega)$ (see Remark
\ref{rem5.7} for more details on its definition). Our next result shows that the weighted Sobolev spaces
$W_{d_D^s}^{s,p}(\Omega)$ can be used to construct spaces of the maximal regularity for
$\mathcal{L}_D$ with $p\in (p_-,q_+)$.

\begin{corollary}\label{cor1.12}
Let $p\in (p_-,q_+)$ and
 $\Omega\subset\mathbb{R}^n$ be an $(\epsilon,\delta,D)$-domain
with $\Omega$ being an $n$-set and $D\subset\partial\Omega$ being a closed uniformly $(n-1)$-set.
Then, for any $s\in(0,1]$, the Bochner space $L^p([0,T);W_{d_D^s}^{s,p}(\Omega))$
is a space of the maximal regularity for $\mathcal{L}_D$.
\end{corollary}

A key to prove Corollary \ref{cor1.12} is to show that $-\mathcal{L}_D$ generates a
Markovian semigroup which has a heat kernel $p_t$ satisfying the following Gaussian upper bound
(see Lemma \ref{lem4.7}): there exist constants $w_0\in \mathbb{R}$ and
$c$, $b\in (0,\infty)$ such that, for any $t\in (0,\infty)$ and
a.e. $x$, $y\in \Omega$, it holds that
\begin{align*}
0\le p_t(x,y)\le ct^{-n/2}\exp\left\{w_0t-\frac{b|x-y|^2}{t}\right\}.
\end{align*}

The remainder of this article is organized as follows. Section \ref{s2} is a preliminary,
in which we provide some fundamental properties of $(\epsilon,\delta,D)$-domains.
Section \ref{s3} is devoted to the proof of Theorem \ref{thm1}. We
prove Theorem \ref{thm-equiv} and Proposition
\ref{prop-Hcx} in  Section \ref{s4}.
Finally, we prove  Proposition \ref{t5} and Corollaries \ref{cor1.11} and \ref{cor1.12} in Section \ref{s5}.

We end this section by making some conventions on symbols. Throughout this article, let $\mathbb{N}:=\{1,2,...\}$ and
$\mathbb{N}_0:=\mathbb{N}\cup\{0\}$. We use $C$ to denote a positive constant that is independent of
the main parameters involved, whose value may differ from line to line. The symbol $A\lesssim B$ means
that $A\le CB$ for some positive constant $C$, while $A\sim B$ means $A\lesssim B$ and $B\lesssim A$.
By a domain, we mean a connected open set.
For any set $E\subset\mathbb{R}^n$, denote by $E^\complement$ the set $\mathbb{R}^n\setminus E$ and
$\overline{E}$ the closure of $E$, and we use $|E|$ to denote its Lebesgue measure;
for any $x\in\mathbb{R}^n$, let
\begin{align*}
{{\mathop\mathrm{dist\,}}}(x,E):=\inf_{y\in E}|x-y|\ \ \text{and} \ \ {{\mathop\mathrm{dist\,}}}(E,F):=\inf_{y\in E}{{\mathop\mathrm{dist\,}}}(y,F)
\end{align*}
denote, respectively, the Euclidean distance of $x$ and $E$ and the distance of $E$, $F\subset\mathbb{R}^n$;
we also use $\mathbf{1}_E$ to denote its characteristic function and, when $|E|\in (0,\infty)$, let
\begin{align*}
f_E:=\fint_Ef(x)\,dx=\frac{1}{|E|}\int_Ef(x)\,dx
\end{align*}
denote the integral mean of a function $f\in L_{\rm loc}^1(\mathbb{R}^n)$ (the set of all locally integrable functions on
$\mathbb{R}^n$) over $E$. Moreover, we denote
by $Q=Q(x,r)$ the cube with center $x$ and edge length $r$, whose edges parallel
to the coordinate axes, and by $\ell(Q)=2r$ the edge length of $Q$. For any positive constant $C$,
we use $CQ$ to denote the concentric cube of $Q$ with edge length $C\ell(Q)$.
Finally, in all proofs, we consistently retain the symbol introduced in the
original theorem (or related statement).

\section{\texorpdfstring{Preliminaries on $(\epsilon,\delta,D)$-domains}{Preliminaries on domains}}\label{s2}

In this section, we present some basic properties on $(\epsilon,\delta,D)$-domains.
Subsection \ref{s2.1} is devoted to a review of the reflection mechanism of $\Omega$.
Based on this mechanism, we establish some technical lemmas in Subsection \ref{s2.2}, which play important roles in the proofs of the main results of this article.

\subsection{\texorpdfstring{Reflection across domains}{Reflection across domains}}\label{s2.1}

A key feature of an $(\epsilon,\delta,D)$-domain $\Omega$ in Definition \ref{ass-1} is that there exists a reflection mechanism for Whitney cubes
related to $\Omega$. To illustrate it, let us recall the Whitney decomposition as follows  (see, e.g., \cite{Gra08,Ste70}).

\begin{lemma}[\cite{Gra08,Ste70}]\label{W}
Let $\Omega\subsetneqq\mathbb{R}^n$ be an open set. Then there exists a family of
cubes $\mathscr{W}:=\{Q_j\}_{j=1}^\infty$
with edges parallel to the coordinate axes and pairwise disjoint interiors such that
\begin{itemize}
\item[{\rm (i)}]$\bigcup_{j=1}^\infty Q_j=\Omega$,

\item[{\rm (ii)}]${\mathop\mathrm{diam\,}} Q_j\le{{\mathop\mathrm{dist\,}}}(Q_j,\partial\Omega)
\le4{\mathop\mathrm{\,diam\,}} Q_j$ for any $j\in\mathbb{N}$,

\item[{\rm (iii)}] $\{Q_j\}_{j=1}^\infty$ are dyadic cubes,

\item[{\rm (iv)}]$\frac{1}{4}{\mathop\mathrm{\,diam\,}}Q_j\le{\mathop\mathrm{diam\,}}Q_k\le
4{\mathop\mathrm{\,diam\,}}Q_j$ if $\overline{Q}_j\cap\overline{Q}_k\neq\emptyset$, and

\item[{\rm (v)}]each cube has at most $12^n$ intersecting cubes.
\end{itemize}
\end{lemma}

Let $\mathscr{W}$ be a Whitney decomposition of $\Omega$
as in Lemma \ref{W}.
We call $\{Q_1,\ldots,Q_m\}\subset\mathscr{W}$
a {\it touching chain} in $\mathscr{W}$ if, for any $i\in \{1,...,m-1\}$, it holds that
$\overline{Q}_i\cap\overline{Q}_{i+1}\ne\emptyset.$

\begin{remark}\label{remark2.3}
Assume that $\Omega$ is an $(\epsilon,\delta,D)$-domain.
Let $\mathscr{W}(\overline{\Gamma})$ and $\mathscr{W}(\overline{\Omega})$ be, respectively, the
Whitney decompositions of $\mathbb{R}^n\setminus\overline{\Gamma}$ and $\mathbb{R}^n\setminus\overline{\Omega}$ as in Lemma \ref{W}.
The {\it shadow} $\mathscr{W}_{\rm i}$ of $\Omega$ in $\mathscr{W}(\overline{\Gamma})$ is defined by setting
\begin{align}\label{Wi}
\mathscr{W}_{\rm i}:=\left\{Q\in\mathscr{W}(\overline{\Gamma}):\, Q\cap\Omega\neq\emptyset\right\},
\end{align}
and let
\begin{align}\label{We}
\mathscr{W}_{\rm e}:=\left\{Q\in\mathscr{W}(\overline\Omega):\, \ell(Q)\le A\delta\ \text{ and }
{{\mathop\mathrm{\,dist\,}}}(Q,\Gamma)<B{{\mathop\mathrm{\,dist\,}}}(Q,D)\right\}
\end{align}
be the truncated cone over $\Gamma$ in $\mathscr{W}(\overline\Omega)$
for some fixed parameters $A\in(0,\infty)$ and $B\in(2,\infty)$
with $A$ small enough and $B$ sufficiently large.
It is easy to verify that, if $D=\partial\Omega$, then $\mathscr{W}_{\rm e}\neq\emptyset$. If $D\neq\partial\Omega$, then,
for any $Q\subset\mathscr{W}_{\rm e}$, it holds that
${{\mathop\mathrm{dist\,}}}(Q,\partial\Omega)\le{{\mathop\mathrm{dist\,}}}(Q,\Gamma)\le B{{\mathop\mathrm{\,dist\,}}}(Q,\partial\Omega).$
\end{remark}

The following lemma shows that there exists a reflection between cubes in $\mathscr{W}_{\rm e}$ and $\mathscr{W}_{\rm i}$.

\begin{lemma}[\cite{Bec19-Ext}]\label{lem3.4.4}
There exist two positive constants $C,M$ such that the following assertions hold.
\begin{itemize}
\item[{\rm (i)}]For any $Q\in\mathscr{W}_{\rm e}$ as in \eqref{We}, there exists some $Q^*\in\mathscr{W}_{\rm i}$ satisfying
\begin{align}\label{2.3}
C^{-1}{\mathop\mathrm{diam\,}}Q\le{\mathop\mathrm{diam\,}}Q^*\le C{\mathop\mathrm{\,diam\,}}Q
\end{align}
and
\begin{align}\label{2.4}
{{\mathop\mathrm{dist\,}}}(Q,Q^*)\le C{\mathop\mathrm{\,diam\,}} Q.
\end{align}

\item[{\rm (ii)}]For any $Q^*\in\mathscr{W}_{\rm i}$ as in \eqref{Wi}, there are at most $M$ cubes $Q\in\mathscr{W}_{\rm e}$ satisfying \eqref{2.3} and \eqref{2.4}.

\item[{\rm (iii)}]For any $Q_1,Q_2\in\mathscr{W}_{\rm e}$ with $\overline{Q}_1\cap\overline{Q}_2\neq\emptyset$,
it holds that ${{\mathop\mathrm{dist\,}}}(Q_1^*,Q_2^*)\le C{\mathop\mathrm{\,diam\,}} Q_1.$
\end{itemize}
\end{lemma}

\begin{remark}
The existence of the reflected cubes in Lemma \ref{lem3.4.4} is based on an important covering property
of quaihyperbolic geodesic curves from \cite{HeKo96},
which says that, for any $x$, $y\in\Omega$,
there exists an intersected chain of balls covering the geodesic curve connecting
$x$ and $y$. The length of this chain depends only on
the quaihyperbolic metric $k_{\mathbb{R}^n\setminus\Gamma}(x,y)$. In view of this, for any
$Q\in\mathscr{W}_{\rm e}$, let $x\in\Omega$ be a point near $Q$ and $y\in\Omega$
satisfy $k_{\mathbb{R}^n\setminus\Gamma}(x,y)<K$. The
reflected cube $Q^*$ can be chosen as the last cube $Q_m$ of the touching
cubes $\{Q_1,\ldots,Q_m\}\subset\mathscr{W}_{\rm i}$ that covers the above chain of balls.
We refer the proof of \cite[Lemma 5.4]{Bec19-Ext} for more details on this construction.
\end{remark}

The following lemma is precisely \cite[Lemmas 5.7 and 5.8]{Bec19-Ext}, which
provides two useful touching chains in $\mathscr{W}_{\rm i}$.

\begin{lemma}[\cite{Bec19-Ext}]\label{lem5.7}
There exists two constants $m\in\mathbb{N}$ and $C\in(0,\infty)$ such that the following statements hold.
\begin{itemize}
\item[{\rm (i)}]
For any $P,Q\in\mathscr{W}_{\rm e}$ satisfying $\overline{P}\cap\overline{Q}\neq\emptyset$, there exists a touching
chain $F_{P,Q}:=\{P^*=S_1,...,S_{m_1}=Q^*\}$ in
$\mathscr{W}_{\rm i}$ with $m_1\le m$. Moreover, let
\begin{align*}
F(Q):=\bigcup_{\genfrac{}{}{0pt}{}{P\in\mathscr{W}_{\rm e}}
{\overline{P}\cap\overline{Q}\neq\emptyset}}\bigcup_{S\in F_{P,Q}}2S.
\end{align*}
Then, for any $x\in\mathbb{R}^n$, it holds that
$\sum_{Q\in\mathscr{W}_{\rm e}}\mathbf{1}_{F(Q)}(x)\le C.$

\item[{\rm (ii)}]
Let $\mathscr{W}(\overline{\Omega})$  be as in Remark \ref{remark2.3} and
\begin{align}\label{We'}
 \mathscr{W}_{\rm e}':=&\left\{Q\in\mathscr{W}_{\rm e}:\ \text{for any }\  P\in \mathscr{W}(\overline{\Omega}) \
 \text{satisfying}\
 \overline{P}\cap\overline{Q}\neq\emptyset \right.\notag\\
 &\qquad\qquad\ \text{ and }\ P\neq Q,\ \text{it holds that}\ P\in\mathscr{W}_{\rm e}\Big\}
\end{align}
the inner part of $\mathscr{W}_{\rm e}$. For any $Q\in\mathscr{W}_{\rm e}
\setminus\mathscr{W}_{\rm e}'$ satisfying ${\mathop\mathrm{diam\,}}Q\le A\delta/4$,
there exists a touching chain $F_Q:=\{Q^*=S_1,...,S_{m_2}\}$ in
$\mathscr{W}_{\rm i}$ with $m_2\le m$
such that $S_{m_2}\cap Q$ is a dyadic cube and satisfies
$C^{-1}|Q|\le |S_{m_2}\cap Q|\le C|Q|.$
Moreover, for any $x\in\mathbb{R}^n$, it holds that
\begin{align*}
\sum_{\genfrac{}{}{0pt}{}{Q\in\mathscr{W}_{\rm e}\setminus\mathscr{W}_{\rm e}'}{\ell(Q)\le A\delta}}\sum_{i=1}^{m_2}\mathbf{1}_{2S_i}(x)\le C.
\end{align*}
\end{itemize}
\end{lemma}

\begin{remark}
\begin{itemize}
\item[{\rm(i)}]Since the constant $m$ in Lemma \ref{lem5.7} is fixed,  we deduce from Lemma \ref{W}(iv) that
any two cubes in the touching chains of Lemma \ref{lem5.7} are comparable.

\item[{\rm(ii)}]Let
\begin{align}\label{we''}
\mathscr{W}_{\rm e}'':=&\left\{Q\in\mathscr{W}_{\rm e}:\ \text{for any }\  P\in \mathscr{W}(\overline{\Omega}) \
 \text{satisfying} \ \overline{P}\cap\overline{Q}\neq\emptyset\right.\notag\\
 &\qquad\qquad
\ \text{ and }\ P\neq Q,\ \text{it holds that}\ P\in\mathscr{W}_{\rm e}'\Big\},
\end{align}
where  $\mathscr{W}_{\rm e}'$ is as in \eqref{We'}.
From the proof of \cite[Lemma 5.8]{Bec19-Ext}, it follows that,
for any $Q\in\mathscr{W}_{\rm e}'\setminus
\mathscr{W}_{\rm e}''$ satisfying ${\mathop\mathrm{diam\,}}Q\le A\delta/16$,
there exists also a touching chain $F_Q$ in $\mathscr{W}_{\rm i}$ satisfying all the assertions of
Lemma \ref{lem5.7}(ii) when the constant $B$ in \eqref{We} is large enough.
\end{itemize}
\end{remark}

\subsection{Several technical lemmas}\label{s2.2}

In this subsection, we collect several technical lemmas on $(\epsilon,\delta,D)$-domains which are used later.

For any cubes $P$ and $Q$, recall that their {\it long distance} $D(P,Q)$ is defined by setting
\begin{align}\label{l-dis}
D(P,Q):={\mathop\mathrm{diam\,}}P+{\mathop\mathrm{diam\,}}Q+{{\mathop\mathrm{dist\,}}}(P,Q).
\end{align}

\begin{lemma}\label{lem4.5-1}
There exists a constant $C\in (0,\infty)$ such that,
for any $P,Q,S\in\mathscr{W}_{\rm e}$ satisfying $\overline{P}\cap\overline{Q}\neq\emptyset$
and for any $x\in Q$ and
$y\in c_nS\setminus B(x,\ell(Q)/10)$ with
$c_n:=1+\frac{1}{16\sqrt{n}}$, it holds that
\begin{align}\label{eqn4.5.1-1}
D(P^*,S^*)\le C|x-y|.
\end{align}
\end{lemma}

\begin{proof}
By \eqref{l-dis} and Lemmas \ref{W} and \ref{lem3.4.4}(i), we have
\begin{align}\label{eqn-DPS}
D(P^*,S^*)
&\sim D(P,S)\sim {\mathop\mathrm{diam\,}}P+{\mathop\mathrm{diam\,}}S
+{{\mathop\mathrm{dist\,}}}(P,S)\notag\\ \notag
&\lesssim {\mathop\mathrm{diam\,}}P+{\mathop\mathrm{diam\,}}S
+{\mathop\mathrm{diam\,}}P+{\mathop\mathrm{diam\,}}Q
+{{\mathop\mathrm{dist\,}}}(Q,S)+{\mathop\mathrm{diam\,}}S\\
&\lesssim {\mathop\mathrm{diam\,}}Q+{\mathop\mathrm{diam\,}}S
+{{\mathop\mathrm{dist\,}}}(Q,S),
\end{align}
where ${\mathop\mathrm{diam\,}} Q\lesssim|x-y|$ because $y\in c_nS\setminus B(x,\ell(Q)/10)$.

To complete the proof of the present lemma, we still need to prove that ${\mathop\mathrm{diam\,}}S$
and ${{\mathop\mathrm{dist\,}}}(Q,S)$ can be bounded by a positive constant multiple of $|x-y|$.
Indeed, if $\overline{S}\cap\overline{Q}\neq\emptyset$, then ${{\mathop\mathrm{dist\,}}}(Q,S)=0$ and
${\mathop\mathrm{diam\,}}S\lesssim{\mathop\mathrm{diam\,}} Q\lesssim|x-y|,$
which, combined with \eqref{eqn-DPS}, shows \eqref{eqn4.5.1-1} in this case.

If $S$ and $Q$ are not touching, then
\begin{align*}
|x-y|\ge{{\mathop\mathrm{dist\,}}}(Q,c_nS)
\ge\max\left\{\frac{1}{4}\ell(Q),\frac{1}{4}\ell(S)\right\}-\frac{1}{32}\ell(S)
\ge\frac{7}{32}\ell(S),
\end{align*}
which implies ${{\mathop\mathrm{dist\,}}}(Q,S)\le {{\mathop\mathrm{dist\,}}}(Q,c_nS)+c_n{\mathop\mathrm{\,diam\,}} S\lesssim|x-y|.$
This also proves \eqref{eqn4.5.1-1} in this case and hence finishes the
proof of Lemma \ref{lem4.5-1}.
\end{proof}

\begin{lemma}\label{dist-2}
Let $x\in\Omega$ and $Q\in\mathscr{W}_{\rm i}$ be as in \eqref{Wi} satisfying $Q\setminus\Omega\neq\emptyset$.
Then, for any $y\in Q\setminus\Omega$, one has ${{\mathop\mathrm{dist\,}}}(x,D)\le2|x-y|.$
\end{lemma}

\begin{proof}
We consider two cases based on the relative position of $x$ and $y$.
If $|x-y|<{\mathop\mathrm{diam\,}} Q$, let $z\in\partial\Omega$ satisfy $|x-y|=|x-z|+|y-z|$.
If $z\in\Gamma$, then,
by the assumption $Q\in\mathscr{W}_{\rm i}$ and Lemma \ref{W}, we conclude that
\begin{align*}
{{\mathop\mathrm{dist\,}}}(Q,\Gamma)\le|y-z|\le|x-y|<{\mathop\mathrm{diam\,}}Q\le{{\mathop\mathrm{dist\,}}}(Q,\Gamma),
\end{align*}
which implies $z\in D$ and hence $|x-y|\ge|x-z|\ge{{\mathop\mathrm{dist\,}}}(x,D)$.

On the other hand, if $|x-y|\ge{\mathop\mathrm{diam\,}}Q$, we can take $z'\in Q\cap D$ and then obtain
\begin{align*}
{{\mathop\mathrm{dist\,}}}(x,D)\le|x-z'|\le|x-y|+|y-z'|\le|x-y|+{\mathop\mathrm{diam\,}}Q\le2|x-y|,
\end{align*}
which completes the proof of Lemma \ref{dist-2}.
\end{proof}

\begin{lemma}\label{bound}
For any $Q\in\mathscr{W}_{\rm e}\setminus\mathscr{W}_{\rm e}'$ with $\mathscr{W}_{\rm e}$ and
$\mathscr{W}_{\rm e}'$, respectively,  as in \eqref{We} and \eqref{We'},
it holds that either $\ell(Q)>A\delta/4$ or
\begin{align}\label{lem-1}
B^{-1}{{\mathop\mathrm{dist\,}}}(Q,\Gamma)<{{\mathop\mathrm{dist\,}}}(Q,D)
\le21{{\mathop\mathrm{\,dist\,}}}(Q,\Gamma),
\end{align}
where the constants $A$ and $B$ are as in \eqref{We}.
\end{lemma}

\begin{proof}
We only need to show that every cube $Q\in\mathscr{W}_{\rm e}\setminus \mathscr{W}_{\rm e}'$ with
$\ell(Q)\le A\delta/4$ satisfies \eqref{lem-1}. From the definitions of $\mathscr{W}_{\rm e}$ and $\mathscr{W}_{\rm e}'$,
we infer that there exists some $S\in\mathscr{W}(\overline{\Omega})\setminus \mathscr{W}_{\rm e}$ such that $\overline{Q}\cap\overline{S}\neq\emptyset$. Thus, by \eqref{We},
we have either ${{\mathop\mathrm{dist\,}}}(S,\Gamma)\ge B{{\mathop\mathrm{\,dist\,}}}(S,D)$ or $\ell(S)>A\delta$.

If $\ell(S)>A\delta$, then, using (iv) of Lemma \ref{W}, it holds that
$\ell(Q)\ge\frac{1}{4}\ell(S)>\frac{1}{4}A\delta,$
which contradicts the assumption $\ell(Q)\le A\delta/4$.
This implies that ${{\mathop\mathrm{dist\,}}}(S,\Gamma)\ge B{{\mathop\mathrm{\,dist\,}}}(S,D)$ and hence
\begin{align}\label{lem2.1-2}
{{\mathop\mathrm{dist\,}}}(S,\partial \Omega) = {{\mathop\mathrm{dist\,}}}(S,D)
\end{align}
because $B>2$. By Lemma \ref{W} again and \eqref{lem2.1-2}, we further find that
\begin{align*}
{{\mathop\mathrm{dist\,}}}(Q,D)&\le{{\mathop\mathrm{dist\,}}}(S,D)
+{\mathop\mathrm{diam\,}}Q+{\mathop\mathrm{diam\,}}S\\
&={{\mathop\mathrm{dist\,}}}(S,\partial\Omega)+{\mathop\mathrm{diam\,}}Q+{\mathop\mathrm{diam\,}}S\\
&\le 4{\mathop\mathrm{\,diam\,}}S+{\mathop\mathrm{diam\,}}Q+{\mathop\mathrm{diam\,}}S\\
&\le 21{\mathop\mathrm{\,diam\,}} Q\le 21{{\mathop\mathrm{\,dist\,}}}(Q,\partial\Omega)\le 21{{\mathop\mathrm{\,dist\,}}}(Q,\Gamma),
\end{align*}
which, together with the definition of $\mathscr{W}_{\rm e}$, implies \eqref{lem-1} and hence completes the proof of Lemma \ref{bound}.
\end{proof}

\begin{lemma}\label{dist-1}
If $\mathscr{W}_{\rm e}'$ is as in \eqref{We'} and
\begin{align}\label{eqn-boze'}
\Omega_{\rm e}':=\bigcup_{Q\in\mathscr{W}_{\rm e}'}Q,
\end{align}
then there exists a constant $C_{\rm e}\in (0,\infty)$ such that, for any $x\in\Omega$ and
$y\in\Omega^{\complement}\setminus\Omega_{\rm e}'$, either $|x-y|\ge C_{\rm e}$ or
$|x-y|\ge C_{\rm e}{{\mathop\mathrm{\,dist\,}}}(x,D)$.
\end{lemma}

\begin{proof}
Since $y\in\Omega^{\complement}\setminus\Omega_{\rm e}'$, it follows that there
exists $Q\in\mathscr{W}(\overline{\Omega})\setminus \mathscr{W}_{\rm e}'$ such that $y\in Q$.
We continue the proof by considering the following two cases based on the position of $Q$.

{\bf Case 1}: $Q\in\mathscr{W}(\overline{\Omega})\setminus \mathscr{W}_{\rm e}$
with $\mathscr{W}_{\rm e}$ as in \eqref{We}.
In this case, if
$\ell(Q)>A\delta$, then, by the assumption $x\in\Omega$, $Q\cap\Omega=\emptyset$, and
Lemma \ref{W}(iv), we have
$|x-y|\ge{{\mathop\mathrm{dist\,}}}(y,\partial\Omega)\ge{\mathop\mathrm{diam\,}}Q\ge\sqrt{n}A\delta.$
On the other hand, if $\ell(Q)\le A\delta$, then, using \eqref{We}, we conclude that
${{\mathop\mathrm{dist\,}}}(Q,\Gamma)\ge B{{\mathop\mathrm{\,dist\,}}}(Q,D)$.
Since $B>2$, it is easy to infer ${{\mathop\mathrm{dist\,}}}(Q,\partial\Omega)={{\mathop\mathrm{dist\,}}}(Q,D)$. Thus,
by Lemma \ref{W}(iv) again, we obtain
\begin{align*}
{{\mathop\mathrm{dist\,}}}(y,D)\le{{\mathop\mathrm{dist\,}}}(Q,D)+{\mathop\mathrm{diam\,}} Q\le 2{{\mathop\mathrm{\,dist\,}}}(Q,\partial\Omega)\le 2|x-y|,
\end{align*}
and hence ${{\mathop\mathrm{dist\,}}}(x,D)\le |x-y|+{{\mathop\mathrm{dist\,}}}(y,D)\le 3|x-y|$,
which proves the present lemma in Case 1 by letting $C_{\rm e}:=\min\{3,\sqrt{n}A\delta\}$.

{\bf Case 2}: $Q\in\mathscr{W}_{\rm e}\setminus\mathscr{W}_{\rm e}'$.
The argument in this case is similar to that used in the proof of Case 1 in view of Lemma \ref{bound}, the details being omitted.

Combining the above two cases then completes the proof of  Lemma \ref{dist-1}.
\end{proof}

Let
\begin{align}\label{eqn-boze}
\Omega_{\rm e}:=\bigcup_{Q\in\mathscr{W}_{\rm e}}c_nQ
\end{align}
with $c_n:=1+\frac{1}{16\sqrt{n}}$.

\begin{corollary}\label{dist-1xx}
There exists a constant $C\in(0,\infty)$ such that, for any $x\in Q\in\mathscr{W}_{\rm e}$,
$y\in\Omega^{\complement}\setminus\Omega_{\rm e}$, and $\xi\in P_j^*\cap\Omega$ for some $P_j\in\mathscr{W}_{\rm e}$
satisfying $\overline{P}_j\cap\overline Q\neq\emptyset$,
it holds that either $|x-y|\ge C$ or $|x-y|\ge C{{\mathop\mathrm{\,dist\,}}}(\xi,D)$.
\end{corollary}

\begin{proof}
By Lemmas \ref{lem3.4.4}(i) and \ref{W}(ii) and the definition of $\Omega_{\rm e}$, we have
\begin{align*}
|x-\xi|\le{\mathop\mathrm{diam\,}}Q+{{\mathop\mathrm{dist\,}}}(Q,P_j^*)+{\mathop\mathrm{diam\,}} P_j^*\lesssim {\mathop\mathrm{diam\,}} Q \lesssim |x-y|,
\end{align*}
and hence
\begin{align}\label{eqn-xyd}
|y-\xi|\le |y-x|+|x-\xi|\lesssim |x-y|.
\end{align}
On the other hand, since $\xi\in\Omega$ and $y\in\Omega^{\complement}\setminus\Omega_{\rm e}$,
we deduce from Lemma \ref{dist-1} that there exists a constant $C\in (0,\infty)$ such that either $|y-\xi|\ge C$ or $|y-\xi|\ge C{{\mathop\mathrm{\,dist\,}}}(\xi,D)$,
which, combined with \eqref{eqn-xyd}, then completes the proof of Corollary \ref{dist-1xx}.
\end{proof}

\begin{lemma}\label{lem4.6-1}
Let $\mathscr{W}_{\rm e}$ and $\mathscr{W}_{\rm e}''$ be defined as, respectively, in \eqref{We} and \eqref{we''}.
Then there exists a constant $C\in (0,\infty)$ such that, for any $Q\in \mathscr{W}(\overline{\Omega})\setminus \mathscr{W}_{\rm e}''$
satisfying $\ell(Q)\le A\delta/16$, $P\in\mathscr{W}_{\rm e}$ satisfying $\overline{P}\cap\overline Q\neq\emptyset$, and
$x\in P^*\cap\Omega$, it holds that
\begin{align}\label{lem4.6-2}
{{\mathop\mathrm{dist\,}}}(x,D)\le C{\mathop\mathrm{\,diam\,}}Q.
\end{align}
\end{lemma}

\begin{proof}
We consider the following three cases for $P$ and $Q$.

{\bf Case 1}: $P^*\cap D\neq\emptyset$. In this case, from Lemmas \ref{W} and \ref{lem3.4.4}, we immediately infer that
\begin{align*}
{{\mathop\mathrm{dist\,}}}(x,D)\le{\mathop\mathrm{diam\,}}P^*\lesssim {\mathop\mathrm{diam\,}}Q,
\end{align*}
which is desired.

{\bf Case 2}: $P^*\subset\Omega$ and $Q\in\mathscr{W}(\overline{\Omega})\setminus \mathscr{W}_{\rm e}$.
In this case, by the definition of $\mathscr{W}_{\rm e}'$ as in \eqref{We'}, it is easy to conclude
$P\in\mathscr{W}_{\rm e}\setminus \mathscr{W}_{\rm e}'$ and $\ell(P)\le 4\ell(Q)\le A\delta/4$.
Using Lemmas \ref{bound} and \ref{lem3.4.4}, we obtain
\begin{align*}
{{\mathop\mathrm{dist\,}}}(P^*,D)&\le{\mathop\mathrm{diam\,}}P^*
+{{\mathop\mathrm{dist\,}}}(P^*,P)+{\mathop\mathrm{diam\,}}P
+{{\mathop\mathrm{dist\,}}}(P,D)\\
&\lesssim {\mathop\mathrm{diam\,}}P+{{\mathop\mathrm{dist\,}}}(P,D)
\lesssim {{\mathop\mathrm{dist\,}}}(P,\partial\Omega)
\lesssim {\mathop\mathrm{diam\,}}P,
\end{align*}
and hence, from Lemmas \ref{W} and \ref{lem3.4.4}
and the assumption $\overline{P}\cap\overline Q\neq\emptyset$, it follows that
\begin{align*}
{{\mathop\mathrm{dist\,}}}(x,D)\le{{\mathop\mathrm{dist\,}}}(P^*,D)+{\mathop\mathrm{diam\,}}
P^*\lesssim {\mathop\mathrm{diam\,}}Q.
\end{align*}

{\bf Case 3}: $P^*\subset\Omega$ and $Q\in\mathscr{W}_{\rm e}\setminus \mathscr{W}_{\rm e}''$.
In this case, by the definition of $\mathscr{W}_{\rm e}''$ as in \eqref{we''}, there exists
$S\in\mathscr{W}_{\rm e}\setminus \mathscr{W}_{\rm e}'$ such that $\overline{S}\cap\overline Q\neq\emptyset$ and consequently
$\ell(S)\le 4\ell(Q)\le A\delta/4$ due to Lemma \ref{W}(iv). From Lemma \ref{bound} and
an argument similar to that used in the proof of Case 2, we deduce that
\begin{align*}
{{\mathop\mathrm{dist\,}}}(x,D)&\le{{\mathop\mathrm{dist\,}}}(P^*,D)+{\mathop\mathrm{diam\,}}P^*\\
&\le{\mathop\mathrm{diam\,}}P^*+{{\mathop\mathrm{dist\,}}}(P^*,P)+{\mathop\mathrm{diam\,}}
P+{\mathop\mathrm{diam\,}}Q\\
&\quad+{\mathop\mathrm{diam\,}}S+{{\mathop\mathrm{dist\,}}}(S,D)+{\mathop\mathrm{diam\,}}P^* \\
&\lesssim {\mathop\mathrm{diam\,}}Q+{{\mathop\mathrm{dist\,}}}(S,D)
\lesssim {\mathop\mathrm{diam\,}}Q.
\end{align*}
Altogether, we obtain \eqref{lem4.6-2}, which hence
completes the proof of Lemma \ref{lem4.6-1}.
\end{proof}

\section{Boundedness of linear extensions}\label{s3}

In this section, we prove Theorem \ref{thm1}, which is a key step to prove our
main results. To this end, we first decompose the norm
$\|\mathcal{E}_Df\|_{\mathcal{W}^{s,p}(\Omega)}$, for any $f\in\mathcal{W}^{s,p}(\Omega)$,
into two types of integrals in Subsection \ref{s3.1}. The estimations of these
two integrals are presented respectively in Subsections \ref{sei} and \ref{seii}.

\subsection{\texorpdfstring{The extension operator $\mathcal{E}_D$}{The extension operator}}\label{s3.1}

Let $\mathscr{W}(\overline{\Omega})=\{Q_j\}_{j=1}^\infty$ be the Whitney decomposition of
$\mathbb{R}^n\setminus \overline{\Omega}$ as in Remark \ref{remark2.3}.
Suppose $\{\psi_j\}_{j=1}^\infty\subset C_{\rm c}^\infty(\mathbb{R}^n)$ (the set of all infinitely
differentiable functions with compact support) is the {\it partition of unity} associated with
$\mathscr{W}({\overline{\Omega}})$, that is, for any $j\in\mathbb{N}$,
\begin{itemize}
\item[{\rm (i)}]
$\psi_j\equiv1$ on $Q_j$ and ${\mathop\mathrm{\,supp\,}}\psi_j\subset c_nQ_j$ with $c_n:=1+\frac{1}{16\sqrt{n}}$,

\item[{\rm (ii)}]
for any $x\in\mathbb{R}^n\setminus \overline{\Omega}$, it holds that
\begin{align}\label{eqn-puy}
\sum_{j=1}^\infty\psi_j(x)=1,
\end{align}

\item[{\rm (iii)}]
there exists a constant $C\in (0,\infty)$ such that, for any $j\in\mathbb{N}$, it holds that
\begin{align}\label{pro-psi-3}
\left\|\nabla\psi_j\right\|_{L^\infty(\mathbb{R}^n)}\le\frac{C}{\ell(Q_j)}.
\end{align}
\end{itemize}

Let $\mathscr{W}_{\rm e}\subset\mathscr{W}(\overline{\Omega})$ be defined as in \eqref{We}.
For any $Q\in\mathscr{W}_{\rm e}$, let $Q^*$ be the reflected cube of $Q$ as in Lemma \ref{lem3.4.4}.
Then, for any $f\in L_{\rm loc}^1(\Omega)$  (the set of all locally integrable functions on $\Omega$),
define the mapping
$Ef:=\{E_{Q_j^*}f\}_{Q_j\in\mathscr{W}_{\rm e}}$ by letting
$E_{Q_j^*}f\in L^1(Q_j^*)$ be the zero extension of $f$
from $\Omega$ to $Q_j^*$, that is, for any $x\in Q_j^*$,
\begin{align}\label{eqn-EQj}
E_{Q_j^*}f(x):=
\begin{cases}
f(x) &\mathrm{if}\ x\in Q_j^*\cap\Omega,\\
0 &\mathrm{if}\ x\in Q_j^*\cap\Omega^{\complement}.
\end{cases}
\end{align}

The following definition of the extension operator is a slight
modification of the one from \cite[p.18]{Bec19-Ext}.

\begin{definition}\label{def-extension1}
Let $\Omega$ be an $(\epsilon,\delta,D)$-domain with $D\subset \partial\Omega$ being
closed and let $f\in L_{\rm loc}^1(\Omega)$. Define the extension $\mathcal{E}_Df$ of $f$ by setting
\begin{align}\label{ext}
\mathcal{E}_Df(x):=
\begin{cases}
f(x) &\mathrm{if}\ x\in\Omega,\\
0 &\mathrm{if}\ x\in D,\\
\displaystyle \sum_{Q_j\in\mathscr{W}_{\rm e}}\left(E_{Q_j^*}f\right)_{Q_j^*}\psi_j(x) &\mathrm{if}\ x\in\overline{\Omega}^{\complement},
\end{cases}
\end{align}
where, for any $Q_j\in\mathscr{W}_{\rm e}$ and $g\in L^1(Q_j^*)$, $g_{Q_j^*}$ denotes its
integral mean over $Q_j^*$ defined by setting
\begin{align*}
g_{Q_j^*}:=\frac{1}{|Q_j^*|}\int_{Q_j^*} g(x)\, dx.
\end{align*}
\end{definition}

\begin{remark}
Although $\mathcal{E}_Df$ is not defined on $\Gamma$, it was proven in
\cite[Lemma 5.1]{Bec19-Ext} that $\Gamma=\partial\Omega\setminus D$
is a Lebesgue null set if $\Omega\subset\mathbb{R}^n$ is an $(\epsilon,\delta,D)$-domain
with $D\subset\partial\Omega$ being closed. Thus, $\mathcal{E}_Df$ is well-defined almost
everywhere in $\mathbb{R}^n$. Moreover, it was proven in \cite[Proposition 6.10]{Bec19-Ext}
that, if $f\in C_D^\infty(\Omega)$, then $\mathcal{E}_Df\in C^\infty(\mathbb{R}^n)$.
\end{remark}

The following proposition shows the boundedness of
$\mathcal{E}_D$ from $\mathcal W^{s,p}(\Omega)$ to $W^{s,p}(\mathbb{R}^n)$.

\begin{proposition}\label{thm1x}
Let $\epsilon\in(0,1]$, $\delta\in(0,\infty]$, and $\Omega\subset \mathbb{R}^n$ be an
$(\epsilon,\delta,D)$-domain with $D\subset \partial\Omega$ being closed.
Assume $s\in(0,1)$, $p\in[1,\infty)$, and $\mathcal{W}^{s,p}(\Omega)$ is the fractional
Sobolev space as in \eqref{eqn0iss}. If $\Omega$ supports the $D$-adapted Hardy
inequality \eqref{eqn-HI}, then the extension operator $\mathcal{E}_D$ defined in
Definition \ref{def-extension1} is bounded from $\mathcal{W}^{s,p}(\Omega)$ to $W^{s,p}(\mathbb{R}^n)$.
\end{proposition}

\begin{proof}
We first show that $\mathcal{E}_D$ is bounded from $L^p(\Omega)$ to $L^p(\mathbb{R}^n)$.
To this end, for any $f\in L^p(\Omega)$, by Definition \ref{def-extension1},
Lemmas \ref{W} and \ref{lem3.4.4}, and \eqref{eqn-EQj}, we find that
\begin{align*}
\left\|\mathcal{E}_Df\right\|_{L^p(\mathbb{R}^n)}^p
&\lesssim \|f\|_{L^p(\Omega)}^p+\sum_{Q_j\in\mathscr{W}_{\rm e}} \int_{c_nQ_j}\left|\psi_j(x)\right|^p\left|\left(E_{Q_j^*}
f\right)_{Q_j^*}\right|^pdx\\ \notag
&\lesssim \|f\|_{L^p(\Omega)}^p+\sum_{Q_j\in\mathscr{W}_{\rm e}} \ell(Q_j)^n\fint_{Q_j^*}\left|E_{Q_j^*}f(x)\right|^pdx\\ \notag
&\lesssim \|f\|_{L^p(\Omega)}^p+\sum_{Q_j\in \mathscr{W}_{\rm e}} \int_{Q_j^*\cap\Omega}|f(\xi)|^p\,d\xi
\lesssim \|f\|_{L^p(\Omega)}^p,
\end{align*}
which implies the boundedness of $\mathcal{E}_D$ from $L^p(\Omega)$ to $L^p(\mathbb{R}^n)$. Thus,
using \eqref{eqn-normfs} and the Hardy inequality \eqref{eqn-HI}, we conclude that,
to finish the proof of the present proposition,  it suffices to verifies that,
for any $f\in\mathcal W^{s,p}(\Omega)$,
\begin{align}\label{main-2}
\left\|\mathcal{E}_Df\right\|_{\dot{W}^{s,p}
(\mathbb{R}^n)}:=&\,\left[\int_{\mathbb{R}^n}\int_{\mathbb{R}^n}
\frac{\left|\mathcal{E}_Df(x)-\mathcal{E}_Df(y)\right|^p}
{|x-y|^{n+sp}}\,dydx\right]^{\frac 1p}\notag\\
\lesssim&\,\|f\|_{L^p(\Omega)}+\|f\|_{\dot{\mathcal{W}}^{s,p}(\Omega)}
+\left\{\int_\Omega\frac{|f(x)|^p}
{[{{\mathop\mathrm{dist\,}}}(x,D)]^{sp}}\,dx\right\}^\frac{1}{p},
\end{align}
where $\|\cdot\|_{\dot W^{s,p}(\mathbb{R}^n)}$
and $\|\cdot\|_{\dot{\mathcal W}^{s,p}(\Omega)}$ are as,
respectively, in \eqref{eqn-normfs} and \eqref{eqn-normfsdm}.

To prove \eqref{main-2},  let $f\in\mathcal{W}^{s,p}(\Omega)$. From \eqref{ext}
and \eqref{main-2}, we first infer that
\begin{align*}
\left\|\mathcal{E}_Df\right\|_{\dot W^{s,p}(\mathbb{R}^n)}^p
&=\int_{\mathbb{R}^n}\int_{\mathbb{R}^n}
\frac{\left|\mathcal{E}_Df(x)-\mathcal{E}_Df(y)\right|^p}{|x-y|^{n+sp}}\,dydx\\ \notag
&\lesssim \int_\Omega\int_\Omega\frac{|f(x)-f(y)|^p}{|x-y|^{n+sp}}\,dydx\\ \notag
&\quad+2\int_\Omega\int_{\Omega^{\complement}}\left|f(x)-\mathop{\sum}
\limits_{Q_j\in\mathscr{W}_{\rm e}}\left(E_{Q_j^*}f\right)_{Q_j^*}
\psi_j(y)\right|^p\frac{1}{|x-y|^{n+sp}}\,dydx\\ \notag
&\quad+\int_{\Omega^{\complement}}\int_{\Omega^{\complement}}
{\left|\mathop{\sum}\limits_{P_j\in\mathscr{W}_{\rm e}}
\left(E_{P_j^*}f\right)_{P_j^*}\psi_j(x)-\mathop{\sum}
\limits_{S_k\in\mathscr{W}_{\rm e}}
\left(E_{S_k^*}f\right)_{S_k^*}\psi_k(y)\right|^p}\frac{1}{|x-y|^{n+sp}}\,dydx.
\end{align*}
Notice that, by \eqref{eqn-normfsdm}, we have
\begin{align*}
\int_\Omega\int_\Omega\frac{|f(x)-f(y)|^p}{|x-y|^{n+sp}}\,dydx
=\|f\|_{\dot{\mathcal W}^{s,p}(\Omega)}^p.
\end{align*}
Thus, to complete the proof of the present proposition,
it suffices to estimate the following two types of integrals:
\begin{align}\label{eqn-integralI}
{\rm{I}}(f):=\int_\Omega\int_{\Omega^{\complement}}
\left|f(x)-\mathop{\sum}\limits_{Q_j\in\mathscr{W}_{\rm e}}
\left(E_{Q_j^*}f\right)_{Q_j^*}\psi_j(y)\right|^p\frac{1}{|x-y|^{n+sp}}\,dydx
\end{align}
and
\begin{align}\label{eqn-integralII}
{\rm{II}}(f):=\int_{\Omega^{\complement}}
\int_{\Omega^{\complement}}\left|\mathop{\sum}\limits_{P_j\in \mathscr{W}_{\rm e}}\left(E_{P_j^*}f\right)_{P_j^*}\psi_j(x)-\mathop{\sum}
\limits_{S_k\in\mathscr{W}_{\rm e}}\left(E_{S_k^*}f\right)_{S_k^*}
\psi_k(y)\right|^p\frac{1}{|x-y|^{n+sp}}\,dydx.
\end{align}
The desired estimations of ${\rm{I}}(f)$ and ${\rm{II}}(f)$ are given
respectively in the following Subsections \ref{sei} and \ref{seii}.
This finishes the proof of Proposition \ref{thm1x}.
\end{proof}

The following proposition shows that
$\mathcal{E}_D$ keeps the vanishing trace condition at $D$ for functions
in $\mathcal{W}^{s,p}(\Omega)$.

\begin{proposition}\label{lem-extensionOP}
Let $s\in(0,1)$, $p\in[1,\infty)$, and $\Omega\subset\mathbb{R}^n$ be an $(\epsilon,\delta,D)$-domain with $\epsilon\in(0,1]$,
$\delta\in(0,\infty]$, and $D\subset \partial\Omega$ being closed.
Let $\mathcal{E}_D$ be the extension operator
defined in Definition \ref{def-extension1}. Then, for any $f\in\mathcal{W}^{s,p}(\Omega)$
satisfying ${{\mathop\mathrm{dist\,}}}({\rm supp}\, f,D)>0$, it holds that
\begin{align}\label{eqn-lem1}
\text{\rm dist}\left({\rm supp}\, \mathcal{E}_Df,D\right)>0.
\end{align}
\end{proposition}

\begin{proof}
Let $f\in\mathcal{W}^{s,p}(\Omega)$ satisfy ${{\mathop\mathrm{dist\,}}}({\rm supp}\, f,D)>0$ and
\begin{align}\label{ND}
N(D):=\left\{x\in\mathbb{R}^n:{{\mathop\mathrm{\,dist\,}}}(x,D)
<\frac{1}{C_1}{{\mathop\mathrm{\,dist\,}}}({\rm supp}\, f,D) \right\}
\end{align}
for some constant $C_1\in (0,\infty)$ to be determined later. To prove \eqref{eqn-lem1},
it suffices to show that, for any $x\in N(D)$, it holds that
\begin{align}\label{eqn-lem2}
\mathcal{E}_Df(x)=0.
\end{align}
Indeed, for any $x\in N(D)$, we may assume that there exists $Q\in\mathscr{W}_{\rm e}$
satisfying $x\in c_nQ$ in view of Definition \ref{def-extension1}. Thus,
we have, for any $x\in N(D)$,
\begin{align*}
\mathcal{E}_Df(x)=\sum_{P\in \mathscr{W}_{\rm e},\overline{P}\cap \overline{Q}\neq\emptyset}\left(E_{P^*}f\right)_{P^*}\psi_{P}(x).
\end{align*}
We now only need to show that $P^*\cap{\rm supp} f=\emptyset$ holds for any
$P\in\mathscr{W}_{\rm e}$ satisfying $\overline{P}\cap\overline{Q}\neq \emptyset$. To this end,
we deduce from Lemmas \ref{W} and \ref{lem3.4.4} that, for any $y\in P^*$,
\begin{align*}
{{\mathop\mathrm{dist\,}}}(y,D)&\le{{\mathop\mathrm{dist\,}}}(P^*,D)+{\mathop\mathrm{diam\,}}P^*\\
&\le{\mathop\mathrm{diam\,}}P^*+{{\mathop\mathrm{dist\,}}}(P^*,P)+{\mathop\mathrm{diam\,}}P
+{{\mathop\mathrm{dist\,}}}(P,D)+{\mathop\mathrm{diam\,}}P^*\\
&\lesssim {\mathop\mathrm{diam\,}}P+{\mathop\mathrm{diam\,}}Q+{{\mathop\mathrm{dist\,}}}(Q,D)
\lesssim {{\mathop\mathrm{dist\,}}}(Q,D)\lesssim {{\mathop\mathrm{dist\,}}}(x,D).
\end{align*}
Hence, $P^*\cap{\rm supp} f=\emptyset$ holds for any such $P$ by letting the constant
$C_1$ in \eqref{ND} be large enough. This implies \eqref{eqn-lem2} and hence
finishes the proof of Proposition \ref{lem-extensionOP}.
\end{proof}

\begin{remark}\label{rem3.5}
In the proof of Theorem \ref{thm-equiv}, Propositions \ref{thm1x} and \ref{lem-extensionOP} are enough.
Thus, we can use Theorem \ref{thm-equiv} in the proof of
Theorem \ref{thm1}, once Propositions \ref{thm1x} and \ref{lem-extensionOP}
are obtained.
\end{remark}

\begin{proof}[Proof of Theorem \ref{thm1}]
By Proposition \ref{thm1x}, we only need to show that, for any $f\in \mathcal W^{s,p}(\Omega)$,
$\mathcal{E}_D(f)$ can be approximated by functions in $C_D^\infty(\mathbb{R}^n)$ with respect to the norm
$\|\cdot\|_{W^{s,p}(\mathbb{R}^n)}$. Indeed, for any $f\in \mathcal W^{s,p}(\Omega)$, by
Remark \ref{rem3.5}, we find that Theorem \ref{thm-equiv} holds.
Thus, $\mathcal W^{s,p}(\Omega)=\mathring{W}_D^{s,p}(\Omega)$ and hence there exist a sequence
$\{f_k\}_k$ in $C_D^\infty(\Omega)$ such that $\lim_{k\to\infty}f_k=f$
according to $\|\cdot\|_{\mathcal{W}^{s,p}(\Omega)}$. Moreover, from Definition \ref{def-extension1}
and Proposition \ref{lem-extensionOP}, we infer that, for every $k$, $\mathcal{E}_Df_k\in C_D^\infty(\mathbb{R}^n)$.
This, together with Proposition \ref{thm1x} again, implies
\begin{align*}
\lim_{k\to\infty}\left\|\mathcal{E}_Df_k
-\mathcal{E}_Df\right\|_{W^{s,p}(\mathbb{R}^n)}
\lesssim \lim_{k\to\infty}\left\|f_k
-f\right\|_{\mathcal{W}^{s,p}(\Omega)}=0,
\end{align*}
which completes the proof of Theorem \ref{thm1}.
\end{proof}

\subsection{Estimation of I-type integral}\label{sei}

This subsection is devoted to the estimation of the integral ${\rm I}(f)$ in \eqref{eqn-integralI}.
Our main result is the following proposition.

\begin{proposition}\label{prop411}
Let $s\in(0,1)$, $p\in[1,\infty)$, $\epsilon\in(0,1]$,
$\delta\in(0,\infty]$, and
$\Omega\subset\mathbb{R}^n$ be an $(\epsilon,\delta,D)$-domain with
$D\subset \partial\Omega$ being closed. Then,
for any $f\in\mathcal{W}^{s,p}(\Omega)$, it holds that
\begin{align*}
{\rm I}(f)\lesssim \|f\|_{L^p(\Omega)}^p+\|f\|_{\dot{\mathcal{W}}^{s,p}
(\Omega)}^p+\int_\Omega\frac{|f(x)|^p}{[{{\mathop\mathrm{dist\,}}}(x,D)]^{sp}}\,dx,
\end{align*}
where the implicit positive constant is independent of $f$.
\end{proposition}

To this end, for any $f\in L_{\rm loc}^1(\Omega)$, $x\in\Omega$, and $y\in\Omega^{\complement}$, let
\begin{align*}
F(x,y):=f(x)-\sum_{Q_j\in \mathscr{W}_{\rm e}}\left(E_{Q_j^*}f\right)_{Q_j^*}\psi_j(y).
\end{align*}
It is easy to find
\begin{align}\label{eqn-Fxy}
F(x,y)&=f(x)\left[1-\sum_{Q_j\in\mathscr{W}_{\rm e}}\psi_j(y)\right]
+\sum_{Q_j\in\mathscr{W}_{\rm e}}\psi_j(y)
\left[f(x)-\left(E_{Q_j^*}f\right)_{Q_j^*}\right]\notag\\
&=:F_1(x,y)+F_2(x,y).
\end{align}
Now, let
\begin{align}\label{eqnI1x}
{\rm I}_1(f):=\int_\Omega \int_{\Omega^{\complement}}\frac{\left|F_1(x,y)\right|^p}{|x-y|^{n+sp}}\,dydx
\end{align}
and
\begin{align}\label{eqnI2}
{\rm I}_2(f):=\int_\Omega\int_{\Omega^{\complement}}\frac{\left|F_2(x,y)\right|^p}{|x-y|^{n+sp}}\,dydx.
\end{align}

\begin{lemma}\label{lem-type-J1}
If $\Omega\subset\mathbb{R}^n$ is an $(\epsilon,\delta,D)$-domain with
$D\subset \partial\Omega$ being closed, then, for any $s\in(0,1)$, $p\in(1,\infty)$,
and $f\in L_{\rm loc}^1(\Omega)$, it holds that
\begin{align*}
{\rm I}_1(f)\lesssim \|f\|_{L^p(\Omega)}^p+\int_\Omega\dfrac{|f(x)|^p}{[{{\mathop\mathrm{dist\,}}}(x,D)]^{sp}}\,dx,
\end{align*}
where ${\rm I}_1(f)$ is as in \eqref{eqnI1x} and the implicit positive constant is independent of $f$.
\end{lemma}

\begin{proof}
Let $\Omega_{\rm e}'$ be as in \eqref{eqn-boze'}. By the definition of $\psi_j$, we have,
for any $y\in\Omega_{\rm e}'$,
\begin{align*}
\sum_{Q_j\in \mathscr{W}_{\rm e}} \psi_j(y)=1.
\end{align*}
Hence, from \eqref{eqnI1x} and Lemma \ref{dist-1}, it follows
that there exists a positive constant $C_{\rm e}$ such that
\begin{align}\label{b1-10}
\notag{\rm I_1}(f)
&\sim \int_\Omega\int_{\Omega^{\complement}\setminus\Omega_{\rm e}'} \frac{|f(x)|^p\left|1-\sum_{Q_j\in\mathscr{W}_{\rm e}}\psi_j(y)\right|^p}
{|x-y|^{n+sp}}\,dydx\\ \notag
&\lesssim \int_\Omega\int_{\left\{y\in\mathbb{R}^n:\,|x-y|\ge C_{\rm e}\right\}} \frac{|f(x)|^p}{|x-y|^{n+sp}}\,dydx
+\int_\Omega\int_{\left\{y\in\mathbb{R}^n:\,|x-y|\ge C_{\rm e}{{\mathop\mathrm{\,dist\,}}}(x,D)\right\}} \frac{|f(x)|^p}{|x-y|^{n+sp}}\,dydx\\ \notag
&\lesssim \int_\Omega |f(x)|^p \int_{C_{\rm e}}^\infty r^{n-1} \frac{1}{r^{n+sp}}\,drdx
+\int_\Omega |f(x)|^p \int_{C_{\rm e}{{\mathop\mathrm{\,dist\,}}}(x,D)}^\infty r^{n-1}\frac{1}{r^{n+sp}}\,drdx\\
&\lesssim \|f\|_{L^p(\Omega)}^p+\int_\Omega \frac{|f(x)|^p}{[{{\mathop\mathrm{dist\,}}}(x,D)]^{sp}}\,dx,
\end{align}
which completes the proof of Lemma \ref{lem-type-J1}.
\end{proof}

\begin{lemma}\label{lem-type-J2}
If $\Omega\subset\mathbb{R}^n$ is an
$(\epsilon,\delta,D)$-domain with
$D\subset \partial\Omega$ being closed,
then, for any $s\in(0,1)$, $p\in(1,\infty)$,
and $f\in L_{\rm loc}^1(\Omega)$, it holds that
\begin{align*}
{\rm I}_2(f)\lesssim \|f\|_{\dot{\mathcal{W}}^{s,p}
(\Omega)}^p+\int_\Omega\dfrac{|f(x)|^p}
{[{{\mathop\mathrm{dist\,}}}(x,D)]^{sp}}\,dx,
\end{align*}
where ${\rm I}_2(f)$ is as in \eqref{eqnI2} and the implicit
positive constant is independent of $f$.
\end{lemma}

\begin{proof}
Using the fact $\text{supp\,}\psi_j\subset c_nQ_j$, we obtain
\begin{align*}
{\rm I_2}(f)&=\int_\Omega \sum_{Q_j\in\mathscr{W}_{\rm e}} \left|f(x)-\left(E_{Q_j^*}f\right)_{Q_j^*}\right|^p
\int_{\Omega^{\complement}} \frac{\left|\psi_j(y)\right|^p}
{|x-y|^{n+sp}}\,dydx\\
&\lesssim \int_\Omega \left[\sum_{Q_j\in\mathscr{W}_{\rm e}} \fint_{Q_j^*}\left|f(x)-E_{Q_j^*}f(\xi)\right|^p\,d\xi
\int_{c_nQ_j}\frac{\left|\psi_j(y)\right|^p}{|x-y|^{n+sp}}\, dy\right]dx.
\end{align*}
By Lemmas \ref{W}(ii) and  \ref{lem3.4.4},
we conclude that, for any $x\in \Omega$, $y\in c_n Q_j$, and
$\xi \in Q_j^*$, it holds that
\begin{align*}
|y-\xi|&\le c_n{\mathop\mathrm{diam\,}}Q_j+{{\mathop\mathrm{dist\,}}}(c_nQ_j,Q_j^*)+{\mathop\mathrm{diam\,}}Q_j^*\\
&\lesssim {\mathop\mathrm{diam\,}}Q_j\lesssim {{\mathop\mathrm{dist\,}}}(Q_j,\partial\Omega)\lesssim |x-y|,
\end{align*}
and hence $|x-\xi|\le |x-y|+|y-\xi|\lesssim |x-y|$.
This, combined with Lemma \ref{dist-2}, implies
\begin{align*}
{\rm I_2}(f)&\lesssim \int_\Omega\sum_{Q_j\in\mathscr{W}_{\rm e}}\int_{Q_j^*} \frac{\left|f(x)-E_{Q_j^*}f(\xi)\right|^p}
{|x-\xi|^{n+sp}}\,d\xi dx\\ \notag
&\lesssim \int_\Omega\sum_{Q_j\in\mathscr{W}_{\rm e}}\int_{Q_j^*\cap\Omega}
\frac{|f(x)-f(\xi)|^p}{|x-\xi|^{n+sp}}\,d\xi dx
+\int_\Omega\sum_{Q_j\in\mathscr{W}_{\rm e}}\int_{Q_j^*\setminus\Omega} \frac{|f(x)|^p}{|x - \xi|^{n+sp}}\,d\xi dx\\ \notag
&\lesssim\|f\|_{\dot{\mathcal{W}}^{s,p}
(\Omega)}^p+\int_\Omega |f(x)|^p \int_{\frac{1}{2} {{\mathop\mathrm{\,dist\,}}}(x,D)}^\infty
r^{n-1} \frac{1}{r^{n+sp}}\,drdx\\ \notag
&\lesssim \|f\|_{\dot{\mathcal{W}}^{s,p}
(\Omega)}^p+\int_\Omega \frac{|f(x)|^p}{[{{\mathop\mathrm{dist\,}}}(x,D)]^{sp}}\,dx,
\end{align*}
which completes the proof of Lemma \ref{lem-type-J2}.
\end{proof}

\begin{proof}[Proof of Proposition \ref{prop411}]
The present proposition follows immediately from \eqref{eqn-Fxy} and Lemmas
\ref{lem-type-J1} and \ref{lem-type-J2}.
\end{proof}

\subsection{Estimation of  II-type integral}\label{seii}

This subsection is devoted to the estimation of the integral ${\rm II}(f)$ in \eqref{eqn-integralII}.
Our main result is the following proposition.

\begin{proposition}
If $\Omega\subset\mathbb{R}^n$ is an $(\epsilon,\delta,D)$-domain with
$D\subset \partial\Omega$ being closed, then, for any $s\in(0,1)$, $p\in(1,\infty)$,
and $f\in \mathcal{W}^{s,p}(\Omega)$, it holds that
\begin{align*}
{\rm II}(f)\lesssim \|f\|_{L^p(\Omega)}^p+\|f\|_{\dot{\mathcal{W}}^{s,p}
(\Omega)}^p
+\int_\Omega\frac{|f(x)|^p}{[{{\mathop\mathrm{dist\,}}}(x,D)]^{sp}}\,dx,
\end{align*}
where ${\rm II}(f)$ is as in \eqref{eqn-integralII} and the implicit positive constant is independent of $f$.
\end{proposition}

\begin{proof}
For any $x$, $y\in\Omega^{\complement}$, let
\begin{align}\label{eqn-Gxy}
G(x,y):=\mathop{\sum}\limits_{P_j\in \mathscr{W}_{\rm e}} \left(E_{P_j^*}f\right)_{P_j^*}\psi_j(x)-\mathop{\sum}
\limits_{S_k\in\mathscr{W}_{\rm e}}
\left(E_{S_k^*}f\right)_{S_k^*}\psi_k(y).
\end{align}
We then have
\begin{align*}
{\rm II}(f)&=\int_{\Omega^{\complement}} \int_{\Omega^{\complement}} \frac{|G(x,y)|}{|x-y|^{n+sp}}\,dydx\\ \notag
&=\sum_{Q\in \mathscr{W}(\overline{\Omega})}\sum_{R\in \mathscr{W}(\overline{\Omega})} \int_Q\int_R
\frac{|G(x,y)|}{|x-y|^{n+sp}}\,dydx\\ \notag
&=\sum_{Q\in \mathscr{W}(\overline{\Omega})}\sum_{R\in \mathscr{W}(\overline{\Omega})} \int_Q\int_{R\cap B\left(x,\frac{\ell(Q)}{10}\right)}
\frac{|G(x,y)|}{|x-y|^{n+sp}}\,dydx
+\sum_{Q\in \mathscr{W}(\overline{\Omega})}\sum_{R\in \mathscr{W}(\overline{\Omega})} \int_Q\int_{R\setminus B\left(x,\frac{\ell(Q)}{10}\right)}\ldots\\ \notag
&=:{\rm II}_{1}(f)+{\rm II}_{2}(f).
\end{align*}

For ${\rm II}_{1}(f)$, we further write
\begin{align}\label{eqn-decom-II-1}
{\rm II}_1(f)&=\sum_{Q\in \mathscr{W}(\overline{\Omega})\setminus \mathscr{W}'_{\rm e}}
\sum_{R\in \mathscr{W}(\overline{\Omega})}\int_{Q} \int_{R\cap B\left(x,\frac{\ell(Q)}{10}\right)} \frac{|G(x,y)|}{|x-y|^{n+sp}}\,dydx
+\sum_{Q\in \mathscr{W}'_{\rm e}}\sum_{R\in \mathscr{W}
(\overline{\Omega})} \ldots \notag\\
&=:{\rm II}_{1,1}(f)+{\rm II}_{1,2}(f).
\end{align}

For ${\rm II}_{2}(f)$, we further write
\begin{align}\label{eqn-decom-II-2}
{\rm II}_2(f)&=\sum_{Q\in \mathscr{W}_{\rm e}'}\sum_{R\in \mathscr{W}_{\rm e}'} \int_Q\int_{R\setminus B\left(x,\frac{\ell(Q)}{10}\right)}
\ldots\, dydx+\sum_{Q\in \mathscr{W}_{\rm e}'}\sum_{R\in \mathscr{W}(\overline{\Omega})\setminus \mathscr{W}_{\rm e}}\ldots
+\sum_{Q\in \mathscr{W}_{\rm e}'}\sum_{R\in \mathscr{W}_{\rm e}\setminus \mathscr{W}'_{\rm e}}\ldots \notag \\ \notag
&\quad+\sum_{Q\in \mathscr{W}(\overline{\Omega})\setminus\mathscr{W}_{\rm e}'}\sum_{R\in \mathscr{W}_{\rm e}'}\ldots
+\sum_{Q\in \mathscr{W}(\overline{\Omega})\setminus\mathscr{W}_{\rm e}'}
\sum_{R\in \mathscr{W}(\overline{\Omega})\setminus \mathscr{W}_{\rm e}}
\ldots+\sum_{Q\in \mathscr{W}(\overline{\Omega})\setminus\mathscr{W}_{\rm e}'} \sum_{R\in \mathscr{W}_{\rm e}\setminus \mathscr{W}'_{\rm e}}\ldots\\
&=:\sum_{j=1}^6{\rm II}_{2,j}(f).
\end{align}
Thus, to finish the proof of the present proposition,
it suffices to obtain the desired estimates
for ${\rm II}_{1,1}(f)$, ${\rm II}_{1,2}(f)$, and $\{{\rm II}_{2,j}(f)\}_{j=1}^6$, which are given in the next several lemmas.
\end{proof}

\begin{lemma}\label{lem-Gxy}
Let $G$ be as in \eqref{eqn-Gxy}. Then the following two assertions hold:
\begin{itemize}
\item [{\rm (i)}]
For any $x\in Q\in\mathscr{W}(\overline{\Omega})$ and $y\in R\in\mathscr{W}(\overline{\Omega})$,
\begin{align*}
G(x,y)&=\sum_{\genfrac{}{}{0pt}{}{P_j\in \mathscr{W}_{\rm e}}{
\overline P_j\cap\overline Q\neq\emptyset}}
\sum_{\genfrac{}{}{0pt}{}{S_k\in \mathscr{W}_{\rm e}}{
\overline{S}_k\cap\overline{R}\neq\emptyset}}\left[\left(E_{P_j^*}
f\right)_{P_j^*}-\left(E_{S_k^*}f\right)_{S_k^*}\right]\psi_j(x)\ \psi_k(y)\\
&\quad+\sum_{\genfrac{}{}{0pt}{}{P_j\in\mathscr{W}_{\rm e}}{ \overline{P}_j\cap\overline{Q}\neq\emptyset}}
\left(E_{P_j^*}f\right)_{P_j^*}\psi_j(x)\
\left[1-\sum_{\genfrac{}{}{0pt}{}{S_k\in \mathscr{W}_{\rm e}}{
\overline{S}_k\cap \overline{R}\neq\emptyset}}\psi_k(y)\right]\\
&\quad-\sum_{\genfrac{}{}{0pt}{}{S_k\in \mathscr{W}_{\rm e}}{\overline{S}_k\cap \overline{R}\neq\emptyset}}\left(E_{S_k^*}f\right)_{S_k^*}\psi_k(y)\
\left[1-\sum_{\genfrac{}{}{0pt}{}{P_j\in\mathscr{W}_{\rm e}}{\overline{P}_j\cap \overline{Q}\neq\emptyset}}\psi_j(x)\right]\\
&=:G_1(x,y)+G_2(x,y)+G_3(x,y).
\end{align*}

\item [{\rm (ii)}] If $x$ and $y$ are as in (i) and, in addition, $|x-y|<\ell(Q)$, then
\begin{align*}
G(x,y)&=\sum_{\genfrac{}{}{0pt}{}{P_j\in \mathscr{W}_{\rm e}}{\overline P_j \cap
\overline Q\ne\emptyset}}
\left(E_{P_j^*}f\right)_{P_j^*}\left[\psi_j(x)-\psi_j(y)\right].
\end{align*}
\end{itemize}
\end{lemma}

\begin{proof}
(i) follows from a straightforward calculation. We only prove (ii).
For any $x\in Q\in \mathscr{W}(\overline{\Omega})$,
$y\in B(x,\ell(Q)/10)$, and $P_j\in\mathscr{W}_{\rm e}$,
if $P_j$ and $Q$ do not touch, by the fact
$\text{supp\,}\psi_j\subset c_nP_j$ and Lemma \ref{W},
we find that $\psi_j(x)=\psi_j(y)=0,$
which, together with \eqref{eqn-Gxy}, implies that
\begin{align*}
G(x,y)&=\mathop{\sum}\limits_{\genfrac{}{}{0pt}{}{P_j\in
\mathscr{W}_{\rm e}}{P_j\approx Q}}\left(E_{P_j^*}f\right)_{P_j^*}\psi_j(x)-\mathop{\sum}
\limits_{\genfrac{}{}{0pt}{}{S_k\in \mathscr{W}_{\rm e}}{\overline{S}_k\cap \overline{Q}\neq\emptyset}}\left(E_{S_k^*}f\right)_{S_k^*}\psi_k(y)\\
&=\sum_{\genfrac{}{}{0pt}{}{P_j\in \mathscr{W}_{\rm e}}{\overline{P}_j\cap \overline{Q}\neq\emptyset}}\left(E_{P_j^*}f\right)_{P_j^*}
\left[\psi_j(x)-\psi_j(y)\right].
\end{align*}
This verifies (ii) and hence finishes the proof of Lemma \ref{lem-Gxy}.
\end{proof}

\begin{lemma}
If $\epsilon\in(0,1]$, $\delta\in(0,\infty]$, and
$\Omega\subset\mathbb{R}^n$ is an
$(\epsilon,\delta,D)$-domain with $D\subset \partial\Omega$ being closed, then, for any
$s\in(0,1)$, $p\in(1,\infty)$, and $f\in L_{\rm loc}^1(\Omega)$, it holds that
\begin{align*}
{\rm II}_{1,1}(f)\lesssim \|f\|_{L^{p}(\Omega)}^p+\int_\Omega\dfrac{|f(x)|^p}{[{{\mathop\mathrm{dist\,}}}(x,D)]^{sp}}\,dx
\end{align*}
with the implicit positive constant independent of $f$, where ${\rm II}_{1,1}(f)$
is as in \eqref{eqn-decom-II-1}.
\end{lemma}

\begin{proof}
Using Lemmas \ref{lem-Gxy} and \ref{lem4.6-1}, we obtain
\begin{align}\label{d3}
\notag{\rm II}_{1,1}(f)&=\sum_{Q\in \mathscr{W}(\overline{\Omega})\setminus\mathscr{W}_{\rm e}'}
\int_Q\int_{B\left(x,\frac{\ell(Q)}{10}\right)}
\left|\sum\limits_{\genfrac{}{}{0pt}{}{P_j\in \mathscr{W}_{\rm e}}{\overline{P}_j\cap\overline{Q}\neq\emptyset}}
\left(E_{P_j^*}f\right)_{P_j^*}\left(\psi_j(x)-
\psi_j(y)\right)\right|^p\frac{1}{|x-y|^{n+sp}}\,dydx\\ \notag
&\lesssim \sum_{Q\in\mathscr{W}(\overline{\Omega})\setminus\mathscr{W}_{\rm e}'}
\int_Q\sum_{\genfrac{}{}{0pt}{}{P_j\in\mathscr{W}_{\rm e}}
{\overline{P}_j\cap\overline{Q}\neq\emptyset}}
\left|\left(E_{P_j^*}f\right)_{P_j^*}\right|^p
\int_{B\left(x,\frac{\ell(Q)}{10}\right)}
\frac{\left\|\nabla
\psi_j\right\|_{L^\infty(\mathbb{R}^n)}^p|x-y|^p}
{|x-y|^{n+sp}}\,dydx\\ \notag
&\lesssim \sum_{Q\in \mathscr{W}(\overline{\Omega})\setminus\mathscr{W}_{\rm e}'}
\ell(Q)^n\sum_{\genfrac{}{}{0pt}{}{P_j\in\mathscr{W}_{\rm e}}{
\overline{P}_j\cap\overline{Q}\neq\emptyset}} \frac{\ell(Q)^{(1-s)p}}{\ell(P_j)^p}\fint_{P_j^*}
\left|E_{P_j^*}f(\xi)\right|^pd\xi\\ \notag
&\lesssim \sum_{\genfrac{}{}{0pt}{}{Q\in\mathscr{W}(\overline{\Omega})
\setminus\mathscr{W}_{\rm e}'}{\ell(Q)>A\delta/16}}\sum_{\genfrac{}{}{0pt}{}{P_j\in
\mathscr{W}_{\rm e}}{\overline P_j\cap\overline{Q}\neq\emptyset}} \ell(Q)^{-sp}\int_{P_j^*\cap\Omega}|f(\xi)|^p\,d\xi
+\sum_{\genfrac{}{}{0pt}{}{Q\in
\mathscr{W}(\overline{\Omega})\setminus\mathscr{W}_{\rm e}'}{\ell(Q)\le A\delta/16}}
\sum_{\genfrac{}{}{0pt}{}{P_j\in \mathscr{W}_{\rm e}}{\overline{P}_j
\cap\overline Q\neq\emptyset}}\ell(Q)^{-sp}\int_{P_j^*\cap\Omega}
|f(\xi)|^p\,d\xi\\ \notag
&\lesssim \sum_{P_j\in\mathscr{W}_{\rm e}}\int_{P_j^*\cap\Omega} |f(\xi)|^p\,d\xi+\sum_{P_j\in\mathscr{W}_{\rm e}}\int_{P_j^*\cap\Omega}
\frac{|f(\xi)|^p}
{[{{\mathop\mathrm{dist\,}}}(\xi,D)]^{sp}}\,d\xi\\
&\lesssim \|f\|_{L^p(\Omega)}^p+\int_\Omega\frac{|f(\xi)|^p}
{[{{\mathop\mathrm{dist\,}}}(\xi,D)]^{sp}}\,d\xi,
\end{align}
which is desired.
\end{proof}

\begin{lemma}\label{lem-estimateK1}
If $\epsilon\in(0,1]$, $\delta\in(0,\infty]$, and
$\Omega\subset\mathbb{R}^n$ is an
$(\epsilon,\delta,D)$-domain with $D\subset \partial\Omega$
being closed, then, for any $s\in(0,1)$, $p\in(1,\infty)$, and $f\in L_{\rm loc}^1(\Omega)$, it holds that
\begin{align*}
{\rm II}_{1,2}(f)\lesssim \|f\|_{L^p(\Omega)}^p+\|f\|_{\dot{\mathcal{W}}^{s,p}
(\Omega)}^p+\int_\Omega\dfrac{|f(x)|^p}{[{{\mathop\mathrm{dist\,}}}(x,D)]^{sp}}\,dx
\end{align*}
with the implicit positive constant independent of $f$, where ${\rm II}_{1,2}(f)$
is as in \eqref{eqn-decom-II-1}.
\end{lemma}

\begin{proof}
Let $\mathscr{W}_{\rm e}'$ be defined as in \eqref{We'}.
It is easy to find that, for any $x\in Q\in\mathscr{W}_{\rm e}'$,
\begin{align*}
\sum_{P_j\in\mathscr{W}_{\rm e},\overline{P}_j\cap\overline{Q}
\neq\emptyset}\psi_j(x)=1.
\end{align*}
Moreover, by Lemma \ref{W}, we obtain, for any
$P_j\in\mathscr{W}_{\rm e}$ satisfying
that $P_j$ does not touch with $Q$ but
$\overline{P}_j\cap\overline S_k\neq\emptyset$ for some $\overline{S}_k\cap\overline{Q}\neq\emptyset$, it holds that
\begin{align*}
\bigcup_{x\in Q}B\left(x,\frac{\ell(Q)}{10}\right)\cap c_nP_j=\emptyset,
\end{align*}
where $c_n:=1+\frac{1}{16\sqrt{n}}$. From this and \eqref{eqn-puy},
we deduce that, for any $y\in B(x,\ell(Q)/10)$
with $x\in Q$,
\begin{align*}
\sum_{{P_j\in\mathscr{W}_{\rm e}},
{\overline{P}_j\cap\overline{Q}\neq\emptyset}}\psi_j(y)=1,
\end{align*}
which, combined with Lemma \ref{lem-Gxy}(ii), implies
\begin{align*}
G(x,y)&=\mathop{\sum}\limits_{\genfrac{}{}{0pt}{}{P_j\in \mathscr{W}_{\rm e}}{\overline{P}_j\cap\overline{Q}\neq\emptyset}}
\left(E_{P_j^*}f\right)_{P_j^*}
\left[\psi_j(x)-\psi_j(y)\right]\\
&=\mathop{\sum}\limits_{\genfrac{}{}{0pt}{}{P_j\in \mathscr{W}_{\rm e}}{\overline P_j\cap\overline{Q}\neq\emptyset}}\left[\left(E_{P_j^*}f\right)_{P_j^*}
-\left(E_{Q^*}f\right)_{Q^*}\right]\left[\psi_j(x)-\psi_j(y)\right].
\end{align*}
Using this, together with \eqref{pro-psi-3} and the assumption $s\in(0,1)$, we obtain
\begin{align}\label{eqn-K2ff}
\notag{\rm II}_{1,2}(f)&=\sum_{Q\in \mathscr{W}_{\rm e}'}\int_Q\int_{B\left(x,\frac{\ell(Q)}{10}\right)}
\left|\mathop{\sum}
\limits_{\genfrac{}{}{0pt}{}{P_j\in \mathscr{W}_{\rm e}}{\overline{P}_j\cap\overline Q\neq\emptyset}}\left[\left(E_{P_j^*}f\right)_{P_j^*}-
\left(E_{Q^*}f\right)_{Q^*}\right]
\left(\psi_j(x)-\psi_j(y)\right)\right|^p\frac{1}{|x-y|^{n+sp}}\,dydx\\ \notag
&\lesssim \sum_{Q\in \mathscr{W}_{\rm e}'}\int_Q
\int_{B\left(x,\frac{\ell(Q)}{10}\right)}
\sum_{\genfrac{}{}{0pt}{}{P_j\in \mathscr{W}_{\rm e}}{
\overline{P}_j\cap\overline{Q}\neq\emptyset}}
\left|\left(E_{P_j^*}f\right)_{P_j^*}-\left(E_{Q^*}f\right)_{Q^*}\right|^p
\frac{\left\|\nabla\psi_j\right\|_{L^\infty}^p|x-y|^p}{|x-y|^{n+sp}}\,dydx\\ \notag
&\lesssim \sum_{Q\in\mathscr{W}_{\rm e}'}\int_Q
\sum_{\genfrac{}{}{0pt}{}{P_j\in\mathscr{W}_{\rm e}}{\overline{P}_j
\cap\overline{Q}\neq
\emptyset}}[\ell(P_j)]^{-p}\left|\left(E_{P_j^*}f\right)_{P_j^*}
-\left(E_{Q^*}f\right)_{Q^*}\right|^p\int_{B\left(x,\frac{\ell(Q)}{10}\right)}
\frac{|x-y|^p}{|x-y|^{n+sp}}\,dydx\\ \notag
&\lesssim \sum_{Q\in\mathscr{W}_{\rm e}'}
\sum_{\genfrac{}{}{0pt}{}{P_j\in\mathscr{W}_{\rm e}}{\overline{P}_j\cap\overline{Q}\neq\emptyset}}
[\ell(Q)]^{n-sp}\left|\left(E_{P_j^*}f\right)_{P_j^*}
-\left(E_{Q^*}f\right)_{Q^*}\right|^p\\
&\lesssim \sum_{Q\in \mathscr{W}_{\rm e}'}
\sum_{\genfrac{}{}{0pt}{}{P_j\in\mathscr{W}_{\rm e}}{\overline{P}_j\cap\overline{Q}\neq\emptyset}}
[\ell(Q)]^{n-sp}\fint_{Q^*}\fint_{P_j^*}
\left|E_{Q^*}f(\xi)-E_{P_j^*}f(\eta)\right|^p\,d\eta d\xi.
\end{align}
Now, for any $\xi\in Q^*$ and $\eta\in P_j^*$, Lemmas \ref{lem3.4.4} and \ref{W}(iv) provide
\begin{align*}
|\xi-\eta|\le {\mathop\mathrm{diam\,}}Q^*+{{\mathop\mathrm{dist\,}}}(Q^*,P_j^*)
+{\mathop\mathrm{diam\,}} P_j^*\lesssim \ell(Q).
\end{align*}
From this, \eqref{eqn-K2ff}, and \eqref{eqn-EQj}, we infer that
\begin{align*}
{\rm II}_{1,2}(f)&\lesssim \sum_{Q\in\mathscr{W}_{\rm e}'}
\sum_{\genfrac{}{}{0pt}{}{P_j\in\mathscr{W}_{\rm e}}{\overline{P}_j\cap
\overline{Q}\neq\emptyset}}\int_{Q^*}\int_{P_j^*}
\frac{\left|E_{Q^*}f(\xi)-E_{P_j^*}f(\eta)\right|^p}
{|\xi-\eta|^{n+sp}}\,d\eta d\xi\\
&\lesssim \sum_{Q\in\mathscr{W}_{\rm e}'}\int_{Q^*\cap\Omega}
\sum_{\genfrac{}{}{0pt}{}{P_j\in\mathscr{W}_{\rm e}}{\overline{P}_j\cap
\overline{Q}\neq\emptyset}}\int_{P_j^*\cap\Omega}
\frac{|f(\xi)-f(\eta)|^p}{|\xi-\eta|^{n+sp}}\,d\eta d\xi\\
&\quad+2\sum_{Q\in \mathscr{W}_{\rm e}'}\int_{Q^*\cap\Omega}
\sum_{\genfrac{}{}{0pt}{}{P_j\in \mathscr{W}_{\rm e}}{\overline{P}_j\cap\overline{Q}
\neq\emptyset}}\int_{P_j^*\setminus\Omega}\frac{|f(\xi)|^p}
{|\xi-\eta|^{n+sp}}\,d\eta d\xi.
\end{align*}
By Lemma \ref{W} again, the first term above can be controlled by
\begin{align*}
\int_{\Omega}\int_{\Omega}\frac{|f(\xi)-f(\eta)|^p}{|\xi-\eta|^{n+sp}}\,d\eta d\xi
\sim \|f\|_{\dot{\mathcal{W}}^{s,p}(\Omega)}^p,
\end{align*}
which is desired. On the other hand, following an argument similar to that
used in the estimation of \eqref{b1-10} with Lemma \ref{dist-1} therein
replaced by Lemma \ref{dist-2}, we obtain the second term is bounded by
\begin{align*}
\sum_{P_j\in \mathscr{W}_{\rm e}}\int_{\Omega}\int_{P_j^*\setminus\Omega}
\frac{|f(\xi)|^p}{|\xi-\eta|^{n+sp}}\,d\eta d\xi
\lesssim \|f\|_{L^p(\Omega)}^p+\int_\Omega\frac{|f(\xi)|^p}
{{{\mathop\mathrm{\,dist\,}}}(\xi,D)^{sp}}\,d\xi,
\end{align*}
which implies that
\begin{align*}
{\rm II}_{1,2}(f)\lesssim \|f\|_{L^p(\Omega)}^p+\|f\|_{\dot{\mathcal{W}}^{s,p}
(\Omega)}^p+\int_\Omega\frac{|f(x)|^p}{[{{\mathop\mathrm{dist\,}}}(x,D)]^{sp}}\,dx
\end{align*}
and hence completes the proof of Lemma \ref{lem-estimateK1}.
\end{proof}

\begin{lemma}\label{lem-estimateII3}
If $\epsilon\in(0,1]$, $\delta\in(0,\infty]$, and
$\Omega\subset\mathbb{R}^n$ is an
$(\epsilon,\delta,D)$-domain with $D\subset \partial\Omega$
being closed, then, for any $s\in(0,1)$, $p\in(1,\infty)$, and
$f\in L_{\rm loc}^1(\Omega)$, it holds that
\begin{align*}
{\rm II}_{2,1}(f)\lesssim \|f\|_{L^{p}(\Omega)}^p+\int_\Omega\dfrac{|f(x)|^p}{[{{\mathop\mathrm{dist\,}}}(x,D)]^{sp}}\,dx
\end{align*}
with the implicit positive constant independent of $f$, where ${\rm II}_{2,1}(f)$
is as in \eqref{eqn-decom-II-2}.
\end{lemma}

\begin{proof}
From Lemma \ref{lem4.5-1}, we deduce that
\begin{align}\label{eqn-k11}
\notag{\rm II}_{2,1}(f)&\lesssim \sum_{Q\in\mathscr{W}_{\rm e}'}\int_Q\int_{\Omega_{\rm e}'\setminus B\left(x,\frac{\ell(Q)}{10}\right)}
\sum_{\genfrac{}{}{0pt}{}{P_j\in \mathscr{W}_{\rm e}}{ \overline{P}_j\cap\overline{Q}\neq\emptyset}}\sum_{S_k\in\mathscr{W}_{\rm e}}
\left|\left(E_{P_j^*}f\right)_{P_j^*}
-\left(E_{S_k^*}f\right)_{S_k^*}\right|^p\frac{\left|\psi_j(x)\, \psi_k(y) \right|^p}{|x-y|^{n+sp}}\,dydx\\ \notag
&\lesssim \sum_{Q\in\mathscr{W}_{\rm e}'}\int_Q
\sum_{\genfrac{}{}{0pt}{}{P_j\in\mathscr{W}_{\rm e}}{\overline{P}_j\cap\overline{Q}\neq\emptyset}}
\sum_{S_k\in\mathscr{W}_{\rm e}}\left|
\left(E_{P_j^*}f\right)_{P_j^*}-\left(E_{S_k^*}
f\right)_{S_k^*}\right|^p\int_{c_nS_k\setminus B\left(x,\frac{\ell(Q)}
{10}\right)}\frac{|\psi_k(y)|^p}{[D(P_j^*,S_k^*)]^{n+sp}}\,dydx\\ \notag
&\lesssim \sum_{Q\in\mathscr{W}_{\rm e}'}\int_Q
\sum_{\genfrac{}{}{0pt}{}{P_j\in\mathscr{W}_{\rm e}}{\overline P_j\cap\overline{Q}\neq\emptyset}}
\sum_{S_k\in\mathscr{W}_{\rm e}}\frac{\bigg|
\left(E_{P_j^*}f\right)_{P_j^*}
-\left(E_{S_k^*}f\right)_{S_k^*}\bigg|^p}{[D(P_j^*, S_k^*)]^{n+sp}}
\int_{c_nS_k}|\psi_k(y)|^p\,dydx\\ \notag
&\lesssim \sum_{Q\in \mathscr{W}_{\rm e}'}
\sum_{\genfrac{}{}{0pt}{}{P_j\in\mathscr{W}_{\rm e}}{\overline{P}_j\cap\overline{Q}\neq\emptyset}}
\sum_{S_k\in \mathscr{W}_{\rm e}}
\frac{\ell(Q)^n\ell(S_k)^n}{[D(P_j^*,S_k^*)]^{n+sp}}
\fint_{P_j^*}\fint_{S_k^*}\left|E_{P_j^*}f(\xi)-
E_{S_k^*}f(\eta)\right|^p\,d\eta d\xi\\
&\lesssim \sum_{Q\in \mathscr{W}_{\rm e}'}
\sum_{\genfrac{}{}{0pt}{}{P_j\in \mathscr{W}_{\rm e}}{ \overline{P}_j\cap\overline{Q}\neq\emptyset}}
\sum_{S_k\in \mathscr{W}_{\rm e}}\int_{P_j^*}\int_{S_k^*}
\frac{\left|E_{P_j^*}f(\xi)-E_{S_k^*}f(\eta)
\right|^p}{|\xi-\eta|^{n+sp}}\,d\eta d\xi.
\end{align}
Moreover, using Lemma \ref{W}, we find that, for any $Q\in \mathscr{W}_{\rm e}'$,
$\#\{P_j\in \mathscr{W}_{\rm e}:\ P_j\approx Q\}\lesssim1$. This, combined with \eqref{eqn-EQj}, then implies
\begin{align*}
{\rm II}_{2,1}(f)&\lesssim \sum_{P_j\in \mathscr{W}_{\rm e}}
\sum_{S_k\in \mathscr{W}_{\rm e}}\int_{P_j^*\cap\Omega}
\int_{S_k^*\cap\Omega}
\frac{|f(\xi)-f(\eta)|^p}{|\xi-\eta|^{n+sp}}\,d\eta d\xi\\ \notag
&\quad+\sum_{P_j\in\mathscr{W}_{\rm e}}\sum_{S_k\in \mathscr{W}_{\rm e}}\int_{P_j^*\cap\Omega}
\int_{S_k^*\setminus\Omega}\frac{|f(\xi)|^p}
{|\xi-\eta|^{n+sp}}\,d\eta d\xi.
\end{align*}
From Lemma \ref{W} again, we infer that the first term above can be controlled by
\begin{align*}
\int_{\Omega}\int_{\Omega}\frac{|f(\xi)-f(\eta)|^p}{|\xi-\eta|^{n+sp}}\,d\eta d\xi\sim\|f\|_{\dot{\mathcal{W}}^{s,p}
(\Omega)}^p,
\end{align*}
which is desired. On the other hand, following an argument similar to that used in
the estimation of \eqref{d3} with Lemma \ref{dist-1} therein replaced by Lemma \ref{dist-2},
we obtain the second term is bounded by
\begin{align*}
\sum_{S_k\in \mathscr{W}_{\rm e}}\int_{\Omega}\int_{S_k^*\setminus\Omega}
\frac{|f(\xi)|^p}{|\xi-\eta|^{n+sp}}\,d\eta d\xi
\lesssim \|f\|_{L^p(\Omega)}^p+\int_\Omega \frac{|f(\xi)|^p}{[{{\mathop\mathrm{dist\,}}}(\xi,D)]^{sp}}\,d\xi.
\end{align*}
Hence, we finally have
\begin{align*}
{\rm II}_{2,1}(f)\lesssim \|f\|_{\dot{\mathcal{W}}^{s,p}(\Omega)}^p+\|f\|_{L^p(\Omega)}^p
+\int_\Omega\frac{|f(\xi)|^p}{[{{\mathop\mathrm{dist\,}}}(\xi,D)]^{sp}}\,d\xi,
\end{align*}
which completes the proof of Lemma \ref{lem-estimateII3}.
\end{proof}

\begin{lemma}\label{lem-estimateII4}
If $\epsilon\in(0,1]$, $\delta\in(0,\infty]$, and $\Omega\subset\mathbb{R}^n$ is an
$(\epsilon,\delta,D)$-domain with $D\subset \partial\Omega$ being closed, then, for any
$s\in(0,1)$, $p\in(1,\infty)$, and $f\in L_{\rm loc}^1(\Omega)$, it holds that
\begin{align*}
{\rm II}_{2,2}(f)\lesssim \|f\|_{L^{p}(\Omega)}^p
+\int_\Omega\dfrac{|f(x)|^p}
{[{{\mathop\mathrm{dist\,}}}(x,D)]^{sp}}\,dx
\end{align*}
with the implicit positive constant independent of $f$, where ${\rm II}_{2,2}(f)$
is as in \eqref{eqn-decom-II-2}.
\end{lemma}

\begin{proof}
By the fact $\bigcup_{S_k\in \mathscr{W}_{\rm e}}{\rm supp}\,\psi_k \subset\Omega_{\rm e}$
as in \eqref{eqn-boze}, we have, for any $x\in Q\in \mathscr{W}_{\rm e}'$ and
$y\in \Omega^{\complement}\setminus [\Omega_{\rm e}\cup B(x,\ell(Q)/10)]$,
it holds that $\psi_k(y)=0$ and
\begin{align*}
G(x,y)=\mathop{\sum}\limits_{\genfrac{}{}{0pt}{}{P_j\in \mathscr{W}_{\rm e}}{\overline P_j\cap \overline{Q}\neq\emptyset}}
\left(E_{P_j^*}f\right)_{P_j^*}\psi_j(x).
\end{align*}
Then, using \eqref{eqn-decom-II-2} and Lemma \ref{W}, we conclude that
\begin{align*}
{\rm II}_{2,2}(f)&\sim \sum_{Q\in \mathscr{W}_{\rm e}'}
\int_Q\int_{\Omega^{\complement}\setminus\Omega_{\rm e}}
\left|\mathop{\sum}
\limits_{P_j\in \mathscr{W}_{\rm e},\overline{P}_j\cap\overline{Q}
\neq\emptyset}\left(E_{P_j^*}f\right)_{P_j^*}\psi_j(x)\right|^p\frac{1}
{|x-y|^{n+sp}}\,dydx\\ \notag
&\lesssim \sum_{Q\in\mathscr{W}_{\rm e}'}\int_Q\sum_{\genfrac{}{}{0pt}{}{P_j\in \mathscr{W}_{\rm e}}{ \overline{P}_j\cap\overline{Q}\neq\emptyset}}
\fint_{P_j^*} \left|E_{P_j^*}f(\xi)\right|^pd\xi \int_{\Omega^{\complement}\setminus\Omega_{\rm e}}\frac{\left|\psi_j(x)\right|^p}{|x-y|^{n+sp}}\,dydx\\ \notag
&\lesssim \sum_{Q\in \mathscr{W}_{\rm e}'}\int_Q\sum_{\genfrac{}{}{0pt}{}{P_j\in \mathscr{W}_{\rm e}}{ \overline{P}_j\cap \overline{Q}\neq\emptyset}}
\ell(P_j)^{-n}\int_{P_j^*\cap\Omega}|f(\xi)|^p\int_{\Omega^{\complement}\setminus\Omega_{\rm e}}\frac{1}{|x-y|^{n+sp}}\,dyd\xi dx,
\end{align*}
where $\Omega_{\rm e}$ is as in \eqref{eqn-boze}.

Now, from Lemma \ref{dist-1xx}, we deduce that there exists a constant $C\in (0,\infty)$
such that, for any $x\in Q\in \mathscr{W}_{\rm e}'$, $y\in\Omega^{\complement}\setminus\Omega_{\rm e}$,
and $\xi\in P_j^*\cap\Omega$ satisfying $\overline{P}_j\cap\overline{Q}\neq\emptyset$,
it holds that either $|x-y|\ge C$ or $|x-y|\ge C{{\mathop\mathrm{\,dist\,}}}(\xi,D)$. Thus, we obtain
\begin{align*}
{\rm II}_{2,2}(f)&\lesssim \sum_{Q\in \mathscr{W}_{\rm e}'}
\int_Q\sum_{\genfrac{}{}{0pt}{}{P_j\in
\mathscr{W}_{\rm e}}{\overline P_j\cap
\overline{Q}\neq\emptyset}}\ell(P_j)^{-n}
\int_{P_j^*\cap\Omega}|f(\xi)|^p
\int_{\left\{y\in\mathbb{R}^n:\,|y-x|\ge C\right\}}\frac{1}
{|x-y|^{n+sp}}\,dyd\xi dx\\ \notag
&\quad+\sum_{Q\in \mathscr{W}_{\rm e}'}\int_Q
\sum_{\genfrac{}{}{0pt}{}{P_j\in \mathscr{W}_{\rm e}}{
\overline P_j\cap \overline{Q}\neq
\emptyset}}\ell(P_j)^{-n}\int_{P_j^*\cap\Omega}|
f(\xi)|^p\int_{\left\{y\in\mathbb{R}^n:\,|y-x|\ge C{{\mathop\mathrm{\,dist\,}}}(\xi,D)\right\}}\frac{1}
{|x-y|^{n+sp}}\,dyd\xi dx\\ \notag
&\lesssim \sum_{Q\in \mathscr{W}_{\rm e}'}\sum_{P_j\in \mathscr{W}_{\rm e}}
\int_{P_j^*\cap\Omega}|f(\xi)|^p\,d\xi
+\sum_{Q\in\mathscr{W}_{\rm e}'}\sum_{\genfrac{}{}{0pt}{}{P_j\in\mathscr{W}_{\rm e}}{ \overline{P}_j\cap\overline{Q}\ne\emptyset}}\int_{P_j^*\cap
\Omega}\frac{|f(\xi)|^p}{[{{\mathop\mathrm{dist\,}}}(\xi,D)]^{sp}}\,d\xi\\ \notag
&\lesssim \|f\|_{L^p(\Omega)}^p+\int_\Omega\frac{|f(\xi)|^p}
{[{{\mathop\mathrm{dist\,}}}(\xi,D)]^{sp}}\,d\xi,
\end{align*}
which is desired and hence completes
the proof of Lemma \ref{lem-estimateII4}.
\end{proof}

\begin{lemma}\label{lem-estimateII5}
If $\epsilon\in(0,1]$, $\delta\in(0,\infty]$, and $\Omega\subset\mathbb{R}^n$
is an $(\epsilon,\delta,D)$-domain with $D\subset \partial\Omega$ being closed,
then, for any $s\in(0,1)$, $p\in(1,\infty)$, and $f\in L_{\rm loc}^1(\Omega)$, it holds that
\begin{align*}
{\rm II}_{2,3}(f)\lesssim \|f\|_{L^{p}(\Omega)}^p+\int_\Omega\dfrac{|f(x)|^p}{[{{\mathop\mathrm{dist\,}}}(x,D)]^{sp}}\,dx
\end{align*}
with the implicit positive constant independent of $f$, where ${\rm II}_{2,3}(f)$ is as
in \eqref{eqn-decom-II-2}.
\end{lemma}

\begin{proof}
By the definition of $\psi_j$, for any $x\in Q\in \mathscr{W}_{\rm e}'$, it is obvious
that $\sum_{P_j\in\mathscr{W}_{\rm e}}\psi_j(x)=1$, which, together with Lemma \ref{lem-Gxy},
implies that, for any $y\in R\setminus B(x,{\ell(Q)}/{10})$ with $R\in {\mathscr{W}}_{\rm e}
\setminus\mathscr{W}_{\rm e}'$, it holds that
\begin{align*}
G(x,y)=\sum_{\genfrac{}{}{0pt}{}{P_j\in \mathscr{W}_{\rm e}}{\overline{P}_j\cap\overline{Q}\neq\emptyset}}
\sum_{S_k\in \mathscr{W}_{\rm e}}\left[\left(E_{P_j^*}f\right)_{P_j^*}
-\left(E_{S_k^*}f\right)_{S_k^*}\right]\psi_j(x)\,
\psi_k(y)+\sum_{P_j\in \mathscr{W}_{\rm e}}\left(E_{P_j^*}
f\right)_{P_j^*}\psi_j(x)
\left[1-\sum_{S_k\in
\mathscr{W}_{\rm e}}\psi_k(y)\right],
\end{align*}
and consequently
\begin{align*}
{\rm II}_{2,3}(f)&\lesssim \sum_{Q\in \mathscr{W}_{\rm e}'}\int_Q\int_{\Omega_{\rm e}\setminus \left[\Omega_{\rm e}'\cup
B\left(x,\frac{\ell(Q)}{10}\right)\right]}\sum_{\genfrac{}{}{0pt}{}{P_j\in \mathscr{W}_{\rm e}}{ \overline{P}_j\cap\overline{Q}\neq
\emptyset}}\sum_{S_k\in\mathscr{W}_{\rm e}}\left|
\left(E_{P_j^*}f\right)_{P_j^*}
-\left(E_{S_k^*}f\right)_{S_k^*} \right|^p
\frac{\left|\psi_j(x)\, \psi_k(y)\right|^p}{|x-y|^{n+sp}}\,dydx \\
&\quad+\sum_{Q\in \mathscr{W}_{\rm e}'}\int_Q
\int_{\Omega_{\rm e}\setminus \left[\Omega_{\rm e}'\cup B\left(x,\frac{\ell(Q)}{10}\right)\right]}
\left|\mathop{\sum}\limits_{P_j\in \mathscr{W}_{\rm e},\overline{P}_j\cap\overline{Q}\neq\emptyset}
\left(E_{P_j^*}f\right)_{P_j^*}
\psi_j(x)\right|^p\frac{1}{|x-y|^{n+sp}}\,dydx\\
&=:{\rm II}_{2,3,1}(f)+{\rm II}_{2,3,2}(f),
\end{align*}
where $\Omega_{\rm e}$ and $\Omega_{\rm e}'$ are respectively as in \eqref{eqn-boze} and \eqref{eqn-boze'}.
Notice that the estimations for ${\rm II}_{2,3,1}(f)$ and ${\rm II}_{2,3,2}(f)$ are respectively similar to those of
${\rm II}_{2,1}$ and ${\rm II}_{2,2}$ in Lemmas \ref{lem-estimateII3} and \ref{lem-estimateII4},
the details being omitted. This implies that
\begin{align*}
{\rm II}_{2,3}(f)\lesssim \|f\|_{\dot{\mathcal{W}}^{s,p}
(\Omega)}^p+\|f\|_{L^p(\Omega)}^p
+\int_\Omega\frac{|f(\xi)|^p}{[{{\mathop\mathrm{dist\,}}}(\xi,D)]^{sp}}\,d\xi
\end{align*}
and hence finishes the proof of Lemma \ref{lem-estimateII5}.
\end{proof}

\begin{lemma}\label{lem-estimateII6}
If $\epsilon\in(0,1]$, $\delta\in(0,\infty]$, and
$\Omega\subset\mathbb{R}^n$ is an
$(\epsilon,\delta,D)$-domain with
$D\subset \partial\Omega$ being closed, then, for any
$s\in(0,1)$, $p\in(1,\infty)$, and $f\in L_{\rm loc}^1(\Omega)$, it holds that
\begin{align*}
{\rm II}_{2,4}(f)\lesssim \|f\|_{L^{p}(\Omega)}^p+\|f\|_{\dot{\mathcal{W}}^{s,p}
(\Omega)}^p
+\int_\Omega\dfrac{|f(x)|^p}{[{{\mathop\mathrm{dist\,}}}(x,D)]^{sp}}\,dx
\end{align*}
with the implicit positive constant independent of $f$, where ${\rm II}_{2,4}(f)$
is as in \eqref{eqn-decom-II-2}.
\end{lemma}

\begin{proof}
Fix $\mu\in(0,1/40]$. From the definitions of $\mathscr{W}_{\rm e}'$
and $\Omega_{\rm e}'$ as in
\eqref{We'} and \eqref{eqn-boze'},  it follows that
\begin{align*}
{\rm II}_{2,4}(f)&\lesssim \sum_{Q\in \mathscr{W}(\overline{\Omega})\setminus \mathscr{W}_{\rm e}'}
\sum_{R\in \mathscr{W}_{\rm e}'} \int_Q \int_R \mathbf{1}_{R \setminus B\left(x,\frac{\ell(Q)}{10}\right)}(y)
\\
&\quad\quad \quad  \times\left|\mathop{\sum}\limits_{P_j\in \mathscr{W}_{\rm e}} \left(E_{P_j^*}f\right)_{P_j^*}\psi_j(x)
-\mathop{\sum}\limits_{S_k\in\mathscr{W}_{\rm e}}
\left(E_{S_k^*}f\right)_{S_k^*}\psi_k(y)\right|^p
\frac{1}{|x-y|^{n+sp}}\,dydx\\
&\lesssim \sum_{R\in \mathscr{W}_{\rm e}'}
\int_R\int_{\Omega^{\complement}\setminus [\Omega_{\rm e}'\cup B(y,\mu \ell(R))]}
\left|\mathop{\sum}\limits_{S_k\in \mathscr{W}_{\rm e}}\left(E_{S_k^*}f\right)_{S_k^*}\psi_k(y)-\mathop{\sum}
\limits_{P_j\in \mathscr{W}_{\rm e}}\left(E_{P_j^*}f\right)_{P_j^*}
\psi_j(x)\right|^p\frac{1}{|x-y|^{n+sp}}\,dxdy,
\end{align*}
where, in the last inequality, we used the fact that, for any
$Q,R\in \mathscr{W}(\overline{\Omega})$, $x\in Q$, and $y\in R$,
it holds that $\mathbf{1}_{R\setminus B(x,{\ell(Q)}/{10})}(y)\le \mathbf{1}_{Q\setminus B(y,\mu\ell(R))}(x)$.
Applying an argument similar to that used in the estimation of ${\rm II}_{2,1}$ as in \eqref{eqn-k11}, we obtain
\begin{align*}
{\rm II}_{2,4}(f)&\lesssim \sum_{R\in\mathscr{W}_{\rm e}'}\int_R\int_{\Omega^{\complement}
\setminus[\Omega_{\rm e}'
\cup B(y,\mu\ell(R))]}\sum_{\genfrac{}{}{0pt}{}{S_k\in \mathscr{W}_{\rm e}}{\overline{S}_k\cap \overline{R}\neq\emptyset}}
\sum_{P_j\in \mathscr{W}_{\rm e}}\left|\left(E_{S_k^*}f\right)_{S_k^*}
-\left(E_{P_j^*}f\right)_{P_j^*}\right|^p\frac{\left|\psi_k(y)\,
\psi_j(x)\right|^p}{|x-y|^{n+sp}}\,dxdy\\ \notag
&\quad+\sum_{R\in \mathscr{W}_{\rm e}'}\int_R
\int_{\Omega^{\complement}\setminus[\Omega_{\rm e}'\cup B(y,\mu \ell(R))]}
\left|\mathop{\sum}\limits_{{\genfrac{}{}{0pt}{}{S_k\in \mathscr{W}_{\rm e}}
{\overline{S}_k\cap\overline{R}\neq\emptyset}}}
\left(E_{S_k^*}f\right)_{S_k^*}\psi_k(y)\right|^p\frac{1}
{|x-y|^{n+sp}}\,dxdy\\ \notag
&\lesssim \|f\|_{L^p(\Omega)}^p+\|f\|_{\dot{\mathcal{W}}^{s,p}
(\Omega)}^p
+\int_\Omega\frac{|f(\xi)|^p}
{{{\mathop\mathrm{\,dist\,}}}(\xi,D)^{sp}}\,d\xi.
\end{align*}
This finishes the proof of Lemma \ref{lem-estimateII6}.
\end{proof}

\begin{lemma}\label{lem-estimateII7}
If $\epsilon\in(0,1]$, $\delta\in(0,\infty]$, and $\Omega\subset\mathbb{R}^n$ is an
$(\epsilon,\delta,D)$-domain with $D\subset \partial\Omega$ being closed, then, for any
$s\in(0,1)$, $p\in(1,\infty)$, and $f\in L_{\rm loc}^1(\Omega)$, it holds that
\begin{align*}
{\rm II}_{2,5}(f)\lesssim \|f\|_{L^{p}(\Omega)}^p+\int_\Omega\dfrac{|f(x)|^p}{[{{\mathop\mathrm{dist\,}}}(x,D)]^{sp}}\,dx
\end{align*}
with the implicit positive constant independent of $f$, where ${\rm II}_{2,5}(f)$
is as in \eqref{eqn-decom-II-2}.
\end{lemma}

\begin{proof}
By the fact $\bigcup_{S_k\in \mathscr{W}_{\rm e}}{\rm supp}\,\psi_k \subset\Omega_{\rm e}$ with
$\Omega_{\rm e}$
as in \eqref{eqn-boze}, we find that, for any $x\in Q\in\mathscr{W}(\overline{\Omega})
\setminus \mathscr{W}_{\rm e}'$ and $y\in\Omega^{\complement}\setminus[\Omega_{\rm e}\cup
B(x,\ell(Q)/10)]$, it holds that $\psi_k(y)=0$ and
\begin{align*}
G(x,y)=\mathop{\sum}\limits_{\genfrac{}{}{0pt}{}{P_j\in \mathscr{W}_{\rm e}}{\overline{P}_j\cap
\overline{Q}\neq\emptyset}}\left(E_{P_j^*}f\right)_{P_j^*} \psi_j(x).
\end{align*}
Then we have
\begin{align*}
{\rm II}_{2,5}(f)&\sim \sum_{Q\in \mathscr{W}(\overline{\Omega})\setminus\mathscr{W}_{\rm e}'}
\int_Q\int_{\Omega^{\complement}\setminus\Omega_{\rm e}}\left|\sum_{\genfrac{}{}{0pt}{}{P_j\in
\mathscr{W}_{\rm e}}{\overline{P}_j\cap\overline{Q}\neq\emptyset}}
\left(E_{P_j^*}f\right)_{P_j^*}\psi_j(x)\right|^p\frac{1}
{|x-y|^{n+sp}}\,dydx\\ \notag
&\lesssim \sum_{Q\in\mathscr{W}(\overline{\Omega})
\setminus \mathscr{W}_{\rm e}'}\int_Q
\sum_{\genfrac{}{}{0pt}{}{P_j\in
\mathscr{W}_{\rm e}}{\overline{P}_j\cap\overline{Q}\neq\emptyset}}
\fint_{P_j^*}\left|E_{P_j^*}f(\xi)\right|^pd\xi
\int_{\Omega^{\complement}\setminus\Omega_{\rm e}}\frac{\left|\psi_j(x)\right|^p}
{|x-y|^{n+sp}}\,dydx\\ \notag
&\lesssim \sum_{Q\in\mathscr{W}(\overline{\Omega})
\setminus\mathscr{W}_{\rm e}'}
\int_Q\sum_{\genfrac{}{}{0pt}{}{P_j\in
\mathscr{W}_{\rm e}}{\overline{P}_j\cap\overline{Q}\neq\emptyset}}
\ell(P_j)^{-n}
\int_{P_j^*\cap\Omega}|f(\xi)|^p
\int_{\Omega^{\complement}\setminus\Omega_{\rm e}}
\frac{1}{|x-y|^{n+sp}}\,dyd\xi dx.
\end{align*}
The remainder of the estimation for ${\rm II}_{2,5}(f)$
is similar to that of
${\rm II}_{2,2}(f)$ as  in Lemma \ref{lem-estimateII4},
the details being omitted.
This finishes the proof of Lemma \ref{lem-estimateII7}.
\end{proof}

\begin{lemma}\label{lem-estimateII8}
If $\epsilon\in(0,1]$, $\delta\in(0,\infty]$, and
$\Omega\subset\mathbb{R}^n$ is an
$(\epsilon,\delta,D)$-domain with $D\subset \partial\Omega$
being closed, then, for any
$s\in(0,1)$, $p\in(1,\infty)$, and $f\in L_{\rm loc}^1(\Omega)$, it holds that
\begin{align*}
{\rm II}_{2,6}(f)\lesssim \|f\|_{L^{p}(\Omega)}^p+\|f\|_{\dot{\mathcal{W}}^{s,p}
(\Omega)}^p+\int_\Omega\dfrac{|f(x)|^p}{[{{\mathop\mathrm{dist\,}}}(x,D)]^{sp}}\,dx
\end{align*}
with the implicit positive constant independent of $f$, where ${\rm II}_{2,6}(f)$
is as in \eqref{eqn-decom-II-2}.
\end{lemma}
\begin{proof}

Using Lemma \ref{lem-Gxy}(i), we obtain
\begin{align*}
{\rm II}_{2,6}(f)&\lesssim \sum_{Q\in{\mathscr{W}}(\overline{\Omega})
\setminus{\mathscr{W}}'_{\rm e}}
\sum_{R\in {\mathscr{W}}_{\rm e}\setminus \mathscr{W}_{\rm e}'}\left[\rule{0pt}{1.2cm}\int_Q\int_{R\setminus B
\left(x,\frac{\ell(Q)}{10}\right)}\left|\left(E_{P_j^*}
f\right)_{P_j^*}-\left(E_{S_k^*}f\right)_{S_k^*}\right|^p
\frac{\left|\psi_j(x)\, \psi_k(y)\right|^p}{|x-y|^{n+sp}}\,dydx\right.\\
&\quad+\int_Q\int_{R\setminus B\left(x,\frac{\ell(Q)}{10}\right)}
\left|\mathop{\sum}
\limits_{\genfrac{}{}{0pt}{}{P_j\in\mathscr{W}_{\rm e}}
{\overline{P}_j\cap\overline{Q}\neq\emptyset}}
\left(E_{P_j^*}f\right)_{P_j^*}\psi_j(x)
\left(1-\mathop{\sum}\limits_{S_k\in\mathscr{W}_{\rm e}}\psi_k(y)\right)
\right|^p\frac{1}{|x-y|^{n+sp}}\,dydx\\
&\left.\quad+\int_Q\int_{R\setminus B\left(x,\frac{\ell(Q)}{10}\right)}\left|\mathop{\sum}
\limits_{\genfrac{}{}{0pt}{}{S_k\in\mathscr{W}_{\rm e}}{\overline{S}_k \cap\overline{R}\neq\emptyset}}\left(E_{S_k^*}f\right)_{S_k^*}\psi_k(y)
\left(1-\mathop{\sum}\limits_{P_j\in \mathscr{W}_{\rm e}}\psi_j(x)\right)\, \right|^p\frac{1}{|x-y|^{n+sp}}\,dydx\rule{0pt}{1cm}\right].
\end{align*}
The first term can be bounded similarly to ${\rm II}_{2,3}(f)$ as in
Lemma \ref{lem-estimateII5}, the details being omitted. This implies that
\begin{align*}
{\rm II}_{2,6}(f)\lesssim \|f\|_{L^p(\Omega)}^p+\|f\|_{\dot{\mathcal{W}}^{s,p}
(\Omega)}^p
+\int_\Omega\frac{|f(\xi)|^p}{[{{\mathop\mathrm{dist\,}}}(\xi,D)]^{sp}}d\xi,
\end{align*}
which completes the proof of Lemma \ref{lem-estimateII8}.
\end{proof}

\section{Equivalences of fractional Sobolev spaces}\label{s4}

In this section, we prove Theorem \ref{thm-equiv} and Proposition
\ref{prop-Hcx}, which establish the
characterizations of fractional Sobolev spaces via the $D$-adapted Hardy inequality.
We begin with the
following proposition  that shows the first identity of \eqref{h-con1}.

\begin{proposition}\label{lemprop-equiv}
Let $s\in(0,1)$, $p\in[1,\infty)$, and $\Omega\subset\mathbb{R}^n$ be an
$(\epsilon,\delta,D)$-domain with $\epsilon\in(0,1]$,
$\delta\in(0,\infty]$,
and $D\subset \partial\Omega$ being closed. If $\Omega$ supports
the $D$-adapted Hardy inequality \eqref{eqn-HI}, then
$W^{s,p}(\Omega)=\mathcal{W}^{s,p}(\Omega).$
\end{proposition}

\begin{proof}
We first prove the inclusion $W^{s,p}(\Omega)\subset \mathcal W^{s,p}(\Omega)$.
For any $f\in W^{s,p}(\Omega)$, let $F\in W^{s,p}(\mathbb{R}^n)$ be the extension of $f$
such that $F|_\Omega=f$ and $\|F\|_{W^{s,p}(\mathbb{R}^n)}\sim \|f\|_{W^{s,p}(\Omega)}$.
By \eqref{eqn-intrinn} and \eqref{eqn-normfsdm}, we have
\begin{align*}
\|f\|_{\mathcal W^{s,p}(\Omega)}\le \|F\|_{W^{s,p}(\mathbb{R}^n)}
\sim \|f\|_{W^{s,p}(\Omega)},
\end{align*}
which immediately implies $f\in \mathcal W^{s,p}(\Omega)$.

To show the reverse inclusion $\mathcal W^{s,p}(\Omega)\subset W^{s,p}(\Omega)$,
for any $f\in \mathcal{W}^{s,p}(\Omega)$, using Proposition \ref{thm1x}, we find that
$\mathcal{E}_Df\in W^{s,p}(\mathbb{R}^n)$, which implies that
$\|f\|_{W^{s,p}(\Omega)}\le \|\mathcal{E}_Df\|_{W^{s,p}(\mathbb{R}^n)}
\lesssim \|f\|_{\mathcal{W}^{s,p}(\Omega)}$
and hence $f\in W^{s,p}(\Omega)$.
This finishes the proof of Proposition \ref{lemprop-equiv}.
\end{proof}

\begin{lemma}\label{lem-textf}
Let $\Omega\subset\mathbb{R}^n$ be a domain and $D$ be a closed part of
$\partial\Omega$. For any $m\in\mathbb{N}$ and $x\in\Omega$, define
\begin{align}\label{eqn-lem-vm}
v_m(x):=
\begin{cases}
1 &\text{if}\ {{\mathop\mathrm{\,dist\,}}}(x,D)\le 1/m,\\
2-m{{\mathop\mathrm{\,dist\,}}}(x,D) &\text{if} \ 1/m<{{\mathop\mathrm{\,dist\,}}}(x,D)\le 2/m, \\
0 &\text{if} \ {{\mathop\mathrm{\,dist\,}}}(x,D)>2/m.
\end{cases}
\end{align}
Then $|v_m(x)-v_m(y)|\le \min\{1,m|x-y|\}$ holds for any $x,y\in\Omega$.
\end{lemma}

\begin{proof}
We only prove the case where ${{\mathop\mathrm{dist\,}}}(x,D)\le 1/m$ and $1/m<{{\mathop\mathrm{dist\,}}}(y,D)\le 2/m$.
The other cases can be showed by a similar argument. Indeed, for those $x$ and
$y$, by \eqref{eqn-lem-vm}, it is easy to verify
\begin{align}\label{eqn-lem4}
|v_m(x)-v_m(y)|=m{{\mathop\mathrm{\,dist\,}}}(y,D)-1\le 1
\end{align}
and ${{\mathop\mathrm{dist\,}}}(y,D)\le {{\mathop\mathrm{dist\,}}}(x,D)+|x-y|\le \frac{1}{m}+|x-y|,$
which, combined with \eqref{eqn-lem4}, implies $|v_m(x)-v_m(y)|\le m|x-y|$ and
hence completes the proof of Lemma \ref{lem-textf}.
\end{proof}

We are now in a position to prove Theorem \ref{thm-equiv}.

\begin{proof}[Proof of Theorem \ref{thm-equiv}]
Using Proposition \ref{lemprop-equiv}, we only need to show $\mathcal  W^{s,p}(\Omega)
=\mathring{W}_D^{s,p}(\Omega)$. From their definitions, it follows that
the inclusion $\mathring{W}_D^{s,p}(\Omega)\subset \mathcal W^{s,p}(\Omega)$ is
trivial. We only need to prove the reverse inclusion
\begin{align}\label{eqn-thm2.3-2}
\mathcal  W^{s,p}(\Omega)\subset \mathring{W}_D^{s,p}(\Omega).
\end{align}
To this end, we claim
\begin{align}\label{claim-sup}
\mathring{W}_D^{s,p}(\Omega)=\mathring{W}_{D,1}^{s,p}(\Omega),
\end{align}
where
\begin{align}\label{eqn-thm2.3-4}
\mathring{W}_{D,1}^{s,p}(\Omega):=\overline{\left\{f\in \mathcal  W^{s,p}(\Omega):\,{{\mathop\mathrm{\,dist\,}}}({\rm supp}\, f,D)
>0\right\}}^{\|\cdot\|_{\mathcal W^{s,p}(\Omega)}}.
\end{align}
Indeed, let $f\in \mathring{W}_D^{s,p}(\Omega)$. By the definition, we find that there
exists a family of smooth functions $\{f_j\}_{j=1}^\infty\subset
C_D^\infty(\Omega)$ such that $f_j\to f$ in $\mathcal W^{s,p}(\Omega)$. Noticing that
$f_j\in  \mathcal W^{s,p}(\Omega)$ and ${{\mathop\mathrm{dist\,}}}({\rm supp}\, f_j,D)>0$ for
any $j\in\mathbb{N}$, we obtain $f\in \mathring{W}_{D,1}^{s,p}(\Omega)$. This implies the
inclusion $\mathring{W}_D^{s,p}(\Omega)\subset \mathring{W}_{D,1}^{s,p}(\Omega)$.

Next, we show the reverse inclusion of \eqref{claim-sup}. Let $u\in
\mathring{W}_{D,1}^{s,p}(\Omega)$. Using definition, we conclude that there exists a family of
functions $\{u_k\}_{k=1}^\infty\subset \mathcal W^{s,p}(\Omega)$ such that
${{\mathop\mathrm{dist\,}}}({\rm supp}\, u_k,D)>0$ for every $k\in\mathbb{N}$ and
\begin{align}\label{eqn-thm2.3-1}
\lim_{k\to\infty}\|u_k-u\|_{\mathcal W^{s,p}(\Omega)}=0.
\end{align}
Let $\mathcal{E}_D$ be the extension operator defined in Definition
\ref{def-extension1}. From Propositions \ref{thm1x} and \ref{lem-extensionOP},
we infer that $\mathcal E_Du_k\in W^{s,p}(\mathbb{R}^n)$ and
${{\mathop\mathrm{dist\,}}}({\rm supp}\, \mathcal{E}_Du_k,D)>0$
holds for any $k\in\mathbb{N}$, which, together with a convolution
argument, implies that $\mathcal{E}_Du_k\in \mathring{W}_D^{s,p}
(\mathbb{R}^n)$ (see, e.g., \cite{CJYZ2024}). Using this, \eqref{eqn-thm2.3-1},
and the boundedness of $\mathcal{E}_D$ from $\mathcal{W}^{s,p}(\Omega)$
to $W^{s,p}(\mathbb{R}^n)$ (see Proposition \ref{thm1x}), we find a sequence
$\{w_j\}_{j=1}^\infty$ in $C_D^\infty(\mathbb{R}^n)$ such that
\begin{align*}
\lim_{j\to\infty}\left\|w_j-\mathcal{E}_Du\right\|_{W^{s,p}(\mathbb{R}^n)}=0.
\end{align*}
Notice that, for any $j\in\mathbb{N}$, one has $w_j|_\Omega\in C_D^\infty(\Omega)$.
This, combined with Proposition \ref{lemprop-equiv}, implies
\begin{align*}
\lim_{j\to\infty}\left\|w_j|_\Omega-u\right\|_{\mathcal W^{s,p}(\Omega)}\le \lim_{j\to\infty}
\left\|w_j-\mathcal{E}_Du\right\|_{W^{s,p}(\mathbb{R}^n)}=0
\end{align*}
and hence $u\in \mathring{W}_D^{s,p}(\Omega)$. Thus, we obtain \eqref{claim-sup}.

We now continue the proof of \eqref{eqn-thm2.3-2}. By \eqref{claim-sup},
we conclude that it suffices to prove
\begin{align}\label{eqn-thm2.3-3}
\mathcal W^{s,p}(\Omega)\subset \mathring{W}_{D,1}^{s,p}(\Omega).
\end{align}
For any $f\in \mathcal  W^{s,p}(\Omega)$, using \eqref{eqn-thm2.3-4}, we only need to find
a sequence $\{f_m\}_{m=1}^\infty$ in $\mathcal  W^{s,p}(\Omega)$ satisfying ${{\mathop\mathrm{dist\,}}}({\rm supp}\,
f_m,D)>0$ such that
\begin{align}\label{eqn-thm2.3-5}
\lim_{m\to\infty}\|f_m-f\|_{\mathcal W^{s,p}(\Omega)}=0.
\end{align}
To this end, for each $m\in\mathbb{N}$, let $v_m$ be as in \eqref{eqn-lem-vm}; we then claim that
\begin{align}\label{claim2-1}
\lim_{m\to\infty}\|v_mf\|_{\mathcal W^{s,p}(\Omega)}=0.
\end{align}
Indeed, by the definition of $v_m$ in \eqref{eqn-lem-vm}, one has
\begin{align*}
\|v_mf\|_{L^p(\Omega)}^p=\int_{D_{2/m}}|v_mf(x)|^pdx\le \int_{D_{2/m}}|f(x)|^pdx\to 0
\end{align*}
as $m \to \infty$, where, for any $\delta>0$,
$D_\delta:= \{x\in\Omega: {{\mathop\mathrm{\,dist\,}}}(x,D)<\delta\}.$

On the other hand, by \eqref{eqn-normfsdm}, we find that
\begin{align}\label{eqn-dp}
\|v_mf\|_{\dot{\mathcal{W}}^{s,p}(\Omega)}^p&=\int_\Omega\int_\Omega
\frac{|v_m(x)f(x)-v_m(y)f(y)|^p}{|x-y|^{n+sp}} \,dydx \notag \\
&\le \int_{D_{3/m}}\int_{D_{3/m}}\frac{|v_m(x)f(x)-v_m(y)f(y)|^p}
{|x-y|^{n+sp}}\,dydx+2\int_{D_{3/m}}
\int_{\Omega\setminus D_{3/m}}\frac{|v_m(x)f(x)|^p}{|x-y|^{n+sp}} \,dydx \notag \\
&=:\mathrm{I}+\mathrm{II}.
\end{align}

For $\mathrm{I}$, from Lemma \ref{lem-textf}, it follows that
\begin{align*}
\mathrm{I}&\lesssim \int_{D_{3/m}}\int_{D_{3/m}}\frac{|v_m(y)|^p\, |f(x)-f(y)|^p}{|x-y|^{n+sp}}\,dydx
+\int_{D_{3/m}}\int_{D_{3/m}}\frac{|f(x)|^p\, |v_m(x)-v_m(y)|^p}{|x-y|^{n+sp}}\,dydx\\
&\lesssim \int_{D_{3/m}}\int_{D_{3/m}} \frac{|f(x)-f(y)|^p}{|x-y|^{n+sp}} \,dydx+\int_{D_{3/m}}\int_{D_{3/m}}
\frac{|f(x)|^p\, \min\left\{1,m^p|x-y|^p\right\}}{|x-y|^{n+sp}}\,dydx\\
&\lesssim \|f\|_{\dot{\mathcal{W}}^{s,p}(D_{3/m})}^p+\int_{D_{3/m}}\int_{\{y\in D_{3/m}:\, |x-y|>1/m\}}\frac{|f(x)|^p}{|x-y|^{n+sp}}\,dydx\\
&\quad+\int_{D_{3/m}}\int_{\{y\in D_{3/m}:\, |x-y|\le1/m\}} \frac{m^p\, |f(x)|^p}{|x-y|^{n-(1-s)p}}\,dydx\\
&\lesssim \|f\|_{\dot{\mathcal{W}}^{s,p}(D_{3/m})}^p+m^{sp}\int_{D_{3/m}}|f(x)|^pdx
\lesssim \|f\|_{\dot{\mathcal{W}}^{s,p}(D_{3/m})}^p+\int_{D_{3/ m}}\frac{|f(x)|^p}{[{{\mathop\mathrm{dist\,}}}(x,D)]^{sp}}\,dx,
\end{align*}
which tends to $0$ as $m\to\infty$ due to the assumption $f\in \mathcal W^{s,p}(\Omega)$.

For $\mathrm{II}$, by an argument similar to that used in the estimation of $\mathrm{I}$,
we conclude that
\begin{align*}
\mathrm{II}\le 2\int_{D_{2/m}}|f(x)|^p \int_{\{y\in\Omega:\, |x-y|>1/m\}}\frac{1}{|x-y|^{n+sp}}\,dydx\lesssim \int_{D_{2/m}}
\frac{|f(x)|^p}{[{{\mathop\mathrm{dist\,}}}(x,D)]^{sp}}\,dx
\end{align*}
tends to $0$ as $m\to\infty$, which proves \eqref{claim2-1}.

With the help of \eqref{claim2-1}, we now come back to the proof of \eqref{eqn-thm2.3-5}.
For any $m\in\mathbb{N}$, let $f_m:=(1-v_m)f$.
Using \eqref{claim2-1}, we obtain
\begin{align*}
\left\|(1-v_m)f\right\|_{\mathcal W^{s,p}(\Omega)}\le \|f\|_{\mathcal  W^{s,p}(\Omega)}+\|v_mf\|_{\mathcal W^{s,p}(\Omega)}<\infty,
\end{align*}
which implies $f_m\in  \mathcal W^{s,p}(\Omega)$. Moreover, by the definition of $v_m$, it is easy to verify
\begin{align*}
{{\mathop\mathrm{dist\,}}}({\rm supp}\, f_m,D)>0.
\end{align*}
From this and \eqref{claim2-1}, we deduce that \eqref{eqn-thm2.3-5}
and hence \eqref{eqn-thm2.3-3} and \eqref{eqn-thm2.3-2} hold. Altogether, \eqref{h-con1}
holds, which completes the proof of Theorem \ref{thm-equiv}.
\end{proof}

\begin{remark}\label{rem-notepf}
We point out that the proof of Theorem \ref{thm-equiv} works also for the case $s=1$.  The only
difference is that, unlike estimate \eqref{eqn-dp} in the fractional case,
we use Leibniz's rule and the fact that $\nu_m$ is a bounded Lipschitz function satisfying
${\rm supp}\, \nu \subset D_{3/m}$ and $|\nabla \nu_m|\lesssim m$ to obtain
\begin{align*}
\|v_mf\|_{\dot{\mathcal{W}}^{1,p}(\Omega)}^p&\lesssim \int_\Omega
|\nabla \nu_m(x)|^p|f(x)|^p\,dx+\int_\Omega
|\nu_m(x)|^p|\nabla f(x)|^p\,dx\\
&\lesssim \int_{D_{3/m}}|f(x)|^p\,dx+\int_{D_{2/m}\setminus D_{1/m}}
\frac{|\nabla f(x)|^p}{[{\rm dist\,}(x,D)]^p}\,dx.
\end{align*}
In particular, under \eqref{eqn-IHI},
$\mathcal  W^{1,p}(\Omega)= \mathring{W}_D^{1,p}(\Omega)$ holds.
\end{remark}

To prove Proposition \ref{prop-Hcx}, we recall the
following result, which is precisely \cite[Proposition 6.6]{Bec19-In}.

\begin{lemma}[\cite{Bec19-In}]\label{mixed-Hb}
Let $O\subset\mathbb{R}^n$ be an open $n$-set with $D\subset \partial O$ being
a closed uniformly $(n-1)$-set. If $p\in(1,\infty)$ and $s\in(0,1)$ satisfies
$sp\neq1$, then there exists a positive constant $C$ such that
\begin{align*}
\left\{\int_O \frac{|f(x)|^p}{[{{\mathop\mathrm{dist\,}}}(x,D)]^{sp}}\,dx\right\}^\frac{1}{p}\le C \|f\|_{W^{s,p}(O)}
\end{align*}
for any $f\in W^{s,p}(O)$ as in \eqref{eqn-fstrace} when $sp<1$ and for any $f\in W_D^{s,p}(O)$
as in \eqref{eqn-fstraceD} when $sp>1$.
\end{lemma}

The following result is a slight improvement of Lemma \ref{mixed-Hb}.

\begin{lemma}\label{mixed-H}
Let $\Omega\subset\mathbb{R}^n$ be an $(\epsilon,\delta,D)$-domain with
$D\subset \partial\Omega$ being a closed uniformly $(n-1)$-set. If $p\in
(1,\infty)$ and $s\in(0,1)$ satisfies $sp\neq1$, then there exists a positive constant $C$
such that, for any $f\in \mathring{W}_D^{s,p}(\Omega)$ as in \eqref{h-con121},
it holds that
\begin{align}\label{ineq-hd}
\left\{\int_\Omega \frac{|f(x)|^p}{[{{\mathop\mathrm{dist\,}}}(x,D)]^{sp}}
dx\right\}^\frac{1}{p}\le C \|f\|_{W^{s,p}(\Omega)}.
\end{align}
\end{lemma}

\begin{proof}
Let $G:= \mathbb{R}^n \setminus D$. It is easy to find that $G$ is a domain in $\mathbb{R}^n$ with
$\partial G=D$. We first claim that $G$ is an $n$-set as in \eqref{def-d-set}.
Indeed, for any $x\in G$ and $r\in(0,1]$, if $B(x,r)\cap D=\emptyset$, then
\begin{align*}
|B(x,r)\cap G|=|B(x,r)|\sim r^n.
\end{align*}
Otherwise, if $B(x,r)\cap D\neq\emptyset$, by the assumption
that $D$ is an $(n-1)$-set, we find that $|D|=0$ and then
\begin{align*}
|B(x,r)\cap G|\ge |B(x,r)|-|D|=|B(x,r)|\sim r^n.
\end{align*}
Thus, using Lemma \ref{mixed-Hb}, we obtain
\begin{align}\label{eqn-HIMi}
\left\{\int_G \frac{|g(x)|^p}{[{{\mathop\mathrm{dist\,}}}(x,D)]^{sp}}
dx \right\}^\frac{1}{p}\lesssim \|g\|_{W^{s,p}(G)}
\end{align}
for any $g\in W^{s,p}(G)$ if $sp<1$ and for any $g \in W_D^{s,p}(G)$ if $sp>1$.

Now assume first $sp<1$. In this case, for any
$f\in W^{s,p}(\Omega)$, by the definition of $W^{s,p}(\Omega)$,
there exists $F\in W^{s,p}(\mathbb{R}^n)$
such that $F|_\Omega=f$ and, moreover,
$F|_G\in W^{s,p}(G)$. Using this and \eqref{eqn-HIMi}, we conclude that
\begin{align*}
\left\{\int_\Omega \frac{|f(x)|^p}{[{{\mathop\mathrm{dist\,}}}(x,D)]^{sp}} dx \right\}^\frac{1}{p}
\le \left\{\int_G \frac{|F(x)|^p}{[{{\mathop\mathrm{dist\,}}}(x,D)]^{sp}} dx \right\}^\frac{1}{p}
\lesssim \|F\|_{W^{s,p}(G)}\lesssim \|F\|_{W^{s,p}(\mathbb{R}^n)}.
\end{align*}
Taking the infimum over all such $F$, we obtain \eqref{ineq-hd} for any $f \in W^{s,p}(\Omega)$,
which, combined with Theorem \ref{thm-equiv}, proves Lemma \ref{mixed-H} in the case $sp<1$.
The case $sp>1$ can be dealt  in a similar way based on \eqref{eqn4.15}, the details being omitted.
Combining the above two cases then completes the proof of Lemma \ref{mixed-H}.
\end{proof}

We are now in a position to show Proposition \ref{prop-Hcx}.

\begin{proof}[Proof of Proposition \ref{prop-Hcx}]
From the proof of Theorem \ref{thm-equiv} and Remark \ref{rem1.4}(ii), we infer that
\begin{align*}
\left\{f\in \mathcal  W^{s,p}(\Omega):\ \int_{\Omega} \dfrac{|f(x)|^p}
{[{\rm dist\,}(x,D)]^{sp}}\,dx<\infty\right\}\subset \mathring{W}_D^{s,p}(\Omega).
\end{align*}
Thus, we only need to prove that, if $f\in \mathring{W}_D^{s,p}(\Omega)$, then
\begin{align*}
\int_\Omega \frac{|f(x)|^p}{[{{\mathop\mathrm{dist\,}}}(x,D)]^{sp}}\,dx
\lesssim\left\|f\right\|^p_{\mathcal  W^{s,p}(\Omega)},
\end{align*}
which is an immediately consequence of Lemma \ref{mixed-H}  when $s\in (0,1)$ and
\cite[Theorem 2.1 and Remark 2.2]{EgTo17} when $s=1$,  together with the assumption that $\Omega$ is an $n$-set.
This finishes the proof of Proposition \ref{prop-Hcx}.
\end{proof}

\section{Real interpolation and applications to elliptic operators}\label{s5}

In this section, we first prove Proposition \ref{t5} in Section \ref{s5.1},
which shows that the real
interpolation space $(L^p(\Omega, \mathcal{W}_{d_D}^{1,p}(\Omega)))_{s,p}$
equals the weighted fractional Sobolev space $\mathcal{W}_{d_D^s}^{s,p}(\Omega)$
in Definition \ref{def-wfs}. Then,  in  Section \ref{s5.2}, we apply Proposition \ref{t5} to
prove Corollaries \ref{cor1.11} and \ref{cor1.12}, which
characterize the domain of the fractional power and the maximal regularity of elliptic operators
with mixed boundary.

\subsection{Characterization of real interpolation}\label{s5.1}

To prove Proposition \ref{t5}, we begin with the following technical lemmas. 
The following lemma is a variant of  \cite[Theorem 1.1]{Bec21}.

\begin{lemma}\label{lem4.5}
Let $s\in (0,1)$,
$p\in (1,\infty)$, and $\Omega\subset\mathbb{R}^n$ be an $(\epsilon,\delta,D)$-domain with
$D\subset \partial\Omega$ being closed. Suppose $\mathcal{E}_D$ is the linear
extension operator defined as in \eqref{ext-x1}. Then $\mathcal{E}_D$ is
bounded from the weighted Sobolev spaces $\mathcal{W}_{d_D^s}^{s,p}(\Omega)$ to $\mathcal{W}_{d_D^s}^{s,p}(\mathbb{R}^n)$.
\end{lemma}

\begin{proof}
Using Remark \ref{rem1.9}, we find that
$\mathcal{W}_{d_D^s}^{s,p}(\Omega)=\mathcal{W}^{s,p}(\Omega)\cap L^p(\Omega,{d_D^{-sp}})\subset \mathring W_D^{s,p}(\Omega)$.
Thus, to finish the proof of Lemma \ref{lem4.5}, we only
need to verify  that $\mathcal{E}_D$ is both bounded from $\mathring W_D^{s,p}(\Omega)$ to
$W^{s,p}(\mathbb{R}^n)$ and bounded from $ L^p(\Omega,{d_D^{-sp}})$ to
$L^p(\mathbb{R}^n,{d_D^{-sp}})$.
Recall that the former boundedness is proven in Theorem \ref{thm1}, it remains to show the latter
one. To this end, for any $f\in  L^p(\Omega,{d_D^{-sp}})$, by \eqref{ext-x1}, we first write
\begin{align*}
\left\|\mathcal{E}_Df\right\|_{L^p(\mathbb{R}^n,{d_D^{-sp}})}^p=\int_{\Omega}\dfrac{|f(x)|^p}{[{\rm dist}\,(x,D)]^{sp}}\,dx
+\int_{\Omega^\complement}\dfrac{|\mathcal{E}_Df(x)|^p}{[{\rm dist}\,(x,D)]^{sp}}\,dx=:\mathrm{I}+\mathrm{II},
\end{align*}
where $\mathrm{I}=\|f\|_{L^p(\Omega,{d_D^{-sp}})}^p$. To estimate $\mathrm{II}$,
using \eqref{ext-x1} again and Lemma \ref{W}, we obtain
\begin{align}\label{4.12}
\mathrm{II}&\lesssim \sum_{Q_j\in \mathscr{W}_e }\int_{c_nQ_j}\dfrac{\left|(E_{Q_j^*}f)_{Q_j^*}\right|^p}{[{\rm dist}\,(x,D)]^{sp}}\,dx
\notag \\
&\lesssim \sum_{Q_j\in \mathscr{W}_e }\int_{c_nQ_j}\left( \frac{1}{|Q_j^*|} \int_{Q_j^*}\left| f(y) \right|^p\,dy \right)
\frac{1}{[{\rm dist}\,(x,D)]^{sp}}\,dx.
\end{align}
Notice that, for any $x\in c_nQ_j$ with $Q_j\in \mathscr{W}_e$ as in \eqref{We} and for any $y\in Q_j^*$ as in Lemma \ref{lem3.4.4},
by Lemma \ref{W} and Remark \ref{remark2.3}, we conclude that
\begin{align*}
{\rm dist}\,(y,D)&\le {\rm diam\,} (Q_j^*)+ {\rm dist}\,(Q_j^*,c_nQ_j)+{\rm diam\,} (c_nQ_j)+{\rm dist}\,(x,D)\\
&\lesssim {\rm diam\,} (Q_j)+{\rm dist}\,(x,D)
\lesssim {\rm dist}\,(x,D),
\end{align*}
which, combined with \eqref{4.12}, implies that
\begin{align*}
 \mathrm{II}\lesssim \sum_{Q_j\in \mathscr{W}_e }\int_{Q_j^*}\dfrac{| f(y) |^p}{[{\rm dist}\,(y,D)]^{sp}}\,dy
 \lesssim  \|f\|_{L^p(\Omega,{d_D^{-sp}})}^p.
\end{align*}

Altogether the estimates of $\mathrm{I}$ and $\mathrm{II}$, we conclude the boundedness of
$\mathcal{E}_D$ from  $ L^p(\Omega,{d_D^{-sp}})$ to $L^p(\mathbb{R}^n,{d_D^{-sp}})$.
This finishes the proof of Lemma \ref{lem4.5}.
\end{proof}

\begin{lemma}\label{lem4.6}
Let $s, s_0\in (0,1)$, $p\in (1,\infty)$, and $\Omega\subset\mathbb{R}^n$ be an $(\epsilon,\delta,D)$-domain with
$D\subset \partial\Omega$ being closed. Then
\begin{align*}
\left(L^p(\Omega), \mathcal{W}_{d_D^{s_0}}^{s_0,p}(\Omega)\right)_{s,p}\subset \mathcal{W}_{d_D^{ss_0}}^{ss_0,p}(\Omega).
\end{align*}
\end{lemma}

\begin{proof}
By Lemma \ref{lem4.5}, we obtain $\mathcal{E}_D$ in \eqref{ext-x1} is a linear
bounded extension from $\mathcal{W}_{d_D^{s_0}}^{s_0,p}(\Omega)$ to $\mathcal{W}_{d_D^{s_0}}^{s_0,p}(\mathbb{R}^n)$. Moreover, let
$\mathcal{R}_\Omega f:=f|_\Omega$ be the restriction operator to $\Omega$.
Using the retract theory of interpolation (see, e.g., \cite{Ber12,Tri95})
and the fact  $\mathcal{W}_{d_D^{s_0}}^{s_0,p}(\Omega)=\mathcal W^{s_0,p}(\Omega)\cap
L^{p}(\Omega,d_D^{-s_0 p})$, we find that
\begin{align*}
\left(L^p(\Omega), \mathcal{W}_{d_D^{s_0}}^{s_0,p}(\Omega)\right)_{s,p}&=\mathcal{R}_\Omega
\left[\left(L^p(\mathbb{R}^n), \mathcal{W}_{d_D^{s_0}}^{s_0,p}(\mathbb{R}^n)\right)_{s,p}\right]\cap
\left(L^p(\Omega), L^{p}\left(\Omega,d_D^{-s_0 p}\right)\right)_{s,p}\\
&\subset\mathcal{R}_\Omega
\left[\left(L^p(\mathbb{R}^n), W^{s_0,p}(\mathbb{R}^n)\right)_{s,p}\right]\cap
\left(L^p(\Omega), L^{p}\left(\Omega,d_D^{-s_0 p}\right)\right)_{s,p}\\
&\subset\mathcal{W}^{ss_0,p}(\Omega)\cap
\left(L^p(\Omega), L^{p}\left(\Omega,d_D^{-s_0 p}\right)\right)_{s,p}.
\end{align*}
On the other hand, from the real interpolation of $L^p(\Omega)$ with change of measures
(see, e.g., \cite[Theorem 5.4.1]{Ber12}),
we deduce that $(L^p(\Omega), L^p(\Omega,d_D^{-s_0 p}))_{s,p}=L^p(\Omega,d_D^{-ss_0 p})$.
This, together with Remark \ref{rem1.9}, implies that
\begin{align*}
\mathcal{W}^{ss_0,p}(\Omega)\cap
\left(L^p(\Omega), L^{p}\left(\Omega,d_D^{-s_0 p}\right)\right)_{s,p}
=\mathcal{W}^{ss_0,p}(\Omega)\cap  L^{p}\left(\Omega,d_D^{-ss_0 p}\right)
=\mathcal{W}_{d_D^{ss_0}}^{ss_0,p}(\Omega)
\end{align*}
and hence finishes the proof of  Lemma \ref{lem4.6}.
\end{proof}

To establish the reverse inclusion in Proposition \ref{t5},
we need a trace characterization of real interpolation from
\cite{Gri69,Bec19-In}. To be precise, let
$\Omega_D:=( \Omega\times \{0\} )\cup ( D\times \mathbb{R} )\subset \mathbb{R}^{n+1}$.
Assume that $\Omega$ is an $n$-set and $D$ a uniformly $(n-1)$-set.  By \cite[Lemma 6.7]{Bec19-In}, we find that
$\Omega_D$ is an $n$-set in $\mathbb{R}^{n+1}$. Using this and following Jonsson and Wallin \cite[p.\,103]{Jon1984}, we
can define the fractional Sobolev space
$\mathcal{W}^{s,p}(\Omega_D)$ on $\Omega_D$ as follows: for any $s\in (0,1)$ and $p\in (1,\infty)$,
\begin{align*}
\mathcal{W}^{s,p}(\Omega_D):=\left\{f\in L^p(\Omega_D):\ \left\|f\right\|_{\mathcal{W}^{s,p}(\Omega_D)}<\infty\right\},
\end{align*}
where
\begin{align}\label{fss}
\left\|f\right\|_{\mathcal{W}^{s,p}(\Omega_D)}:=\left[\int_{\Omega_D}
\left| f({\mathbf x}) \right|^p\, \mathcal{H}^{n}(d{\mathbf x})\right]^{\frac 1p}
+\left[\iint_{{\Omega_D}\times {\Omega_D}} \dfrac{|f({\mathbf x})-f({\mathbf y})|^{p}}{|{\mathbf x}-{\mathbf y}|^{n-sp}} \,
\mathcal{H}^{n}(d{\mathbf x})\mathcal{H}^{n}(d{\mathbf y})\right]^{\frac1p}
\end{align}
and $\mathcal{H}^{n}$ denotes the $n$-dimensional Hausdorff measure on $\Omega_D$.

Let $\mathcal{E}_0:\ L_{\rm loc}^1(\Omega)\to L_{\rm loc}^1(\Omega_D)$ be a zero extension
defined by setting, for any $f\in L_{\rm loc}^1(\Omega)$ and $\mathbf x\in \Omega_D$,
\begin{align}\label{eqn4.13}
\mathcal{E}_0f(\mathbf{x}):=
\begin{cases}
f(x) & \mbox{if } \ \mathbf{x}=(x,0)\in \Omega\times \{0\}, \\
 0 & \mbox{if}\ \mathbf{x}\in D\times \mathbb{R}.
\end{cases}
\end{align}
The following lemma is from the proof of \cite[Lemma 6.10]{Bec19-In}, which
says that, if there exists a zero extension from a subspace $X^{s,p}(\Omega)$
of $W^{s,p}(\Omega)$ to the fractional Sobolev space $\mathcal{W}^{s,p}(\Omega_D)$
on the $n$-set $\Omega_D$ of  $\mathbb{R}^{n+1}$,
then, using the extension theory of  Jonsson-Wallin \cite{Jon1984} and the trace theory of
real interpolation of Grisvard \cite{Gri69}, we can show that $X^{s,p}(\Omega)$
is contained in the real interpolation space
$(L^p(\Omega), \mathring{W}_D^{1,p}(\Omega))_{s,p}$; we omit the details.

\begin{lemma}[\cite{Bec19-In}]\label{lem4.7}
Let $s\in (0,1)$, $p\in (1,\infty)$, and
$\Omega\subset\mathbb{R}^n$ be an $(\epsilon,\delta,D)$-domain with $\Omega$ being an $n$-set and
$D\subset \partial\Omega$ being a closed uniformly $(n-1)$-set. Suppose $X^{s,p}(\Omega)$ is a subspace of $W^{s,p}(\Omega)$
satisfying that the  zero-extension $\mathcal{E}_0$ in \eqref{eqn4.13} is bounded from $X^{s,p}(\Omega)$ to
$\mathcal{W}^{s,p}(\Omega_D)$. Then it holds that
\begin{align*}
X^{s,p}(\Omega)\subset \left(L^p(\Omega), \mathring{W}_D^{1,p}(\Omega)\right)_{s,p}.
\end{align*}
\end{lemma}

With the help of the above technical lemmas, we now turn to the proof of Proposition  \ref{t5}.

\begin{proof}[Proof of Proposition \ref{t5}]
We prove the present proposition by considering the following two cases based on the size of $sp$.

{\bf Case 1}: $sp\ne 1$. In this case,  let $\mathcal{E}_D$ be the linear extension in \eqref{ext-x1} and
$\mathcal{R}_\Omega:L^1_{\rm{loc}}(\mathbb{R}^n)\to L^1(\Omega)$ the trace operator
defined by setting $\mathcal{R}_\Omega f:=f|_\Omega$ for any $f\in L_{\rm loc}^1(\mathbb{R}^n)$.
Recall that
$\mathcal{E}_D$ is a bounded linear operator from $\mathring{W}_{D}^{1,p}(\Omega)$
to $\mathring{W}_{D}^{1,p}(\mathbb{R}^n)$ (see \cite[Theorem 1.2]{Bec19-Ext}) and
$\mathcal{R}_\Omega$ is a bounded linear operator
from $\mathring{W}_{D}^{1,p}(\mathbb{R}^n)$ to $\mathring{W}_{D}^{1,p}(\Omega)$. This,
combined with the fact
$\mathcal{R}_\Omega\circ \mathcal{E}_D=I$ on $\mathring{W}_{D}^{1,p}(\Omega)$, implies
that $(\mathcal{E}_D,\mathcal{R}_\Omega)$ is a retract from $\mathring{W}_{D}^{1,p}(\Omega)$
to $\mathring{W}_{D}^{1,p}(\mathbb{R}^n)$. Similarly,
we find that $(\mathcal{E}_D,\mathcal{R}_\Omega)$ is also a retract from $L^p(\Omega)$
to $L^p(\mathbb{R}^n)$. Thus, using the retract theory of interpolation (see, e.g., \cite{Ber12,Tri95}), we obtain
\begin{align*}
\left(L^{p}(\Omega),\mathring W_D^{1,p}(\Omega) \right)_{s,p}=
\mathcal{R}_\Omega\circ
\left(L^{p}(\mathbb{R}^n),\mathring W_D^{1,p}
(\mathbb{R}^n) \right)_{s,p}.
\end{align*}
Moreover, recall that it was proven in \cite[Theorem 1.1]{Bec19-In} that
\begin{align*}
\left(L^{p}(\mathbb{R}^n),\mathring W_D^{1,p}
(\mathbb{R}^n) \right)_{s,p}=
\begin{cases}
\mathring W^{s,p}(\mathbb{R}^n) & \mbox{if } \ sp<1, \\
\mathring W_D^{s,p}(\mathbb{R}^n) & \mbox{if}\ sp>1.
\end{cases}
\end{align*}
This, together with the following facts
\begin{align*}
\mathcal{R}_\Omega\circ\mathring W^{s,p}
(\mathbb{R}^n)=\mathcal{R}_\Omega\circ W^{s,p}(\mathbb{R}^n)
=W^{s,p}(\Omega)=\mathring W_D^{s,p}(\Omega)= \mathcal{W}_{d_D^{s}}^{s,p}(\Omega)
\end{align*}
for $sp<1$  (see Remarks \ref{rem-1.5}(iii) and \ref{rem1.9}) and
\begin{align*}
\mathcal{R}_\Omega\circ\mathring W_D^{s,p}(\mathbb{R}^n)=\mathcal{R}_\Omega\circ W_D^{s,p}(\mathbb{R}^n)
=W^{s,p}_D(\Omega)=\mathring W_D^{s,p}(\Omega)= \mathcal{W}_{d_D^{s}}^{s,p}(\Omega)
\end{align*}
for $sp>1$, proves that \eqref{int-hh} holds.

{\bf Case 2:} $sp=1$. In this case, by the iteration theorem of real interpolation (see, e.g., \cite{Tri92})
and {\bf Case 1}, we only need to show
that, for any $s_0\in (\frac{1}{p},1)$, it holds that
\begin{align*}
\left(L^p(\Omega), \mathcal{W}_{d_D^{s_0}}^{s_0,p}(\Omega)\right)_{s,p}= \mathcal{W}_{d_D^{ss_0}}^{ss_0,p}(\Omega).
\end{align*}
From Lemma \ref{lem4.6}, it follows that
\begin{align*}
\left(L^p(\Omega), \mathcal{W}_{d_D^{s_0}}^{s_0,p}(\Omega)\right)_{s,p}\subset  \mathcal{W}_{d_D^{ss_0}}^{ss_0,p}(\Omega).
\end{align*}
Thus, by Lemma \ref{lem4.7}, to finish the proof of the present proposition, it suffices
to show that the zero extension $\mathcal{E}_0$ in \eqref{eqn4.13} is bounded from $\mathcal{W}_{d_D^s}^{s,p}(\Omega)$ to
$\mathcal{W}^{s,p}(\Omega_D)$ for any $s\in (0,1)$ and $p\in (1,\infty)$. We verify this by using some ideas from
the proof of \cite[Proposition {6.8}]{Bec19-In}. More precisely,
for any $f\in \mathcal{W}_{d_D^s}^{s,p}(\Omega)$, from \eqref{fss} and \eqref{eqn4.13} we infer that
\begin{align}\label{eqn-ze}
\left\|  \mathcal{E}_0f\right\|_{\mathcal{W}^{s,p}(\Omega_D)}^p&= \int_{\Omega_D}
\left| f({\mathbf x}) \right|^p \mathcal{H}^{n}(d{\mathbf x})
+ \iint_{{\Omega_D}\times {\Omega_D}} \dfrac{|\mathcal{E}_0f({\mathbf x})
-\mathcal{E}_0f({\mathbf y})|^{p}}{|{\mathbf x}-{\mathbf y}|^{n-sp}} \,
\mathcal{H}^{n}(d{\mathbf x})\mathcal{H}^{n}(d{\mathbf y}) \notag\\
&\le \|f\|^p_{L^p(\Omega)}+\iint_{\genfrac{}{}{0pt}{}{{\Omega_D}\times {\Omega_D}}{|\mathbf{x}-\mathbf{y}|<1}
} \ldots\,\mathcal{H}^{n}(d{\mathbf x})\mathcal{H}^{n}(d{\mathbf y}) +\iint_{\genfrac{}{}{0pt}{}{{\Omega_D}\times {\Omega_D}}
{|\mathbf{x}-\mathbf{y}|\ge 1}}
\dots\,\mathcal{H}^{n}(d{\mathbf x})\mathcal{H}^{n}(d{\mathbf y}) \notag\\
&=:\|f\|^p_{L^p(\Omega)}+\mathrm{I}+\mathrm{II}.
\end{align}

For $\mathrm{II}$, by the fact that $\mathcal{H}^n|_{\Omega}$ reduces
to the $n$-dimensional Lebesgue measure on $\mathbb{R}^n$
and the fact that $D\times \mathbb{R}$ is an $n$-set, we obtain
\begin{align*}
\mathrm{II}&=\iint_{\genfrac{}{}{0pt}{}{{\Omega}\times {\Omega}}
{|{x}-{y}|\ge 1}}
\dfrac{|f({ x})-f({y})|^{p}}{|{ x}-{ y}|^{n-sp}} \,d{ x}d{y}+2\iint_{\genfrac{}{}{0pt}{}{{(D\times \mathbb{R})}
\times \Omega}
{|\mathbf{x}-( y,0)|\ge 1}}
\dfrac{|f({ y})|^{p}}{|{\mathbf x}-{( y,0)}|^{n-sp}} \,
\mathcal{H}^{n}(d{\mathbf x})d{ y}\\
&\lesssim \|f\|^p_{\dot {\mathcal W}^{s,p}(\Omega)}+\|f\|_{L^p(\Omega)}^p,
\end{align*}
which is desired.

To estimate $\mathrm{I}$, for any ${y}\in \Omega$, let
$$F(y):=\int_{\genfrac{}{}{0pt}{}{{D\times \mathbb{R}}}
{|\mathbf{x}-( y,0)|< 1}}
\dfrac{1}{|{\mathbf x}-{( y,0)}|^{n-sp}} \,
\mathcal{H}^{n}(d{\mathbf x})$$
Then, for any $|\mathbf{x}-( y,0)|\ge 1$, we have
$d(y,D)\le d((y,0),\mathbf{x})<1$. Thus, there exists $k_0=k_0(y)$ such that
$2^{-k_0-1}\le d(y,D)<2^{-k_0},$
which implies that $d((y,0), \mathbf{x})\ge d(y,D)\ge 2^{-{k_0}-1}$ and hence,
 using the fact that $D\times \mathbb{R}$ is an $n$-set, we conclude that
\begin{align}\label{eqn-eofF}
F(y)&\le \sum_{k=0}^{k_0} \int_{\genfrac{}{}{0pt}{}{D\times \mathbb{R}}
{2^{-k-1}\le |\mathbf{x}-({y},0)|<2^{-k}} }\dfrac{1}{|{\mathbf x}-{( y,0)}|^{n-sp}} \,
\mathcal{H}^{n}(d{\mathbf x}) \notag\\
&\lesssim\sum_{k=0}^{k_0} 2^{k(n+sp)} 2^{-nk}\lesssim 2^{k_0sp}
\sim \frac{1}{{\rm dist\,}(y,D)^{sp}}.
\end{align}
From \eqref{eqn-ze} and \eqref{eqn-eofF}, we deduce that
\begin{align*}
\mathrm{I}\lesssim \int_\Omega \dfrac{|f(y)|^{p}}{[{\rm dist\,}(y,D)]^{sp}}\,dy,
\end{align*}
which, combined with the estimates of $\mathrm{II}$, shows that
$\mathcal{E}_0$  is bounded from $\mathcal{W}_{d_D^s}^{s,p}(\Omega)$ to
$\mathcal{W}^{s,p}(\Omega_D)$.
This finishes the proof of Proposition \ref{t5}.
\end{proof}

\subsection{Applications to elliptic operators}\label{s5.2}

This subsection is devoted to the proofs of Corollaries \ref{cor1.11} and \ref{cor1.12}.
We begin with showing that the bilinear form
 $(\mathfrak a,\mathcal{W}_{d_D}^{1,p}(\Omega))$ in \eqref{eqn-DF} is a Dirichlet form
 in $L^2(\Omega)$.
Recall in \cite{FuOsTa11}  that a bilinear form $(\mathcal{E}, {\mathcal{F}})$ in $L^2(\Omega)$
is called a \emph{Dirichlet form} if $\mathcal{E}$ is a closed symmetric bilinear form with the domain
$\mathcal{F}$ dense in $L^2(\Omega)$ and $\mathcal{E}$ having the {Markovian property}:
for any $f\in \mathcal{F}$, it holds that $g:=0\vee f\wedge 1\in  \mathcal{F}$ and
\begin{align}\label{eqn-MP}
\mathcal{E}(g,g)\le \mathcal{E}(f,f).
\end{align}

\begin{lemma}\label{lem-DF}
Let $\Omega\subset\mathbb{R}^n$ be an $(\epsilon,\delta,D)$-domain with $\Omega$ being an $n$-set and
$D\subset \partial\Omega$ being a closed uniformly $(n-1)$-set.
Then $(\mathfrak a,\mathcal{W}_{d_D}^{1,p}(\Omega))$ in \eqref{eqn-DF} is a Dirichlet form in $L^2(\Omega)$.
\end{lemma}

To prove Lemma \ref{lem-DF}, we need some basic properties on weak derivatives from \cite{ArEl97}.

\begin{lemma}[\cite{ArEl97}]\label{lem-weakd}
Let $p\in (1,\infty)$ and $\Omega\subset \mathbb{R}^n$ be a domain. Suppose $f,g\in \mathcal{W}^{1,p}(\Omega)$.
Then the following assertions hold.
\begin{itemize}
\item [{\rm (i)}] Let $f^+:=0\vee f:=\max\{0,f\}$ and $ f^-:=0\vee (-f)$. Then $f^+, f^-, |f|\in W^{1,p}(\Omega)$ and,
for any $i\in \{1,\ldots,n\}$,  it holds that
$D_i( f^+)=\mathbf{1}_{\{f>0\}}D_i  f$,
$D_i(f^-)=-\mathbf{1}_{\{f<0\}}D_i  f$,  and $D_i(|f|)=({\rm sgn}\, f)\,D_i  f$,
where $D_i$ denotes the weak derivative with respect to $x_i$.

\item [{\rm (ii)}] Assume $f\ge 0$. Then  $f\wedge 1:=\min\{f,1\}\in \mathcal{W}^{1,p}(\Omega)$ and,
for any $i\in \{1,\ldots,n\}$, it holds that
$D_i( f\wedge 1)=\mathbf{1}_{\{f<1\}}D_i  f$.

\item [{\rm (iii)}] $f\vee g$, $f\wedge g\in \mathcal{W}^{1,p}(\Omega)$ and, for any $i\in \{1,\ldots,n\}$,
it holds that
$$D_i( f\vee  g)=\mathbf{1}_{\{f\ge g\}}D_i  f+\mathbf{1}_{\{f< g\}}D_i  g$$
and
$D_i( f\wedge  g)=\mathbf{1}_{\{f\le g\}}D_i  f+\mathbf{1}_{\{f> g\}}D_i  g$.
\end{itemize}
\end{lemma}

\begin{proof}[Proof of Lemma \ref{lem-DF}]
From \eqref{eqn-DF} and Remark \ref{rem1.9}, it follows that $(\mathfrak a,\mathcal{W}_{d_D}^{1,p}(\Omega))$ is a closed
symmetric bilinear form in $L^2(\Omega)$. On the other hand, it follows from Remark \ref{rem1.9} that
\begin{align*}
C_{\rm c}^\infty(\Omega)\subset\mathring{W}_{D}^{1,p}(\Omega)=\mathcal{W}_{d_D}^{1,p}(\Omega).
\end{align*}
This implies that $\mathcal{W}_{d_D}^{1,p}(\Omega)$ is a dense space in $L^2(\Omega)$.
Thus, to show  $(\mathfrak a,\mathcal{W}_{d_D}^{1,p}(\Omega))$ is a Dirichlet form,
we only need to verify the Markovian property \eqref{eqn-MP}.
To this end, for any $f\in \mathcal{W}_{d_D}^{1,p}(\Omega)$,
by Lemma \ref{lem-weakd}, we have, for any $i\in \{1,\ldots,n\}$,
\begin{align*}
D_i(0\vee f\wedge 1)=\mathbf{1}_{\{f^+<1\}}D_i (f^+)=\mathbf{1}_{\{f^+<1\}}\mathbf{1}_{\{f>0\}}D_i (f)\in L^p(\Omega),
\end{align*}
which, together with the fact $0\vee f\wedge 1\le f$, implies that $0\vee f\wedge 1\in \mathcal{W}_{d_D}^{1,p}(\Omega)$ and
$$\mathfrak a\left(0\vee f\wedge 1,0\vee f\wedge 1\right)\le \mathfrak a(f,f).$$
This verifies \eqref{eqn-MP} and hence shows
that $(\mathfrak a,\mathcal{W}_{d_D}^{1,p}(\Omega))$ is a Dirichlet form in $L^2(\Omega)$, which completes the
proof of Lemma   \ref{lem-DF}.
\end{proof}

Under the assumptions of Lemma \ref{lem-DF},
let $\mathcal{L}_D$ be the associated operator  of the Dirichlet form $(\mathfrak a,\mathcal{W}_{d_D}^{1,p}(\Omega))$.
It is known that $\mathcal{L}_D$ is a nonnegative definite self-adjoint operator on $L^2(\Omega)$
(see, e.g., \cite[Theorem 1.3.1]{FuOsTa11}).
Moreover, for any $s\in (0,1]$,
one can define the factional power $\mathcal{L}_D^s$ of $\mathcal{L}_D$ by setting
\begin{align*}
\mathcal{L}_D^s:=(z^s)(\mathcal{L}_D)
\end{align*}
via the functional calculus on $L^2(\Omega)$ with domain
\begin{align}\label{eqn-dfp}
{\rm dom}_2(\mathcal{L}_D^s):=\left\{f\in L^2(\Omega):\
\mathcal{L}_D^s\in L^2(\Omega)\right\}
\end{align}
(see, e.g., \cite{Haa06}). In particular, for the square root $\mathcal{L}_D^{1/2}$, it holds that
\begin{align}\label{eqn-kr}
{\rm dom}_2(\,\mathcal{L}_D^{1/2})=\mathcal{W}_{d_D}^{1,2}(\Omega).
\end{align}

To extrapolate the above property from $p=2$ to general $p\in (1,\infty)$,
we need the Markovian semigroup
$\{P_t\}_{t>0}$ generated by $-\mathcal{L}_D$ to have
a positive heat kernel that has the local Gaussian upper bound,
which is established in the following lemma.

\begin{lemma}\label{lem-Gauss}
Let $(\mathfrak a,\mathcal{W}_{d_D}^{1,p}(\Omega))$
be the Dirichlet form in \eqref{eqn-DF}  and $\mathcal{L}_D$
the associated operator on $L^2(\Omega)$. Then $\mathcal{L}_D$ generates a positive semigroup
having a heat kernel $p_t(\cdot,\cdot)\in L^\infty(\Omega,\Omega)$ such that there exist constants $w_0\in \mathbb{R}$ and
$c, b\in (0,\infty)$ such that, for any $t\in (0,\infty)$ and
a.e. $x, y\in \Omega$, it holds that
\begin{align}\label{eqn-GE}
0\le p_t(x,y)\le ct^{-n/2}\exp\left\{w_0t-\frac{b|x-y|^2}{t}\right\}.
\end{align}
\end{lemma}

\begin{remark}\label{rem5.7}
\begin{itemize}
\item [(i)]
The Gaussian upper bound \eqref{eqn-GE} for the heat kernel offers significant advantages
in extending the functional properties of $\mathcal{L}_D$ from $L^2(\Omega)$ to $L^p(\Omega)$
for any $p\in (1,\infty)$.
For instance, it was proven in \cite[Theorem 5.9]{ArEl97} that, for any $p\in (1,\infty)$ and $w\in (w_0,\infty)$ with
$w_0$ as in \eqref{eqn-GE}, $\mathcal{L}_D+w$ has a bounded $H_\infty$
functional calculus on $L^p(\Omega)$.  Moreover,
$-\mathcal{L}_D$ generates a holomorphic semigroup on $L^p(\Omega)$
with holomorphic sector containing at least
$$\Sigma(\pi/2):=\{z\in \mathbb{C}:\ |{\rm arg}\, z|<\pi/2 \}.$$
This implies that   $\mathcal{L}_D$ is a sectorial operator with spectral angle $0$
(see \cite{DuRo96}). Thus, following \cite{Haa06} one can extend the definition
of the fractional power $\mathcal{L}_D^s$ from $L^2(\Omega)$ to $L^p(\Omega)$ for any
$p\in (1,\infty)$ with the domain
\begin{align*}
{\rm dom}_p(\mathcal{L}_D^s)=\left\{f\in L^p(\Omega):\
\mathcal{L}_D^s\in L^p(\Omega)\right\}.
\end{align*}
For $s=1/2$, let $(p_-,q_+)$ be the internal
of the maximal interval of all $p\in (1,\infty)$ such that
\begin{align*}
{\rm dom}_p(\mathcal{L}_D^{1/2})=\mathring{W}^{1,p}_D(\Omega).
\end{align*}
We refer to \cite{Bec24} for a systematic study of the characterization of $(p_-,q_+)$.

\item [(ii)] Based on the Gaussian upper bound \eqref{eqn-GE},
we infer from \cite[Theorem 2.2]{HiMo00} that, for any $p\in (1,\infty)$,
the space $L^p([0,T); L^p(\Omega))$ is a space of the maximal regularity for $\mathcal{L}_D$ as in \eqref{eqn-mr}.
Recall in \cite{HNVW16} that, for any $p\in (1,\infty)$ and any Banach space $X$, the \emph{Bochner space} $L^p([0,T); X)$
is defined by setting
\begin{align}\label{eqn-BS}
L^p([0,T); X):=\left\{u:  [0,T)\to X :\
\left[\int_0^T \|u(t,\cdot)\|^p_X\,dt\right]^{\frac 1p}<\infty\right\}.
\end{align}
\end{itemize}
\end{remark}

To prove Lemma \ref{lem-Gauss}, we need the following definition of the
Gaussian admissible space from \cite{ArEl97}.

\begin{definition}\label{def-ads}
Let $\Omega \subset \mathbb{R}^n$ be a domain. A linear space $\mathcal{V}\subset
\mathcal{W}^{1,2}(\Omega)$ is called a \emph{Gaussian admissible space} if $\mathcal{V}$ satisfies
the following
\begin{itemize}
  \item [(i)] $\mathcal{V}$ is a closed subspace of $\mathcal{W}^{1,2}(\Omega)$,
  \item [(ii)] $\mathring W^{1,2}(\Omega)\subset \mathcal{V}$,
  \item [(iii)] there exists a bounded linear extension $\mathcal{E}_{\mathcal{V}}$ from $\mathcal{V}$ to $W^{1,2}(\mathbb{R}^n)$ such that, for any
  $f\in \mathcal{V}\cap L^1(\Omega)$,
  $\|\mathcal{E}_{\mathcal{V}}f\|_{L^1(\mathbb{R}^n)}\lesssim \|f\|_{L^1(\Omega)}$,
  \item [(iv)] if $f\in \mathcal{V}$, then $|f|$, $|f|\wedge 1\in \mathcal{V}$, and
  \item [(v)] if $f\in \mathcal{V}$ and $g\in \mathcal{W}^{1,2}(\Omega)$ satisfy
  $|g|\le f$, then $g\in \mathcal{V}$.
\end{itemize}
\end{definition}

Arendt and Elst \cite[Theorem 4.4]{ArEl97} proved the following interesting result which reduces the proof of Gaussian upper bounds
of heat kernels to verifying that the conditions of Gaussian admissible spaces
as in Definition \ref{def-ads} hold.

\begin{lemma}[\cite{ArEl97}]\label{lem-aeGauss}
Let $\Omega\subset\mathbb{R}^n$ be a domain and  $\mathcal{V}\subset
\mathcal{W}^{1,2}(\Omega)$ be a {Gaussian admissible space}. Suppose
$(\mathfrak a,\mathcal{V})$ is a closed bilinear form  defined similar to \eqref{eqn-DF} and $\mathcal{L}_\mathcal{V}$
is the associated operator on $L^2(\Omega)$. Then $\mathcal{L}_\mathcal{V}$ generates a positive semigroup
having a heat kernel $p_t(\cdot,\cdot)\in L^\infty(\Omega,\Omega)$ that satisfies the Gaussian upper
bound \eqref{eqn-GE}.
\end{lemma}

We now turn to the proof of Lemma \ref{lem-Gauss}.

\begin{proof}[Proof of Lemma \ref{lem-Gauss}]
With the help of Lemma \ref{lem-aeGauss}, we only need to verify that $\mathcal{W}_{d_D}^{1,p}(\Omega)$
is a {Gaussian admissible space} as in Definition \ref{def-ads}.  By Remark \ref{rem1.9} and Lemma \ref{lem4.5},
we find that (i)-(iii) and (v) of Definition \ref{def-ads} hold.
Definition \ref{def-ads}(iv) follows from an argument similar to that used in the proof of Lemma
\ref{lem-DF}, the details being omitted. This finishes the proof of Lemma \ref{lem-Gauss}.
\end{proof}

Finally, we are in a position to prove Corollaries \ref{cor1.11} and \ref{cor1.12}.

\begin{proof}[Proof of Corollary \ref{cor1.11}]
Let ${\rm dom}_2(\mathcal{L}_D^{1/2})$ be the domain of the square root of
$\mathcal{L}_D$ as in \eqref{eqn-dfp}. From \eqref{eqn-kr}, we deduce
${\rm dom}_2(\,\mathcal{L}_D^{1/2})=\mathcal{W}_{d_D}^{1,2}(\Omega)$.
Moreover, using the complex interpolation of the domains of fractional powers
(see, e.g., \cite[Theorem 6.6.9]{Haa06}), we obtain,  for any $s\in (0,1)$,
\begin{align*}
\left[L^2(\Omega), {\rm dom}_2(\,\mathcal{L}_D^{1/2})\right]_s
={\rm dom}_2(\,\mathcal{L}_D^{s/2}),
\end{align*}
which, combined with Proposition \ref{t5} and the coincidence of
the real and complex interpolations in the $L^2$ scale (see, \cite[Corollary 4.37]{Lua}),
further implies that
\begin{align*}
{\rm dom}_2(\,\mathcal{L}_D^{s/2})=\left[L^2(\Omega), {\rm dom}_2(\,\mathcal{L}_D^{1/2})\right]_s
=\left(L^2(\Omega), \mathcal{W}_{d_D}^{1,2}(\Omega)\right)_{s,2}
=\mathcal{W}_{d_D^s}^{s,2}(\Omega).
\end{align*}
This finishes the proof of Corollary \ref{cor1.11}.
\end{proof}

To show Corollary \ref{cor1.12}, we need the following lemma from
\cite{DaGr75} (see also \cite[Theorem 9.3.6]{Haa06}).

\begin{lemma}[\cite{DaGr75}]\label{lem-DG}
Let $T\in (0,\infty)$, $p\in [1,\infty)$, $q \in [1, \infty]$,
$\theta \in  (0, 1)$, and ${\rm Re}\,\alpha  \in (0,\infty)$. Suppose that $A$ is a sectorial operator on the
Banach space $X$ with spectral angle $< \pi/2$. Then the space
$L^p([0, T); (X, {\rm dom}_X(A^\alpha))_{\theta,q})$ is a space of the maximal regularity for $A$.
\end{lemma}

\begin{proof}[Proof of Corollary \ref{cor1.12}]
Corollary \ref{cor1.12} is an immediately consequence of
Lemma \ref{lem-DG} in view of both Remark \ref{rem5.7}(i)
and Proposition \ref{t5}. We omit the details.
\end{proof}

\noindent\textbf{Acknowledgements}\quad The authors would like to
express their sincere thanks to Dr. Sebastian Bechtel
for his many helpful comments which definitely improve the quality 
of this article.

\medskip

\noindent\textbf{Data Availability}\quad Data sharing is not applicable
to this article as obviously no datasets were generated or
analyzed during the current study.

\section*{Declarations}

\noindent\textbf{Conflict of interest}\quad The authors have no
conflict of interest to declare that are relevant to the content of this
article.

\bigskip

\noindent Jun Cao

\medskip

\noindent School of Mathematical Sciences, Zhejiang University of Technology, Hangzhou 310023, The People's Republic of China

\smallskip

\noindent{\it E-mail address}: \texttt{caojun1860@zjut.edu.cn}

\bigskip

\noindent Dachun Yang (Corresponding author)

\medskip

\noindent Laboratory of Mathematics and Complex Systems (Ministry of Education of China),
School of Mathematical Sciences, Beijing Normal University, Beijing 100875, The People's Republic of China

\smallskip

\noindent{\it E-mail address}: \texttt{dcyang@bnu.edu.cn}

\bigskip

\noindent  Qishun Zhang

\medskip

\noindent School of Mathematics,
Renmin University of China, Beijing 100872, The People's Republic of China

\smallskip

\noindent{\it E-mail address}: \texttt{qszhang@ruc.edu.cn}


\end{document}